\documentclass[reqno]{amsart}
\usepackage[top=1.1in, bottom=1in, left=1in, right=1in]{geometry}
\usepackage{amsfonts}
\usepackage{amssymb}
\usepackage{amsthm}
\usepackage{amsmath}
\usepackage{mathrsfs}
\usepackage{color}
\usepackage{enumerate}
\usepackage[numbers,sort&compress]{natbib}
\usepackage{pdfsync}
\usepackage{esint}
\usepackage{graphicx}
\usepackage{float}
\usepackage{caption}
\usepackage{subfigure}
\usepackage[titletoc]{appendix}

\allowdisplaybreaks

\usepackage{tikz} 
\usetikzlibrary{calc}
\usetikzlibrary{intersections}
\usetikzlibrary{patterns}

\usepackage[colorlinks,
            linkcolor=black,       
            anchorcolor=blue,      
            citecolor=blue,        
            ]{hyperref}

\makeatother

\numberwithin{equation}{section}

\newtheorem{theorem}{Theorem}[section]
\newtheorem{definition}[theorem]{Definition}

\newtheorem{lemma}[theorem]{Lemma}
\newtheorem{remark}[theorem]{Remark}

\newtheorem{proposition}[theorem]{Proposition}
\newtheorem{corollary}[theorem]{Corollary}
\numberwithin{equation}{section}

\newcommand*{\Id}{\ensuremath{\mathrm{Id}}}
\newcommand*{\supp}{\ensuremath{\mathrm{supp\,}}}

\newcommand{\aint}{{\fint}}

\newcommand{\T}{{\mathbb{T}}}

\newcommand{\rr}{\mathring{R}}
\newcommand{\ru}{\mathring{R}_q^u}
\newcommand{\rb}{\mathring{R}_q^B}
\newcommand{\p}{\partial}
\renewcommand{\P}{\mathbb{P}}

\renewcommand{\div}{{\mathrm{div}}}
\newcommand{\curl}{{\mathrm{curl}}}

\renewcommand{\u}{{u_q}}
\newcommand{\h}{{B_q}}

\renewcommand{\d}{{\rm d}}

\newcommand{\g}{g_{(\tau)}}

\newcommand{\norm}[1]{\lVert#1\rVert}
\newcommand{\abs}[1]{|#1|}

\newcommand{\la}{\lambda_q}
\newcommand{\laq}{\lambda_{q+1}}

\newcommand{\rs}{r_{\perp}}
\newcommand{\rp}{r_{\parallel}}
\newcommand{\va}{\varepsilon}

\newcommand{\wo}{w_{q+1}^{(o)}}
\newcommand{\dqo}{d_{q+1}^{(o)}}
\newcommand{\wdc}{\wt D_{(k)}^c}

\newcommand{\R}{{\mathbb R}}
\newcommand{\lbb}{\overline{\lambda}}

\newcommand{\thq}{\theta_{q+1}}
\newcommand{\mq}{m_{q+1}}

\newcommand{\tr}{\mathring{\wt R^u_q}}
\newcommand{\trb}{\mathring{\wt R^B_q}}

\def\a{{\alpha}}
\def\vf{{\varphi}}
\def\lbb{\lambda}
\def\wt{\widetilde}
\def\9{{\infty}}
\def\ve{{\varepsilon}}
\def\na{{\nabla}}
\def\bbr{{\mathbb{R}}}

\def\({\left(}
\def\){\right)}

\newcommand{\cH}{{\mathcal{H}}}
\newcommand{\cL}{{\mathcal{L}}}
\newcommand{\cW}{{\mathcal{W}}}

\begin{document}
	
\title[] {Sharp non-uniqueness of weak solutions to 3D magnetohydrodynamic equations}

\author{Yachun Li}
\address{School of Mathematical Sciences, CMA-Shanghai, MOE-LSC, and SHL-MAC,  Shanghai Jiao Tong University, China.}
\email[Yachun Li]{ycli@sjtu.edu.cn}
\thanks{}

\author{Zirong Zeng}
\address{School of Mathematical Sciences, Shanghai Jiao Tong University, China.}
\email[Zirong Zeng]{beckzzr@sjtu.edu.cn}
\thanks{}

\author{Deng Zhang}
\address{School of Mathematical Sciences, CMA-Shanghai, Shanghai Jiao Tong University, China.}
\email[Deng Zhang]{dzhang@sjtu.edu.cn}
\thanks{}



\begin{abstract}
   We prove the non-uniqueness of weak solutions to 3D hyper
   viscous and resistive MHD in the class $L^\gamma_tW^{s,p}_x$,
   where the exponents $(s,\gamma,p)$ lie in two supercritical
   regimes.
   The result reveals that
   the scaling-invariant Lady\v{z}enskaja-Prodi-Serrin (LPS) condition
   is the right criterion to detect non-uniqueness,
   even in the highly viscous and resistive
   regime beyond the Lions exponent.
   In particular,
   for the classical viscous and resistive MHD,
   the non-uniqueness is sharp near the endpoint
   $(0,2,\infty)$ of the LPS condition.
   Moreover, the constructed weak solutions admit the partial regularity
   outside a small fractal singular
   set in time with zero $\mathcal{H}^{\eta_*}$-Hausdorff dimension,
   where $\eta_*$ can be any given small positive constant.
   Furthermore,
   we prove the strong vanishing viscosity and resistivity result,
   which yields the failure of Taylor's conjecture
   along some subsequence
   of weak solutions to the hyper viscous and resistive MHD
   beyond the Lions exponent.
\end{abstract}

\maketitle

{ \small
	\tableofcontents
}

\section{Introduction and main results}

\subsection{Introduction} \label{Subsec-intro}

We are concerned with the
three-dimensional magnetohydrodynamic (MHD for short) system on the torus $\T^3:=[-\pi,\pi]^3$,
\begin{equation}\label{equa-MHD}
	\left\{\aligned
	&\p_tu + \nu_1\Delta u+(u\cdot \nabla )u -(B\cdot \nabla )B + \nabla P=0,  \\
	&\p_tB + \nu_2\Delta B+(u\cdot \nabla )B- (B\cdot \nabla )u  =0, \\
	& \div u = 0, \quad \div B = 0,
	\endaligned
	\right.
\end{equation}
where  $u=(u_1,u_2,u_3)^\top(t,x)\in \R^3 $, $B=(B_1,B_2,B_3)^\top(t,x)\in \R^3$
and $P=P(t,x)\in \R$
correspond to the  velocity field,
magnetic field and pressure of the fluid, respectively,
and $\nu_1, \nu_2\geq 0$ are the viscous and resistive coefficients, respectively.
In particular,
in the case without magnetic fields,
equations \eqref{equa-MHD}
reduce to the classical incompressible Navier-Stokes equation
(NSE for short).

In the celebrated papers \cite{leray1934},
Leray proved the existence of global solutions
in $L^\infty_tL^2_x \cap L^2_t\dot{H}^1_x$,
which now refers to Leray-Hopf solutions,
also due to the important contribution of Hopf \cite{hopf1951}
in the case of bounded domains.
Leray-Hopf solutions enjoy several nice properties,
including the partial regularity and weak-strong uniqueness.
This kind of solutions also have been obtained for MHD
by Sermange-Teman \cite{ST83}.
We refer to the monographes \cite{b1993,b2003,d2001,LQ14,Tsai18}
for more detailed surveys on NSE and MHD equations.
However, the uniqueness of Leray-Hopf solutions still remains
one of the most challenging problems.

In the last decade, major progresses have been achieved for the non-uniqueness
for various fluid models.
In the breakthrough work \cite{dls09}, De Lellis and Sz\'{e}kelyhidi, Jr. introduced the convex integration method
to construct wild solutions to 3D Euler equations.
Afterwards,
one milestone is the resolution of the flexible part of
$C^{1/3}_{t,x}$ regularity threshold
in Onsager's conjecture
for Euler equations, developed in a series of works
\cite{B15,dls10,dls13,bdis15,bdls16,dls14},
and finally solved by Isett \cite{I18} and Buckmaster-De Lellis-Sz\'{e}kelyhidi, Jr.-Vicol \cite{bdsv19} for dissipative solutions. See also \cite{cvy21,lxx16} and references therein 
for applications to compressible Euler systems.

Another significant progress is that,
Buckmaster-Vicol \cite{bv19b} proved that
weak/distributional solutions to 3D NSE
are not unique in the space $C_tL^2_{x}$.
It was proved there even more that,
there are indeed infinitely many weak solutions with the same initial data
and with prescribed energy profiles.
The intermittent convex integration, developed in \cite{bv19b},
has been further applied to many other models,
including, for instance,
3D hyperdissipative NSE up to the Lions exponent \cite{lt20,bcv21},
 hypodissipative NSE \cite{CDM20,DR19,lq20},
the stationary NSE \cite{luo19}, SQG equation \cite{bsv19}
and transport equations \cite{cl21,cl22,ms18}.
Remarkably,
new intermittent convex integration has been  recently developed
in  \cite{bmnv21,nv22} for Euler equations,
which, particularly,
enables to construct non-conservative weak solutions
to 3D Euler equations in the class $C_t(H^\beta \cap L^{1/(1-2\beta)})$
with $\beta<1/2$,
hence by interpolation also in the class $C_tB^s_{3,\infty}$ for $s<1/3$. 
We refer to \cite{bv19r,bv21,CD12,dls17} for comprehensive surveys on convex integration
method. We also refer to \cite{js14,js15} for non-uniqueness of Leray-Hopf solutions based on 
spectral assumptions and the recent work \cite{ABC21} for non-unique Leray-Hopf solutions to forced 3D NSE.

In general, the scaling exponent is a very helpful criterion
for the well-posedness/ill-posedness.
We refer to the papers  by Klainerman \cite{K00,K17}
for comprehensive discussions of scaling exponents
for general nonlinear PDEs.

For the NSE, one important critical scaling in the
mixed Lebsgue space $L^\gamma_tL^p_x$ is the
Lady\v{z}enskaja-Prodi-Serrin (LPS for short) condition
\begin{align} \label{LPS-condition}
   \frac{2}{\gamma} + \frac{d}{p} = 1
\end{align}
with $d$ being the underlying spatial dimension,
which is critical under the scaling
\begin{align}
   u(t,x) \mapsto \lbb u (\lbb^2 t, \lbb x), \ \
   P(t,x) \mapsto \lbb^2 P (\lbb^2 t, \lbb x).
\end{align}
It is well-known that,
weak solutions in the (sub)critical spaces $L^\gamma_tL^p_x$
with $2/\gamma + 3/p\leq 1$ are unique and even Leray-Hopf.
See \cite{FJR72, iss03, cl20.2} and references therein.

On the flexible side,
in the recent remarkable paper \cite{cl20.2},
Cheskidov-Luo achieved the non-uniqueness
of weak solutions to NSE
in the spaces $L^\gamma_tL^\infty$ for any $1\leq \gamma<2$.
The non-uniqueness is sharp as $L^2_tL^\infty_x$ is the critical endpoint space
in view of the LPS condition.
Moreover, in another endpoint case
where $(\gamma, p)= (\infty, 2)$ when $d=2$,
the sharp non-uniqueness in $L^\infty_tL^p_x$,
$1\leq p<2$, has been proved for 2D NSE in \cite{cl21.2}.
Very recently,
for the 3D hyperdissipative NSE
with viscosity beyond the Lions exponent $5/4$,
the sharp non-uniqueness at two endpoints of generalized LPS conditions
has been proved in \cite{lqzz22}.

Regarding MHD equations,
the corresponding non-uniqueness and related turbulence
have attracted significant interests in literature.
In particular,
several progresses have been made for the ideal MHD,
which corresponds to \eqref{equa-MHD} when
the viscous and resistive coefficients vanish,
that is,
\begin{equation}\label{equa-MHD-ideal}
	\left\{\aligned
	&\p_tu +(u\cdot \nabla )u -(B\cdot \nabla )B + \nabla P=0,  \\
	&\p_tB +(u\cdot \nabla )B- (B\cdot \nabla )u  =0, \\
	& \div u = 0, \quad \div B = 0.
	\endaligned
	\right.
\end{equation}
The lack of uniqueness of regular weak solutions
is in agreement with the natural stochasticity
observed in turbulent regimes
(\cite{FL22}).

A first attempt at constructing non-vanishing smooth strict subsolutions to
3D ideal MHD is the work by Faraco-Lindberg \cite{FL18}.
Afterwards, bounded weak solutions to \eqref{equa-MHD-ideal} were
constructed by  Faraco-Lindberg-Sz\'{e}kelyhidi \cite{fls21},
which violate the  energy conservation
while preserve the magnetic helicity
(see \S \ref{Subsec-Main-Result} below).
Bounded solutions
with prescribed total energy and cross helicity
also have been constructed by
Faraco-Lindberg-Sz\'{e}kelyhidi
\cite{fls21,fls21.2}.
For weak solutions with non-conserved magnetic helicity,
the first example was constructed by Beekie-Buckmaster-Vicol
\cite{bbv20},
based on the intermittent convex integration scheme.
We also would like to mention the recent work by Faraco-Lindberg-Sz\'ekelyhidi \cite{fls21.2},
where
the conjecture that $L^3_{t,x}$ is the threshold
for magnetic helicity conservation is solved,
based on the convex integration via staircase laminates.

In contrast, to the best of our knowledge,
there seem not many works on the non-uniqueness problem
for the viscous and resistive MHD.
One major difficulty here lies in the strong viscosity
and resistivity
but the limitation of spatial intermittency
restricted by the specific geometry of MHD
(see \S \ref{Subsec-Main-Result} below for detailed discussions).
Dai \cite{dai18} constructed non-unique weak solutions to Hall MHD equations,
where the Hall nonlinearity takes the dominant effect.
Moreover, Beekie-Buckmaster-Vicol \cite{bbv20} constructed
the intermittent shear flows with 1D intermittency,
which actually permit to control hypo viscosity and resistivity
$(-\Delta)^\alpha$
for any $\alpha\in (0,1/2)$.
Recently,
the authors \cite{lzz21} proved the non-uniqueness for MHD type equations
with the viscosity and resistivity up to the Lions exponent,
by constructing intermittent flows
featuring both 2D spatial intermittency
and 1D temporal intermittency.

\subsection{Main results}   \label{Subsec-Main-Result}

The main interest of this paper is to further study the non-uniqueness problem
for MHD in the following two aspects:
\begin{enumerate}
  \item[$\bullet$] Sharp non-uniqueness of weak solutions to MHD \eqref{equa-MHD}
  in the spaces $L^\gamma_t L^\infty_x$, $1\leq \gamma<2$,
  which approach the endpoint space $L^2_tL^\infty_x$
  with respect to the critical LPS condition.

  \item[$\bullet$] Non-uniqueness of weak solutions to
  MHD type equations (see \eqref{equa-gMHD} below)
  in the highly viscous and resistive regime,
  particularly, beyond the Lions exponent.
\end{enumerate}

We would consider a general formulation of MHD equations with
hyper viscosity and resistivity:
\begin{equation}\label{equa-gMHD}
	\left\{\aligned
	&\p_tu+ \nu_1 (-\Delta)^{\alpha } u+(u\cdot \nabla )u -(B\cdot \nabla )B + \nabla P=0,  \\
	&\p_tB+ \nu_2 (-\Delta)^{\alpha } B+(u\cdot \nabla )B- (B\cdot \nabla )u  =0, \\
	& \div u = 0, \quad \div B = 0,
	\endaligned
	\right.
\end{equation}
where $\nu_1, \nu_2\geq 0$
and $\alpha \in [1, 3/2)$.
Note that, the viscosity and resistivity exponents can be larger than the Lions exponent $5/4$.

It should be mentioned that,
one major obstruction at constructing non-unique weak solutions
is that, the specific geometry of MHD restricts the choice of oscillations directions,
and so, in particular, limits the spatial intermittency of building blocks in
the convex integration scheme.
Thus, the control of viscosity and resistivity in \eqref{equa-MHD}
becomes significantly hard.

On the other hand,
high dissipations usually help to establish well-posedness of PDEs.
One well-known result due to Lions \cite{lions69}
is that, for any divergence-free $L^2$ initial data,
there exist unique Leray-Hopf solutions to 3D NSE
when the viscosity exponent  $\alpha \geq 5/4$.
This kind of well-posedness also holds for
hyper viscous and resistive MHD equations,
due to Wu \cite{wu03}.
Hence,
it is non-trivial to find
non-unique weak solutions particularly
in the very high viscous and resistive regime.
New appropriate criterion,
rather than dissipation,
is required to detect non-unique weak solutions here.

We will show that,
the scaling-invariant LPS condition is the right
criterion to detect non-uniqueness:
the uniqueness still would fail
even in the high viscous and resistive regime
beyond the Lions exponent,
if the (sub)criticality of state space is violated.

As a matter of fact,
the hyper viscous and resistive MHD \eqref{equa-gMHD}
is invariant under the scaling
\begin{align}  \label{scal-gMHD}
   u(t,x) \mapsto \lbb^{2\alpha-1} u(\lbb^{2\alpha} t, \lbb x),\ \
   B(t,x) \mapsto \lbb^{2\alpha-1} B(\lbb^{2\alpha} t, \lbb x),
\end{align}
and $P(t,x ) \mapsto \lbb^{4\alpha-2} P(\lbb^{2\alpha} t, \lbb x)$.
The critical exponent $(s,\gamma, p)$
of the mixed Lebesgue spaces  $L^\gamma_tW^{s,p}_x$
satisfies the generalized
Lady\v{z}enskaja-Prodi-Serrin condition  
\begin{align}\label{critical-gLPS-mhd}
      \frac{2\alpha}{\gamma} + \frac{3}{p} = 2\alpha-1+s.
\end{align}
We show that the non-uniqueness of weak solutions
do exhibit in the spaces $L^\gamma_t W^{s,p}_x$,
where the exponents $(s,\gamma,p)$
lie in the following two supercritical regimes,
with respect to the scaling \eqref{scal-gMHD}:
\begin{align} \label{Super-S1}
		\mathcal{S}_1:=\bigg\{ (s,\gamma,p)\in [0,3)\times [1, \infty]\times[1,\infty]: 0\leq s<  \frac{2\a}{\gamma}+\frac{2\a-2}{p}+1-2\a \bigg\},
\end{align}
and
\begin{align} \label{Super-S2}
		\mathcal{S}_2:=\bigg\{ (s,\gamma,p)\in [0,3)\times [1, \infty] \times[1,\infty]: 0\leq s< \frac{4\a-4}{\gamma}+ \frac{2}{p}+1-2\a \bigg\}.
\end{align}

Before stating the main results,
let us mention that the weak solutions to \eqref{equa-gMHD}
is taken in the standard distributional sense.

\begin{definition}[Weak solutions]   \label{weaksolu}
	We say that $(u, B) \in L^2([0,T]; L^2(\T^3))$ is a weak solution to
the  MHD equations \eqref{equa-gMHD} if
	\begin{itemize}
		\item For all $t\in [0,T]$, $(u(t,\cdot), B(t,\cdot))$ are divergence free in the sense of distributions and have zero spatial mean.
		\item Equations \eqref{equa-gMHD} hold in the sense of distributions, i.e.,
		for any divergence-free test functions $\varphi  \in C_0^{\infty} ([0,T] \times \mathbb{T}^3)$,
		\begin{align*}
		\int\limits_{\mathbb{T}^3} u_0 \vf(0,x) \d x+\int_0^T \int_{\mathbb{T}^3}  \partial_t \varphi \cdot u - \nu_1(-\Delta)^{\alpha }\varphi \cdot u  + \nabla\varphi :( u \otimes u - B \otimes B) \d x \d t &= 0,
			\\
		\int\limits_{\mathbb{T}^3} B_0 \vf(0,x) \d x+\int_0^T	\int_{\mathbb{T}^3}  \partial_t \varphi \cdot B - \nu_2 (-\Delta)^{\alpha }\varphi \cdot B + \nabla\varphi  :(B \otimes u- u \otimes B) \d x \d t  &= 0.
		\end{align*}
	\end{itemize}
\end{definition}

The main result of this paper is formulated in Theorem \ref{Thm-Non-gMHD} below.

\begin{theorem} \label{Thm-Non-gMHD}
Let $\alpha\in [1,3/2)$ and $(\tilde{u},\tilde{B})$ be any smooth,
divergence-free and mean-free vector fields on $[0,T] \times \T^3$.
Then, there exists $\beta'\in(0,1)$,
such that for any $\va_*, \eta_*>0$
and for any $(s,p,\gamma)\in \mathcal{S}_1 \cup \mathcal{S}_2$,
there exist vector fields $u,B$ and a set
\begin{align*}
	\mathcal{G} = \bigcup\limits_{i=1}^\infty (a_i,b_i)  \subseteq [0,T],
\end{align*}
such that the following holds:
\begin{enumerate}[(i)]
	\item Weak solution: $(u,B)$ is a weak solution
		to \eqref{equa-gMHD} with the initial datum $(\wt u(0),\wt B(0))$
		and has zero spatial mean.
	\item Regularity: $u,B \in H^{\beta'}_{t,x} \cap L^\gamma_tW^{s,p}_x$,
	and
	\begin{align*}
		(u,B)|_{\mathcal{G}\times \mathbb{T}^3} \in C^\infty (\mathcal{G}\times \mathbb{T}^3)\times C^\infty (\mathcal{G}\times \mathbb{T}^3).
	\end{align*}
	Moreover,
	if there exists $t_0\in (0,T)$ such that
	$(\wt u,\wt B)$ is the solution to \eqref{equa-gMHD} on $[0,t_0]$,
	then $(u,B)$ agrees with $(\wt u,\wt B)$ near $t=0$.
	\item The Hausdorff dimension of the singular set
	$\mathcal{B} = [0,T]/\mathcal{G}$ satisfies
	\begin{align*}
		   d_{\mathcal{H}}(\mathcal{B}) <\eta_*.
	\end{align*}
	In particular, the singular set $\mathcal{B}$
	has zero Hausdorff $\mathcal{H}^{\eta_*}$ measure,
	i.e.,  $\mathcal{H}^{\eta_*}(\mathcal{B})=0$.
	\item Small deviations of temporal support:
	 $$\supp_t (u,B)  \subseteq N_{\va_*}(\supp_t (\tilde{u},\tilde{B})).$$
	\item Small deviations on average:
	$$\|u-\tilde{u}\|_{L^1_tL^2_x}+\|u-\tilde{u}\|_{L^\gamma_tW^{s,p}_x}\leq \va_*,\quad \|B-\tilde{B}\|_{L^1_tL^2_x}+\|B-\tilde{B}\|_{L^\gamma_tW^{s,p}_x}\leq \va_*.$$
\end{enumerate}
\end{theorem}

We note that,
the borderline of $\mathcal{S}_1$ regime in particular
contains the endpoint $(s,\gamma, p) = (0,2,\infty)$ of the
LPS condition \eqref{critical-gLPS-mhd} when $\alpha =1$.
Hence,
by virtue of Theorem \ref{Thm-Non-gMHD},
we obtain the non-uniqueness for the classical viscous and resistive MHD \eqref{equa-MHD},
which is sharp in view of the uniqueness result
in the critical endpoint space $L^2_tL^\infty_x$
in Theorem \ref{Thm-uniq-glps} of Appendix.

\begin{corollary} [Sharp non-uniqueness for MHD]  \label{Cor-Non-MHD}
Consider the viscous and resistive MHD \eqref{equa-MHD}.
Then, for any $1\leq \gamma <2$,
there exist infinitely many weak solutions in $L^\gamma_tL^\infty_x$
with the same initial data.
\end{corollary}

Our last result is concerned with the strong
vanishing viscosity and resistivity limits,
which relates both the ideal MHD \eqref{equa-MHD-ideal}
and hyper viscous, resistive MHD equations \eqref{equa-gMHD}.
Its interesting relationship to the Taylor conjecture
will be discussed in the comments below.

\begin{theorem} [Strong vanishing viscosity and resistivity limit]  \label{Thm-vanish-viscos}
	Let $\alpha\in [1,3/2)$ and $\wt \beta>0$.
	Let $(u,B)\in H^{\wt{\beta}}_{t,x} \times  H^{\wt{\beta}}_{t,x}$ be any mean-free weak solution
	to the ideal MHD \eqref{equa-MHD-ideal}.
	Then, there exist $\beta' \in (0, \wt \beta)$ and a sequence of weak solutions
	$(u^{(\nu_{n})}, B^{(\nu_{n})}) \in H^{\beta'}_{t,x} \times H^{\beta'}_{t,x}$
	to the hyper viscous and resistive MHD \eqref{equa-gMHD},
    where  $\nu_n=(\nu_{1,n},\nu_{2,n})$
    and $\nu_{1,n}$, $\nu_{2,n}$
    are the viscosity and resistivity coefficients,  respectively,
	such that
	as $\nu_{n}\rightarrow 0$,
	\begin{align}\label{convergence}
		(u^{(\nu_{n})}, B^{(\nu_{n})}) \rightarrow (u,B) \quad\text{strongly in}\  H^{\beta'}_{t,x} \times  H^{\beta'}_{t,x}.
	\end{align}
\end{theorem}

\paragraph{\bf Comments on main results.}

We present some comments on the main results below.

{\bf $(i)$ Sharp non-uniqueness of MHD near the endpoint space $L^2_tL^\infty_x$.}
In the recent remarkable paper \cite{cl20.2},
Cheskidov-Luo proved the non-uniqueness of weak solutions to NSE in
the spaces $L^\gamma_tL^\infty_x$ for any $1\leq \gamma<2$ and any dimensions.
This result is sharp because weak solutions are unique
in the endpoint space $L^2_tL^\infty_x$,
which is critical with respect to the LPS condition \eqref{LPS-condition}.
Sharp non-uniqueness for the 2D NSE near another endpoint space was proved by Cheskidov-Luo \cite{cl21.2}.
Recently, for the 3D hyper-dissipative NSE,
sharp non-uniqueness near both endpoints of the generalized Lady\v{z}enskaja-Prodi-Serrin condition was proved in \cite{lqzz22}.

To the best of our knowledge,
Corollary \ref{Cor-Non-MHD} provides the first non-uniqueness result
for the viscous and resistive MHD near the endpoint space $L^2_t L^\infty_x$.
In view of Theorem~\ref{Thm-uniq-glps}, i.e.,
the uniqueness of weak solutions to MHD in $L^2_tL^\infty_x$,
the non-uniqueness result in Corollary \ref{Cor-Non-MHD} is sharp.

We would expect the non-uniqueness also exhibit in the remaining supercritical regimes.
Actually, the second set $\mathcal{S}_2$ includes a part of
the supercritical regime.
One difficulty of the current MHD situation lies in the
specific geometry of MHD,
which restricts the choice of oscillating directions,
and makes it very hard to construct building blocks with 3D spatial intermittency
as in the NSE context.
Another difficulty is due to the $L^2$-critical nature of
the current convex integration scheme,
as observed by Cheskidov-Luo \cite{cl20.2}.
In view of the LPS condition \eqref{critical-gLPS-mhd},
there exists some regime of exponents $(\gamma, p)$
both larger than two, which seem out of reach of
the current method.

{\bf $(ii)$ Non-uniqueness in the highly viscous and resistive regime: beyond the Lions exponent.}
As mentioned above, the high dissipation usually helps to establish well-posedness of PDEs,
and so is the obstacle for the ill-posedness.
For instance, in the case of 3D NSE,
the viscosity $(-\Delta)$ was the major obstruction at
constructing non-unique weak solutions.
The crucial novelty in \cite{bv19b} to overcome this difficulty
is the construction of $L^2$-based building blocks
with 3D spatial intermittency.
However, due to the MHD geometry,
it is still open to construct 3D spatial intermittent building blocks.

Another thing that deserves attention is that,
it is indeed impossible to construct non-unique
weak solutions in $C_tL^2_x$ in the very high viscous and
resistive regime beyond the Lions exponent,
due to the well-posedness results by Lions \cite{lions69}
in the NSE case and by Wu \cite{wu03} in the MHD case.

Theorem \ref{Thm-Non-gMHD} shows that,
even in the high dissipative regime beyond the Lions exponent,
the LPS condition serves as the compass to
find non-unique weak solutions,
which do exist in the supercritical
mixed Lebesgue spaces.

{\bf $(iii)$ Partial regularity of weak solutions.}
To our best knowledge,
Theorem \ref{Thm-Non-gMHD} provides the first example for weak solutions
to MHD,
whose singular sets in time have small Hausdorff dimensions.
The interesting phenomenon here is that,
though weak solutions are not unique in the supercritical space
$L^\gamma_tW^{s,p}_x$,
they seem not so ``bad":
they can be very close to Leray-Hopf solutions
and are smooth in a majority of time,
of which the complement has
zero Hausdorff $\mathcal{H}^{\eta_*}$ measure
with $\eta_*$ being any given small positive constant.

Let us mention that,
the delicate partial regularity of non-unique weak solutions was first
discovered by Buckmaster-Colombo-Vicol \cite{bcv21} in $C_tL^2_x$ for
the hyperdissipative NSE, up to the Lions exponent $5/4$,
whose singular set in time has
Hausdorff dimension strictly less than one.
Weak solutions with singular temporal sets of small Hausdorff dimension
also have been constructed in $L^\gamma_tL^p_x$
for NSE by Cheskidov-Luo \cite{cl20.2},
and in $L^\gamma_tW^{s,p}$ for hyperdissipative NSE with viscosity beyond the Lions exponent $5/4$
\cite{lqzz22}.
The crucial ingredient here is the
use of gluing technique,
developed in a series of works
\cite{B15,bdis15,bdls16,dls10,dls13,bdsv19,dls14,I18}
in the resolution of Onsager's conjecture,
combined with suitable regularity and stability
estimates for MHD equations.

{\bf $(iv)$ Strong vanishing viscosity and resistivity limit and Taylor's conjecture.}
It would be interesting to see the relationship between
the vanishing viscosity and resistivity limit in Theorem \ref{Thm-vanish-viscos}
and the Taylor conjecture.

The ideal MHD \eqref{equa-MHD-ideal} has several global invariants:
\begin{enumerate}
	\item[$\bullet$] The total  energy:\ \
		$\displaystyle \mathcal{E}(t) =\frac12\int_{\T^3}|u(t,x)|^2+|B(t,x)|^2 \d x$;
	\item [$\bullet$] The cross helicity: \ \
		$\displaystyle \mathcal{H}_{\omega,B}(t) =\int_{\T^3}u(t,x)\cdot B(t,x) \d x$;
	\item [$\bullet$] The magnetic helicity: \ \
		$\displaystyle \mathcal{H}_{B,B}(t):=\int_{\T^3}A(t,x)\cdot B(t,x) \d x$.
\end{enumerate}	
Here $A$ is  a mean-free periodic vector field satisfying $\curl A=B$.
The famed {\it Taylor conjecture}, commonly accepted in the plasma physics literature,
is that the magnetic helicity is expected to be conserved in the infinite conductivity limit \cite{taylor74,taylor86}.
This conjecture is valid under weak ideal limits,
namely, the weak limits of Leray-Hopf solutions to MHD \eqref{equa-MHD}.
It has been proved by Faraco-Lindberg \cite{fl20}
in simply connected, magnetically closed domain,
and by Faraco-Lindberg-MacTaggart-Valli \cite{FLMV22}
in multiply connected domains.
On the other hand,
Beekie-Buckmaster-Vicol \cite{bbv20}
constructed weak solutions with non-conserved magnetic helicity,
which shows that the ideal MHD version of Taylor's conjecture is false.
The delicate point here is that,
the conservation of magnetic helicity requires much milder regularity
than that of total energy and cross helicity.
Actually, the energy and cross helicity are conserved
by weak solutions in $B^{\alpha}_{3,\infty}$ with $\alpha>1/3$
or $B^{1/3}_{3,c(\mathbb{N})}$,
while the magnetic helicity is conserved
in the less regular space
$B^\alpha_{3,\infty}$ with $\alpha>0$
or in the endpoint space $L^3_{t,x}$.
See \cite{CKS97,KL07,A09,FL18}.

Theorem \ref{Thm-vanish-viscos} provides another viewpoint that,
in contrast to weak ideal limits,
there indeed exists certain sequence of weak
(non-Leray-Hopf) solutions
even for the hyper viscous and resistive MHD beyond Lions' exponent,
such that
the Taylor conjecture fail in the limit along this sequence.
\\

\paragraph{\bf Distinctions between NSE and MHD
in intermittent convex integration schemes.}
As we have already seen above,
one major distinction between NSE and MHD is that,
the strong coupling between the velocity and magnetic fields in MHD
restricts the choice of oscillating directions.
This makes it quite difficult to construct spatial building blocks
with strong spatial intermittency adapted to the specific
geometric structure of MHD.

Actually, in the context of NSE,
spatial building blocks
with 3D intermittency
(known as intermittent Beltrami flows in \cite{bv19b},
and intermittent jets in \cite{bcv21}) can be constructed.
Such strong intermittent building blocks
enable to construct non-unique weak solutions
in $C_tL^2_x$ for the 3D NSE \cite{bv19b},
and even for the hyperdissipative NSE up to the Lions exponent \cite{lt20,bcv21}.
Moreover, the intermittent jets also enable us to
achieve the sharp non-uniqueness at the endpoint $(\gamma, p)=(\infty, \frac{3}{2\alpha-1})$
of the LPS condition,
for the hyperdissipative NSE beyond the Lions exponent (\cite{lqzz22}).

In contrast, due to the geometry of MHD equations,
the shear flows constructed in \cite{bbv20} have 1D spatial intermittency,
and the intermittent flows constructed in \cite{lzz21} have 2D spatial intermittency,
both are less than 3D spatial intermittency
in the context of NSE.
Hence, unlike the NSE case \cite{lqzz22},
these spatial intermittencies are not able to control
the viscosity and resistivity $(-\Delta)$ of MHD.
The construction of non-unique weak solutions to MHD
in the endpoint space $C_tL^{\frac{3}{2\alpha-1}}_x$
is still open.

Another delicate fact is that,
the intermittent jets for NSE
have disjoint supports,
while the intermittent flows in \cite{bbv20,lzz21} for MHD
may interact with each other.
This in particular gives rise to the interaction errors
in the Reynolds and magnetic stresses.
See \S \ref{Sec-perturb-S1} and \S \ref{Sec-Rey-mag-stress-A1} below.
Stronger intermittency needs to be gained
to control these interaction errors.
The keypoint here is to exploit
the geometric nature of
the interactions of intermittent flows.
As a matter of fact,
the interaction of intermittent shear flows in \cite{bbv20}
is given by thickened lines which give 2D intermittency,
and the interaction of intermittent flows in \cite{lzz21}
concentrates on small cuboids
which have 2D intermittency
(see \cite[Figure 3]{lzz21}).

Furthermore, the distinction between two equations
also lies in the construction of the
amplitudes of velocity and magnetic perturbations.
Because of the anti-symmetry of magnetic nonlinearity,
the construction of magnetic amplitudes requires
a second geometric lemma
in a small neighborhood of the null matrix
(i.e., Lemma \ref{geometric lem 2} below),
which is different from the geometric lemma
near the identity matrix (i.e., Lemma \ref{geometric lem 1})
for velocity amplitudes.
Moreover, when constructing velocity amplitudes,
a new matrix $\mathring{G}^B$ (see \eqref{def:G} below)
also will be required,
which does not appear in the NSE case.
These distinctions lead to different algebraic identities
\eqref{mag oscillation cancellation calculation}
and \eqref{vel oscillation cancellation calculation},
respectively,
for the nonlinear effects of magnetic perturbations
and velocity perturbations.

Cancellations provided by the specific structure of MHD nonlinearities
also would be used
towards the derivation of regularity and stability estimates,
when constructing the concentrated velocity and magnetic fields
in the gluing stage.

\subsection{Outline of the proof}

The proof of Theorem \ref{Thm-Non-gMHD} proceeds in two main stages:
the gluing stage and the convex integration stage.

The starting point of analysis is to consider
approximate solutions to the following
relaxation system with Reynolds and magnetic stresses:
for each integer $q \in \mathbb{N}$,
\begin{equation}\label{equa-mhdr}
	\left\{\aligned
	&\p_t \u+\nu_1(-\Delta)^{\alpha} \u+ \div(\u\otimes\u-\h\otimes\h)+\nabla P_q=\div \ru,  \\
	&\p_t \h+\nu_2 (-\Delta)^{\alpha} \h+ \div(\h\otimes\u-\u\otimes\h)=\div \rb , \\
	&\div \u = 0, \quad \div \h= 0,
	\endaligned
	\right.
\end{equation}
where the Reynolds stress $\ru$ is a symmetric traceless $3\times 3$ matrix,
and the magnetic stress $\rb$ is a skew-symmetric $3\times 3$ matrix.

The gluing stage aims to
construct new Reynolds and magnetic stresses,
which, particularly, have small disjoint temporal supports,
and, simultaneously, maintain small amplitudes.
Then, in the convex integration stage,
the crucial point is to construct appropriate velocity and magnetic perturbations,
featuring both the temporal and spatial intermittency,
to decrease the effect of Reynolds and magnetic stresses
and to control the hyper viscosity and resistivity
$(-\Delta)^\alpha$ where $\alpha$ can be larger than $5/4$.

In both stages,
the fundamental quantities are the frequency $\lambda_q$
and the amplitude $\delta_{q}$ below:
\begin{equation}\label{la}
	\la=a^{(b^q)}, \ \
	\delta_{q}= \lbb_1^{3\beta} \lambda_{q}^{-2\beta}.
\end{equation}
Here $a\in \mathbb{N}$ is a large integer to be determined later,
$\beta>0$ is the regularity parameter,
$b\in 2\mathbb{N}$ is a large integer of multiple $2$ such that
\begin{align}\label{b-beta-ve}
   b>\frac{1500}{\varepsilon\eta_*}, \ \
   0<\beta<\frac{1}{100b^2},  
\end{align}
where $\varepsilon\in \mathbb{Q}_+$ is sufficiently small such that,
for the given $(s,\gamma,p)\in \mathcal{S}_1$
\begin{equation}\label{ne3.1}
	\varepsilon\leq\frac{1}{20}\min\{ \frac 32 -\alpha,\,\frac{2\alpha}{\gamma}+\frac{2\a-2}{p}-(2\a-1)-s \}\quad \text{and}\quad b(2-\a-8\ve)\in \mathbb{N}.
\end{equation}
and for the given $(s,\gamma,p)\in \mathcal{S}_2$
\begin{equation}\label{e3.1}
	\varepsilon\leq\frac{1}{20}\min\{\frac 32-\alpha,\,\frac{4\a-4}{\gamma}+\frac{2}{p}-(2\a-1)-s \}\quad \text{and}\quad b\ve\in\mathbb{N},
\end{equation}
The objectives of both the gluing stage and the convex integration stage
are quantified in the iterative estimates
in Theorems \ref{Thm-nunr} and \ref{Thm-Iterat} below.

{\bf $\bullet$ Gluing stage.}
Let us first introduce  the notion of well-preparedness
as in the NSE context \cite{cl20.2,lqzz22}.

\begin{definition} (Well-preparedness)
 We say that
 the smooth solution $(u_q, B_q,\rr^u_q,\rr^B_q)$ to \eqref{equa-mhdr} on $[0,T]$ is well-prepared
 if there exist a set $I \subseteq [0,T]$, a length scale $\theta>0$ and
 $\eta>0$,
 such that $I$ is a union of at most $\theta^{-\eta}$ many closed intervals of length scale $5\theta$ and
  \begin{align*}
        \ru(t)=0,\ \rb(t)=0, \quad
	   \text{if}\  \operatorname{dist}(t,I^c)\leq \theta.
  \end{align*}
\end{definition}

We would divide the whole interval $[0,T]$ into $m_{q+1}$ many subintervals $[t_i, t_{i+1}]$,
and denote by $\theta_{q+1}$ the length scale of bad sets supporting
new concentrated Reynolds and magnetic stresses.
The parameters $m_{q+1}$ and $\theta_{q+1}$ are chosen
in the way
\begin{align}\label{def-mq-thetaq}
T/\mq=\la^{-12}, \quad
\theta_{q+1}:=(T/m_{q+1})^{1/\eta} \simeq \lambda_q^{-\frac{12}{\eta}},
\end{align}
where $\eta$ is a small constant such that
$$0<\frac{\eta_*}{2}<\eta<\eta_*<1.$$
Without loss of generality, we assume $\mq$ is an integer
such that the time interval is perfectly divided.

The main result in this stage is that,
we can construct new concentrated solutions to \eqref{equa-mhdr}
supported on $m_{q+1}$ intervals of length $\theta_{q+1}$,
and, importantly,
maintain the $L^1_{t,x}$-decay of Reynolds and magnetic stresses.
We note that,
the $L^1_{t,x}$-decay property is not destroyed in the concentrating procedure.
This is the content of Theorem \ref{Thm-nunr} below.

\begin{theorem}  [Well-preparedness]   \label{Thm-nunr}
Let $\alpha \in [1,3/2)$ and $(s,p,\gamma)\in \mathcal{S}_1 \cup  \mathcal{S}_2$.
If $(u_q,B_q,\rr^u_q,\rr^B_q)$ is a well-prepared smooth solution to \eqref{equa-mhdr}
for some set $I_q$ and a length scale $\theta_q$,
then there exists another well-prepared solution $(\wt u_q, \wt B_q,\tr,\trb)$ to \eqref{equa-mhdr}
for a new set $I_{q+1}\subseteq I_q$, $0,T\notin I_{q+1}$
and the smaller length scale $\theta_{q+1}(<\theta_q/2)$,
satisfying:
 \begin{align}
    & \|(\wt u_q, \wt B_q)\|_{L^\infty_t \cH^3_x}\lesssim \lambda_{q}^{5},\label{nuh3}\\
    & \|(\wt u_q -\u, \wt B_q -\h) \|_{L^{\infty}_t \cL^2_x} \lesssim  \la^{-3},\label{uuql2}\\
	& \|(\tr, \trb)\|_{L^{1}_{t}\cL^1_x} \leq  \la^{-\frac{\ve_R}{4}}\delta_{q+1},\label{nrl1}\\
    & \|(\partial_t^M \nabla^N \tr, \partial_t^M \nabla^N \trb)\|_{L^\infty_t \cH^3_x}
	\lesssim \thq^{-M-N-1} \lambda_q^{5}, \label{nrh3}
\end{align}
where $0\leq M\leq 1$, $0\leq N\leq 9$ and the implicit constants are independent of $q$.
Moreover, we have
\begin{align}
    & (\tr(t), \trb(t))=0 \quad \text{if} \quad \operatorname{dist}(t,I_{q+1}^c)\leq \frac{3}{2}\thq,\label{suppnr}  \\
    &  \supp_t (\wt u_{q},\wt B_q)\subseteq N_{2T/\mq}(\supp_t(u_{q},B_q)).  \label{suppwtuq}
\end{align}
\end{theorem}

The proof of Theorem \ref{Thm-nunr} is based on the gluing technique.
Precisely,
in every small neighborhood of length $\theta_{q+1}$
at $t_i(=iT/m_{q+1})$, $0\leq i \leq  m_{q+1}-1$,
we generate the velocity and magnetic fields to MHD equations \eqref{equa-gMHD}.
Then,
we glue all these local solutions together,
by using a partition of unity,
to get new concentrated velocity and magnetic fields.
The important roles here are played by the regularity and
stability estimates of solutions to (relaxation) MHD equations,
which are contained in Propositions \ref{Prop-LWP-Hyper-NLSE} and \ref{Prop-est-vi},
respectively.

{\bf $\bullet$ Convex integration stage.}
In this stage, the crucial inductive estimates of well-prepared solutions to \eqref{equa-mhdr} at level $q \in \mathbb{N}$
are as follows
\begin{align}
	& \|(\u, \h)\|_{L^\9_t \cH^{ N+3}_x} \lesssim \lambda_{q}^{ N+5}, \label{ubh3} \\
	& \|(\p_t \u, \p_t \h)\|_{L^\9_t \cH^{N }_x} \lesssim  \lambda_{q}^{N+5}, \label{ubpth2} \\
	& \|(\ru, \rb)\|_{L^\9_t \cH^{\wt N}_x} \lesssim  \lambda_{q}^{N+6},\label{rhN}	\\
	& \|(\ru, \rb)\|_{L^{1}_{t} \cL^1_x} \leq  \la^{-\ve_R}\delta_{q+1}, \label{rl1}
\end{align}
where $0\leq N\leq 4$, $\wt N=3,4$,
the implicit constants are independent of $q$,
and $\ve_R>0$ is a small parameter such that
\begin{align}\label{def-ver}
	\ve_R< \frac{b\beta}{10} < \frac{\ve \eta_*}{1.5\times 10^6}.
\end{align}
We note that, by \eqref{def-ver}, 
$$\la^{-\ve_R}\delta_{q+1}\ll 1\quad \text{for all}\quad q\geq 1, $$
and for $q=0$,
$$
\lambda_0^{-\ve_R}\delta_{1}= a^{b\beta-\ve_R}\gg 1,
$$
by choosing $a$ large enough.

The main iteration estimates
in the convex integration stage
are formulated in Theorem \ref{Thm-Iterat} below.

\begin{theorem} [Main iteration]\label{Thm-Iterat}
	Let $\alpha \in [1,3/2)$ and $(s,p,\gamma)\in \mathcal{S}_1 \cup \mathcal{S}_2$.
	Then, there exist $\beta\in (0,1)$,
    $M^*$ and $a_0=a_0(\beta, M^*)$ large enough,
    such that for any integer $a\geq a_0$,
	the following holds:
	
Suppose that
$(\u,\h,\ru,\rb)$ is a well-prepared solution to \eqref{equa-mhdr}
for some set $I_q$ and the length scale $\theta_q$
	and satisfies  \eqref{ubh3}-\eqref{rl1}.
	Then, there exists another well-prepared solution $(u_{q+1}, B_{q+1}, \mathring{R}^u_{q+1}, \mathring{R}^B_{q+1})$
	to \eqref{equa-mhdr} for some set $I_{q+1}\subseteq I_q$,
	$0,T\notin I_{q+1}$, and the length scale $\theta_{q+1}<\theta_q/2$,
	and $(u_{q+1}, B_{q+1}, \mathring{R}^u_{q+1}, \mathring{R}^B_{q+1})$ satisfies \eqref{ubh3}-\eqref{rl1} with $q+1$ replacing $q$.

	In addition, we have
	\begin{align}
		&\|(u_{q+1}-u_{q}, B_{q+1}-B_{q})\|_{L^{2}_{t} \cL^2_x} \leq  M^*\delta_{q+1}^{\frac{1}{2}}, \label{u-B-L2tx-conv}\\
        &\|(u_{q+1}-u_{q}, B_{q+1}-B_{q}) \|_{L^1_t \cL^{2}_{x}} \leq  \delta_{q+2}^{\frac{1}{2}},  \label{u-B-L1L2-conv}\\
		&\|(u_{q+1} - u_q, B_{q+1} - B_q)\|_{L^\gamma_t \cW^{s,p}_x} \leq  \delta_{q+2}^{\frac{1}{2}},\label{u-B-Lw-conv}
	\end{align}
and
\begin{align}
&\supp_t (u_{q+1}, B_{q+1}, \mathring{R}^u_{q+1}, \mathring{R}^B_{q+1})
\subseteq N_{\delta_{q+2}^{\frac12}}( \supp_t  (\u,\h,\ru,\rb)).\label{supprub}
\end{align}
\end{theorem}

The proof of Theorem \ref{Thm-Iterat} relies crucially on
the construction of appropriate velocity and magnetic perturbations,
adapted to the MHD geometry.
As a matter of fact,
different building blocks would be used
in the supercritical regimes
$\mathcal{S}_1$ and $\mathcal{S}_2$.

In the supercritical regime $\mathcal{S}_1$,
the spatial integrability can be close to infinity,
while the temporal integrability is less than $2$.
This fact indicates that,
the building blocks may have acceptable few spatial intermittency,
but require strong temporal intermittency.
The intermittent shear flows,
constructed by Beekie-Buckmaster-Vicol \cite{bbv20},
give the desired spatial intermittency,
and the temporal building blocks in \cite{cl20.2,lqzz22,lzz21}
are able to provide sufficiently strong temporal intermittency.

Hence, besides the parameters $\lbb_q$ and $\delta_{q}$,
three more parameters $\rs$, $\tau$ and $\sigma$ are needed
in the building blocks,
which parameterize the spatial concentration, temporal oscillation
and temporal concentration, respectively. 
It should be mentioned that, 
all these parameters shall be balanced with each other,
such that
the following leading order constrains are satisfied
\begin{subequations}\label{constset2}
	\begin{align}
		\lambda^s\rs^{\frac{1}{p}-\frac 12}\tau^{\frac12-\frac{1}{\gamma}} &\ll 1  \quad\ (w_{q+1}^{(p)},d_{q+1}^{(p)}\in L^\gamma_tW^{s,p}_x)  \label{setpw.2} \\
		\sigma\lambda^{-1}\rs^{\frac12}\tau^{\frac12}&\ll 1 \quad\ (\text{Time derivative error for}\ w_{q+1}^{(p)},d_{q+1}^{(p)}) \label{setpt.2} \\
		\lambda^{2\a-1}\rs^{\frac12}\tau^{-\frac12}&\ll 1 \quad\ (\text{Hyperdissipativity error for}\ w_{q+1}^{(p)},d_{q+1}^{(p)})  \label{setdeltap.2} \\
		\lambda^{-1}\rs^{-1 }&\ll 1 \quad\ (\text{Oscillation error for}\ w_{q+1}^{(p)},d_{q+1}^{(p)})  \label{setrosc1.2}
	\end{align}
\end{subequations}
It turns out that,
there do exist appropriate parameters to fulfill
the above constraints.
The specific choice is given by \eqref{larsrp-endpt2} below.

In contrast,
in the supercritical regime $\mathcal{S}_2$,
the spatial integrability may be less than two,
while the temporal integrability can be close to infinity.
Hence, stronger spatial intermittency is required,
in order to compensate the possibly weak temporal intermittency.
In this case, we use the spatial building blocks constructed in \cite{lzz21},
which provide 2D spatial intermittency. 
This requires five parameters $\rs, \rp, \mu, \tau, \sigma$ to
parameterize the spatial-temporal building blocks.
Note that,
the extra two parameters $\rp$ and $\mu$,
respectively, parameterize the further spatial concentration
and balance high temporal oscillations.
The crucial constrains in the convex integration scheme are captured by 
\begin{subequations}\label{constset}
	\begin{align}
		\lambda^s\rs^{\frac{1}{p}-\frac12}\rp^{\frac1p-\frac12}\tau^{\frac12-\frac{1}{\gamma}} &\ll 1 \quad\ (  w_{q+1}^{(p)},d_{q+1}^{(p)}\in L^\gamma_tW^{s,p}_x)   \label{setpw} \\
		\mu\rs^{\frac32}\rp^{-\frac12}\tau^{-\frac12}&\ll 1 \quad\ (\text{Time derivative error for}\  w_{q+1}^{(p)},d_{q+1}^{(p)})  \label{setpt} \\
		\lambda^{2\a-1}\rs^{\frac12}\rp^{\frac12}\tau^{-\frac12}&\ll 1  \quad\ (\text{Hyperdissipativity error for}\  w_{q+1}^{(p)},d_{q+1}^{(p)}) \label{setdeltap} \\
		\lambda^{2\a-1}\mu^{-1}&\ll 1  \quad\ (\text{Hyperdissipativity error for}\   w_{q+1}^{(t)},d_{q+1}^{(t)}) \label{setdeltat} \\
		\lambda^{-1}\rs^{-1 }&\ll 1 \quad\ (\text{Oscillation error for}\   w_{q+1}^{(p)},d_{q+1}^{(p)})\label{setrosc1} \\
		\mu^{-1} \sigma\tau&\ll 1 \quad\ (\text{Oscillation error for}\   w_{q+1}^{(t)},d_{q+1}^{(t)}) \label{setrosc2}
	\end{align}
\end{subequations}
Obviously, the admissible set of the constraints \eqref{setpw}-\eqref{setrosc2}
shall be different from that of \eqref{setpw.2}-\eqref{setrosc1.2}. 
The precise choice of these five parameters 
will be given by \eqref{larsrp} below.

It should be mentioned that,
in both supercritical regimes $\mathcal{S}_1$ and $\mathcal{S}_2$,
it is crucial to exploit the temporal intermittency,
particularly, to handle the hyper viscosity and resistivity
beyond the Lions exponent.
Actually,
as we have already seen above,
the uniqueness of weak solutions in $C_tL^2_x$ holds in
the high-dissipative regime when $\alpha\geq 5/4$.
Hence,
it is impossible to construct non-unique weak solutions in $C_tL^2_x$
with merely spatial intermittent flows.
Below we shall see that,
the temporal building blocks would provide 2D intermittency
in the supercritical regime $\mathcal{S}_2$, 
the temporal building blocks in $\mathcal{S}_1$ can even provide 
the stronger 3D intermittency,
which in particular compensates the weak spatial intermittency 
and enables to control the hyper viscosity and resistivity 
beyond the Lions exponent. \\

{\bf Organization of paper. }
The remaining part of this paper is structured as follows.
First,
\S \ref{Sec-Concen-Rey} is devoted to the gluing stage.
We prove the regularity and stability results for the hyper
viscous and resistive MHD equations
and then use the gluing technique to prove Theorem \ref{Thm-nunr}.
Then,
in \S \ref{Sec-perturb-S1} and \S \ref{Sec-stress-S1},
we mainly treat the supercritical regime $\mathcal{S}_1$
by using the intermittent convex integration approach.
The other supercritical regime $\mathcal{S}_2$ is later treated in \S \ref{Sec-S2}.
Then, in \S \ref{Sec-Proof-Main},
we prove the main results, namely, Theorems \ref{Thm-Non-gMHD} and  \ref{Thm-Iterat}.
At last, some standard tools of convex integration method
and the uniqueness of weak solutions are collected in
Appendices \ref{Sec-App-A} and \ref{Sec-App-B}. \\

{\bf Notations.}
The mean of $u \in L^1(\T^n)$ is given by
$
\aint_{\T^n} u \d x =  {|\T^n|}^{-1} \int_{\T^n} u \d x ,
$
where $|\cdot|$ denotes the Lebesgue measure.
For $p\in [1,+\infty]$ and $s\in \R$, we use the following shorthand notations
\begin{align*}
	L^p_x:=L^p_x(\T^3),\quad H^s_x:=H^s_x(\T^3), \quad W^{s,p}_x:=W^{s,p}_x(\T^3).
\end{align*}
Moreover, let
\begin{align*}
	\norm{u}_{W^{N,p}_{t,x}}:=\sum_{0\leq m+|\zeta|\leq N} \norm{\p_t^m \na^{\zeta} u}_{L^p_{t,x}}, \ \
     \norm{u}_{C_{t,x}^N}:=\sum_{0\leq m+|\zeta|\leq N}
	  \norm{\p_t^m \na^{\zeta} u}_{C_{t,x}},
\end{align*}
where $\zeta=(\zeta_1,\zeta_2,\zeta_3)$ is the multi-index
and
$\na^\zeta:= \partial_{x_1}^{\zeta_1} \partial_{x_2}^{\zeta_2}
\partial_{x_3}^{\zeta_3}$.
We also consider the product space
$\mathcal{L}^p_x:= L^p_x \times L^p_x$,
equipped with the norm
$\|(u,v)\|_{\cL^p_x}:= \|u\|_{L^p_x} + \|v\|_{L^p_x}$
if $(u,v) \in \cL^p_x$.
Similar notations also apply to $\cH^s_x$ and $\cW^{s,p}$.

Given any Banach space $X$,
$\gamma \in [1,\infty]$,
$L^\gamma_tX$ denotes  the space of integrable functions
from $[0,T]$ to $X$,
equipped with the norm
$\|u\|_{L^\gamma_tX}:=  (\int_{0}^T \|u(t)\|_X^\gamma dt)^{1/\gamma}$
(with the usual adaptation when $\gamma =\infty$).
In particular, we write
$L^\gamma_tL^p_x :=L^\gamma_t([0,T];L^p(\T^3))$,
$L^p_{t,x}:= L^p_tL^p_x$ and
$C_{t,x}:=C_tC_x$.

We also adopt the notations from \cite{LQ14}.
Let $u,\,v$ be two vector fields,
the corresponding second order tensor product is defined by
\begin{align*}
u\otimes v:=(u_iv_j)_{1\leq i,j\leq 3}.
\end{align*}
For any second-order tensor $A=(a_{ij})_{1\leq i,j\leq 3}$,
set
\begin{align*}
\div A:= \(\sum_{j=1}^3 {\p_{x_j} a_{1j}} ,\sum_{j=1}^3 {\p_{x_j}  a_{2j}},\sum_{j=1}^3 {\p_{x_j}  a_{3j}} \)^\top.
\end{align*}
The right product of a vector field $v=(v_1,v_2,v_3)^\top$ to a second-order tensor $A=(a_{ij})_{1\leq i,j\leq 3}$ is defined by
\begin{align*}
Av:=\(\sum_{j=1}^{3} a_{1j}v_j,\sum_{j=1}^{3} a_{2j}v_j,\sum_{j=1}^{3} a_{3j}v_j \)^\top.
\end{align*}
In particular, for any scalar function $f$ and second-order tensor $A=(a_{ij})_{1\leq i,j\leq 3}$,
one has the Leibniz rule
\begin{align*}
	\div(fA)= f\div A+A\nabla f.
\end{align*}
For any $3\times 3$ matrices $A=({A_{ij}})$ and $S=({S_{ij}})$, let
	$ A:S=\sum_{i,j=1}^{3}A_{ij}S_{ij}$.

Let $N_{\va_*}(A)$ denote the $\ve_*-$ neighborhood of $A\subseteq [0,T]$,
i.e.,
\begin{align*}
	N_{\va_*}(A):=\{t\in [0,T]:\ \exists s\in A,\ s.t.\ |t-s|\leq \va_*\}.
\end{align*}

Let $\P_{H}$ denote the Helmholtz-Leray projector, i.e.,
$\P_{H}=\Id-\nabla\Delta^{-1}\div$.

The notation $a\lesssim b$ means that $a\leq Cb$ for some constant $C>0$.

\section{Concentrating the Reynolds and magnetic stresses}  \label{Sec-Concen-Rey}

Let us start with the gluing stage.
This section mainly consists of four parts:
the regularity estimates for  hyper viscous and resistive MHD,
the stability estimates for relaxed MHD equations,
the temporal gluing of local solutions
and the proof of Theorem \ref{Thm-nunr}.

\subsection{Regularity estimates} \label{Subsec-Reg-esti}

Let us first establish the regularity estimates of local strong solutions to \eqref{equa-gMHD}
when $\a \in [1,3/2)$.
This is the content of Proposition \ref{Prop-LWP-Hyper-NLSE} below.

\begin{proposition} [Regularity estimates] \label{Prop-LWP-Hyper-NLSE}
Let $\a \in [1,3/2)$, $v_0,B_0\in H_x^3$ be
any mean-free and divergence-free fields.
Consider the
Cauchy problem for \eqref{equa-gMHD} with initial condition
$(u,B)|_{t=t_0}=(u_0,B_0)$.
Then,
there exists $c>0$ sufficiently small such that
the following holds:

$(i)$. If
\begin{align} \label{t*-t0-v0H3}
	 0<t_*-t_{0}\leq \frac{c}{\|(u_0,B_0)\|_{\cH^3_x} },
\end{align}
then there exists a unique strong solution $(u,B)$ to \eqref{equa-gMHD} on
$[t_0,t_{*}]$ satisfying
\begin{align}
 \|(u, B)\|_{C([t_0,t_{*} ]; \cL^{2}_x)}^{2}
+ \int_{t_{0}}^{t_{*}} \|(u(t),B(t))\|_{\dot{\mathcal{H}}_x^{\alpha}}^{2} \d t
& \lesssim \|(u_0, B_0)\|_{\cL^{2}_x}^{2}, \label{vl2}\\
  \|(u,B)\|_{C([t_0,t_{*}]; \cH^{3}_x)}
  &\lesssim \left\|(u_0,B_0)\right\|_{\cH^{3}_x}. \label{vh3}
\end{align}
$(ii)$. If
\begin{align}  \label{con-pdvh3}
 0< t_{*}-t_0 \leq \min\left\{1,\, \frac{ c}{ \norm{(u_0,B_0)}_{\cH^3_x} (1+ \norm{(u_0,B_0)}_{\cL^2_x})}\right\},
\end{align}
then it holds that for any $N\geq 0$ and $M \in\{0,1\}$,
\begin{align} \label{pdvh3}
\sup_{t\in [t_0,t_{*}]}(|t-t_0|^{M+\frac{N}{2\alpha}}
\norm{(\partial_t^M \na^N u(t),\partial_t^M \na^N B(t))}_{\cH_x^3}
     \lesssim  \norm{(u_0,B_0)}_{\cH^3_x}\,,
\end{align}
where the implicit constant depends on $\alpha,\, N$ and $M$.
\end{proposition}

\begin{proof}
$(i)$. By the energy inequality, one has
\[
\frac 12 \frac{d}{dt} (\norm{u(t)}_{L^2_x}^2+\norm{B(t)}_{L^2_x}^2)
\leq - \nu_1\norm{u}_{\dot{H}^\alpha_x}^2- \nu_2\norm{B}_{\dot{H}^\alpha_x}^2.
\]
which directly yields \eqref{vl2} by Gronwall's inequality.

Next, taking the $\dot{H}^3$ inner product of
the velocity equation in \eqref{equa-gMHD} with $u$ and
the magnetic equation  with $B$,
since $\nabla \cdot u=\nabla \cdot B = 0$,
we have
\begin{align}\label{inner-h3}
	&\quad \frac 12 \frac{\d}{\d t}(\norm{u(t)}_{\dot{H}^3_x}^2+\norm{B(t)}_{\dot{H}^3_x}^2)
  + \nu_1 \norm{u(t)}_{\dot{H}_x^{3+\alpha}}^2
  + \nu_2 \norm{B(t)}_{\dot{H}_x^{3+\alpha}}^2\notag\\
&= - ((u\cdot \nabla )u, u)_{\dot{H}^3_x}
   + ( (B\cdot \nabla )B,u )_{\dot{H}^3_x}
   - ((u\cdot \nabla )B,B)_{\dot{H}^3_x}
   + ((B\cdot \nabla )u,B )_{\dot{H}^3_x}.
\end{align}
By Kato's inequality (cf. \cite[P.155]{rrs16}),
\begin{align}\label{e2.8}
	((u\cdot \nabla )u, u)_{\dot{H}^3_x}
	\lesssim \|\nabla u\|_{L^\9_x} \|u \|_{H^3_x}^2,
\end{align}
and, similarly,
\begin{align}\label{e2.9}
	((u\cdot \nabla )B, B)_{\dot{H}^3_x}
	\lesssim \|(\nabla u, \nabla B)\|_{\cL^\9_x}
	\|(u, B)\|_{\cH^3_{x}}^2.
\end{align}
For the remaining two terms on the right-hand-side of \eqref{inner-h3},
it is important to use the cancellation of the
fourth-order derivative terms,
which permits to obtain the desirable $\dot{H}^3_x$ estimate.
Actually, we compute
\begin{align}
&((B\cdot \nabla )B, u)_{\dot{H}^3_x}+((B\cdot \nabla )u,B)_{\dot{H}^3_x}
= \sum_{|\a|=3} \sum_{i,j}\int_{\T^3} \nabla^\a (B_i\p_{x_i} B_j) \nabla^\a u_j+ \nabla^\a (B_i\p_{x_i} u_j) \nabla^\a B_j \d x.
\end{align}
Since $\nabla \cdot u=\nabla \cdot B = 0$,
using the integration-by-parts formula
we obtain the cancellation
for the terms of the highest fourth order derivatives
\begin{align}
  \sum_{i,j}\int_{\T^3} B_i \nabla^3 \p_{x_i} B_j \nabla^3 u_j+ B_i \nabla^3 \p_{x_i} u_j \nabla^3 B_j \d x=0.
\end{align}
Moreover, H\"{o}lder's inequality gives
\begin{align}
&  \sum_{i,j}\int_{\T^3} \nabla^3 B_i  \p_{x_i} B_j \nabla^3 u_j+ \nabla^3 B_i \p_{x_i} u_j \nabla^3 B_j \d x
\lesssim \|\nabla B\|_{L^\9_x}\|B\|_{\dot H^3_x}\|u\|_{\dot H^3_x}
        +\|\nabla u\|_{L^\9_x}\|B\|_{\dot H^3_x}^2,
\end{align}
and
\begin{align}
	&\sum_{i,j} \int_{\T^3} \nabla B_i  \nabla^{2}\p_{x_i} B_j \nabla^\a u_j
	+ \nabla B_i \nabla^{2}\p_{x_i}'s u_j \nabla^\a B_j \d x
	 \lesssim \|\nabla B\|_{L^\9_x}\|B\|_{\dot H^3_x}\|u\|_{\dot H^3_x} .
\end{align}
We also derive from the H\"{o}lder's inequality
and the Gagliardo-Nirenberg inequality that
\begin{align} \label{h3-e2.13}
&\quad \sum_{i,j} \int_{\T^3} \nabla^{2} B_i  \nabla^{1}\p_{x_i} B_j \nabla^\a u_j+ \nabla^{2} B_i \nabla^{1}\p_{x_i} u_j \nabla^\a B_j \d x \notag\\
&\lesssim \|\nabla^2 B\|_{L^4_x}^2\|u\|_{\dot{H}^3_x} +\|\nabla^2 B\|_{L^4_x} \|\nabla^2 u\|_{L^4_x} \|B\|_{\dot{H}^3_x}\notag\\
&\lesssim \|\nabla  B\|_{L^\9_x}\| B\|_{\dot{H}^3_x} \|u\|_{\dot{H}^3_x}
+\|\nabla  B\|_{L^\9_x}^{\frac12}\|\nabla  u\|_{L^\9_x}^{\frac12}\| u\|_{\dot{H}^3_x}^{\frac12} \|B\|_{\dot{H}^3_x}^{\frac32}  .
\end{align}
Thus, combining \eqref{inner-h3}-\eqref{h3-e2.13} together
and using the Sobolev embedding $\dot{H}^2_x \hookrightarrow L^\infty_x$
we come to
\begin{align}\label{local-h3}
\frac{\d}{\d t} \|(u(t), B(t))\|_{\dot{\mathcal{H}}^3_x}^2
+ \| (u(t), B(t))\|_{\dot{\cH}_x^{3+\alpha}}^2
\lesssim  \|(u(t), B(t))\|_{\dot \cH^3_x}^3,
\end{align}
which, via Gronwall's inequality, yields
\begin{align}\label{local-h3-gron}
	\|(u(t), B(t))\|_{\dot{\cH}^3_x}^2
	 \leq C \| (u_0, B_0)\|_{\dot \cH^3_x}^2
	  \exp\left\{\int_{0}^{t} \|(u(s),B(s))\|_{\dot \cH^3_x} \d s \right\}.
   \end{align}

We claim that
\begin{align}
    \sup _{t \in\left[t_0,t_{*}\right]}\|(u(t),B(t))\|_{\dot \cH^{3}_x}
  &\leq 2\left\|(u_0,B_0)\right\|_{\dot \cH^{3}_x}. \label{claim-vh3}
\end{align}
To this end, plugging \eqref{claim-vh3} into \eqref{local-h3-gron}, we get
\begin{align}\label{est-bs-1}
	\|(u(t), B(t))\|_{\dot{\cH}^3_x}^2
	\leq C  \| (u_0, B_0)\|_{\dot \cH^3_x}^2 \exp\left\{2t_*\norm{(u_0,B_0)}_{\dot \cH^3_x}\right\}.
\end{align}
In view of \eqref{t*-t0-v0H3},
we deduce that for $c$ sufficiently small
such that $\exp\left\{2t_*\norm{(u_0,B_0)}_{\cH^3_x}\right\}\leq e^{2c} \leq \frac32$,
\begin{align*}
  \|(u(t), B(t))\|_{\dot{\cH}^3_x}^2
  \leq \frac32  C \|(u_0, B_0)\|_{\dot \cH^3_x}^2,
\end{align*}
which, via bootstrap arguments, yields \eqref{claim-vh3}
for any $t\in[t_0, t_*]$,
thereby proving \eqref{vh3}.

$(ii)$.
Now we prove estimate \eqref{pdvh3}.
For this purpose,
we reformulate \eqref{equa-gMHD} as
\begin{subequations}\label{eq-mild}
\begin{align}
u(t) = e^{-(t-t_0) \nu_1 (-\Delta)^{\alpha}} u_0
       - \int_{t_0}^t e^{-(t-s) \nu_1 (-\Delta)^{\alpha}} {\mathbb P}_H \div (u(s) \otimes u(s)-B(s)\otimes B(s)) \d s, \label{eq-umild}\\
B(t) = e^{-(t-t_0) \nu_2 (-\Delta)^{\alpha}} B_0
       - \int_{t_0}^t e^{-(t-s) \nu_2 (-\Delta)^{\alpha}} {\mathbb P}_H \div (B(s) \otimes u(s)-u(s)\otimes B(s)) \d s, \label{eq-bmild}
\end{align}
\end{subequations}

Without loss of generality,
we consider only the case $t_0=0$ below.
We will frequently use the following estimates, via
the Gagliardo-Nirenberg inequality,
for any $f,g\in H^3_x$,
\begin{align}\label{esti-GN-H3}
	\|fg\|_{H^3_x}\leq \|f\|_{H^3_x}\|g\|_{L^\9_x}+ \|f\|_{L^\9_x}\|g\|_{H^3_x}, \ \
    \|f\|_{L^\9_x}\leq \|f\|_{H^3_x}^\frac12\|f\|_{L^2_x}^\frac12.
\end{align}

First note that,
in the case where $N=M=0$,
\eqref{pdvh3} follows immediately from \eqref{vh3}.

$(ii.1)$. The case where $M=0$ and $N=1$.
Applying Lemma \ref{Lem-semi-est} to equations
\eqref{eq-umild}-\eqref{eq-bmild}
and using the boundedness of ${\mathbb P}_H$ on $L^2$,
we get
\begin{align}\label{est-d1ub}
	&\quad t^{\frac{1}{2\alpha}}\norm{(\nabla u(t),\nabla B(t))}_{\cH_x^3} \notag\\
	&\leq  C_*\norm{ (u_0,B_0)}_{\cH^3_x}
	+C_* t^{\frac{1}{2\alpha}}
	\int_{0}^t (t-s)^{-\frac{1}{2\alpha}}
	\bigg(\|\div(u(s)\otimes u(s))\|_{H^3_x}+\|\div(B(s)\otimes B(s))\|_{H^3_x} \notag\\
	&\qquad \qquad \qquad \qquad \qquad \qquad \qquad
	  + \|\div(B(s)\otimes u(s))\|_{H^3_x}+\|\div(u(s)\otimes B(s) )\|_{H^3_x} \bigg)\d s.
\end{align}
where $C_*(>1)$ depends on $\alpha$.

We claim that, for all $t\in [0,t_*]$,
\begin{align}\label{claim-n1ub}
t^{\frac{1}{2\alpha}}\norm{(\nabla u(t),\nabla B(t))}_{\cH_x^3}
\leq 2 C_* \norm{(u_0, B_0)}_{\cH_x^3}.
\end{align}

We prove \eqref{claim-n1ub} by using bootstrap arguments.
More precisely, applying \eqref{claim-n1ub}, Gagliardo-Nirenberg inequality, \eqref{vh3},
\eqref{pdvh3} with $N=0$, $M=0$,
and the inequality \eqref{esti-GN-H3} to the nonlinearities of \eqref{est-d1ub}, we obtain
\begin{align}\label{est-bu}
	\|\div (B\otimes u)\|_{H^3_x}
	&\lesssim  \|u\|_{H^3_x} \|\nabla B\|_{H^3_x}^{\frac12} \|\nabla B\|_{L^2_x}^\frac12
	  +\|u\|_{H^3_x}^\frac12 \|u\|_{L^2_x}^{\frac12} \|\nabla B\|_{H^3_x}\notag\\
	&\leq  C_1 C_* (s^{-\frac{1}{4\a}} \|(u_0,B_0)\|_{\cH^3_x}^{\frac{5}{3}} \|(u_0,B_0)\|_{\cL^2_x}^\frac{1}{3}
	+ s^{-\frac{1}{2\a}}\|(u_0,B_0)\|_{\cH^3_x}^\frac32 \|(u_0,B_0)\|_{\cL^2_x}^{\frac12}) ,
\end{align}
which, along with Young's inequality, yields that
\begin{align}\label{est-divbu}
	\|\div (B\otimes u)\|_{H^3_x}
	\leq C_2 C_*s^{- \frac{1}{2\alpha}} \norm{(u_0,B_0)}_{\cH^3_x}^{\frac 32}
	( s^{\frac{3}{4\alpha}} \norm{(u_0,B_0)}_{\cH^3_x}^\frac12 +\norm{(u_0,B_0)}_{\cL^2_x}^{\frac 12}).
\end{align}
Similarly, $\|\div (u \otimes u)\|_{H^3_x}$, $\|\div (B \otimes B)\|_{H^3_x}$ and $\|\div (u \otimes B)\|_{H^3_x}$ obey the same upper bound.
Thus, turning back to \eqref{est-d1ub},
we get 
\begin{align}\label{est-d1bu}
	&\quad t^{\frac{1}{2\alpha}}
	\norm{(\nabla u(t),\nabla B(t))}_{\cH_x^3} \notag\\
	&\leq  C_*\norm{(u_0,B_0)}_{\cH^3_x}
	  + C_3 C_* \norm{(u_0,B_0)}_{\cH^3_x}^{2} t^{\frac{1}{2\alpha}} \int_{0}^t  (t-s)^{-\frac{1}{2\alpha}}	s^{\frac{1}{4\alpha}} \d s\notag\\
	&\quad + C_3 C_* \norm{(u_0,B_0)}_{\cH^3_x}^{\frac32}\norm{(u_0,B_0)}_{\cL^2_x}^{\frac12}   t^{\frac{1}{2\alpha}} \int_{0}^t  (t-s)^{-\frac{1}{2\alpha}}s^{-\frac{ 1}{2\alpha}} \d s \notag\\
	&\leq  C_*\norm{(u_0,B_0)}_{\cH^3_x}
	(1+  C_4 C_* t^{1+\frac{1}{4\a}}  \norm{(u_0,B_0)}_{\cH^3_x}
	+ C_4C_* t^{1-\frac{1}{ 2\alpha}}  \norm{(u_0,B_0)}_{\cH^3_x}^{\frac 12}
	\norm{(u_0,B_0)}_{\cL^2_x}^{\frac12}),
\end{align}
where $C_4(\geq 1)$ is a universal constant.
Thus, choosing $t_*$ sufficiently small such that
\begin{align*}
0<t_*\leq
\min \left\{ 1, \frac{1}{16C_4^2C_*^2
\norm{(u_0,B_0)}_{\cH^3_x}
(1+ \norm{(u_0,B_0)}_{\cL^2_x})} \right\},
\end{align*}
we arrive at
\begin{align}\label{est-d1bu-end}
 t^{\frac{1}{2\alpha}}\norm{(\nabla u(t),\nabla B(t))}_{\cH_x^3}
 \leq \frac32 C_*\norm{(u_0,B_0)}_{\cH^3_x},
\end{align}
which, via bootstrap argument, yields \eqref{claim-n1ub}, as claimed.
Thus, \eqref{pdvh3} is proved when $N=1$ and $M=0$.

$(ii.2)$. The case where $N=2$ and $M=0$.
We apply Lemma \ref{Lem-semi-est} again to
\eqref{eq-bmild} and then use \eqref{esti-GN-H3}
to get
\begin{align}\label{est-d2b}
	t^{\frac{1}{ \alpha}}\norm{\nabla^2 B(t)}_{H^3_x}
	&\leq C_* \norm{B_0}_{H^3_x}
	  +  C_* t^{\frac{1}{ \alpha}}
	\int_{0}^{\frac t2} (t-s)^{-\frac{3}{2\alpha}}
	 \|(u(s),B(s))\|_{\cH_x^3}^{\frac 32}
	 \|(u(s),B(s))\|_{\cL_x^3}^{\frac 12}  \d s \\
    &  + C_*t^{\frac{ 1}{ \alpha}}\int_{\frac t2}^t (t-s)^{-\frac{1}{2\alpha}}
    (\|\nabla^2 (B(s)\otimes u(s))\|_{H^3_x}+\|\nabla^2 (u(s)\otimes B(s))\|_{H^3_x} )\d s.
\end{align}

We claim that
\begin{align}\label{claim-n2ub}
	t^{\frac{1}{ \alpha}}\norm{(\nabla^2 u(t),\nabla^2 B(t))}_{\cH_x^3}
	\leq 2 C_* \norm{(u_0, B_0)}_{\cH_x^3},
\end{align}
holds for all $t\in [0,t_*]$.

In order to prove \eqref{claim-n2ub},
let us estimate the right-hand-side of \eqref{est-d2b}.
By the Leibniz rule we have
 \begin{align}\label{est-n2bu-non}
\|\nabla^2 (B\otimes u)\|_{H^3_x}+\|\nabla^2 (u\otimes B)\|_{H^3_x}
 &\leq 2\sum_{j=0}^{2} (\norm{\nabla^j u}_{ H^3_x}\norm{\nabla^{n-j} B}_{L^\9_{ x}}
 +   \norm{\nabla^{n-j} B}_{ H^3_x}\norm{\nabla^j u}_{L^\9_{ x}}).
  \end{align}
Note that,
by the Gagliardo-Nirenberg inequality,
the Young inequality \eqref{vh3} and \eqref{esti-GN-H3},
\begin{align}\label{est-d2bu-part1}
	&\quad \sum_{j=0}^{2} \norm{\nabla^j u}_{ H^3_x}\norm{\nabla^{n-j} B}_{L^\9_{ x}}\notag\\
	&\lesssim
	\norm{ u}_{H^3_x}\norm{\nabla^{2} B}_{H^3_x}^{\frac12}\norm{\nabla^{2} B}_{L^2_x}^{\frac12} +  \norm{\nabla u}_{H^3_x}\norm{\nabla  B}_{H^3_x}^{\frac12}\norm{\nabla  B}_{L^2_x}^{\frac12}+ \norm{\nabla^2 u}_{H^3_x}\norm{  B}_{H^3_x}^{\frac12}\norm{  B}_{L^2_x}^{\frac12}\notag\\
	&\leq C_1'C_* \Big(s^{-\frac{1}{2\a}} \norm{(u_0,B_0)}_{\cH^3_x}^{\frac{11}{6}} \norm{(u_0,B_0)}_{\cL^2_x}^{\frac16}
	+ s^{-\frac{3}{4\a}}  \norm{(u_0,B_0)}_{\cH^3_x}^{\frac{5}{3}} \norm{(u_0,B_0)}_{\cL^2_x}^{\frac13}\notag\\
	&\qquad\qquad + s^{-\frac{1}{\a}}  \norm{(u_0,B_0)}_{\cH^3_x}^{\frac{3}{2}} \norm{(u_0,B_0)}_{\cL^2_x}^{\frac12} \Big)\notag\\
	&\leq C_2'C_* s^{- \frac{1}{ \alpha}} \norm{(u_0,B_0)}_{\cH^3_x}^{\frac 32}
( s^{\frac{3}{4\alpha}} \norm{(u_0,B_0)}_{\cH^3_x}^\frac12 +\norm{(u_0,B_0)}_{\cL^2_x}^{\frac 12} ).
\end{align}
One can estimate the other terms in \eqref{est-n2bu-non}
in a similar manner.
It follows that
 \begin{align}\label{est-n2bu-non-end}
\|\nabla^2 (B\otimes u)\|_{H^3_x}+\|\nabla^2 (u\otimes B)\|_{H^3_x}
&\leq 4C_2'C_* s^{- \frac{1}{ \alpha}} \norm{(u_0,B_0)}_{\cH^3_x}^{\frac 32}
( s^{\frac{3}{4\alpha}} \norm{(u_0,B_0)}_{\cH^3_x}^\frac12 +\norm{(u_0,B_0)}_{\cL^2_x}^{\frac 12} ).
\end{align}
Plugging this into \eqref{est-d2b}
and using \eqref{vl2} and \eqref{vh3}
lead to
\begin{align}\label{est-d2b2}
t^{\frac{1}{ \alpha}}\norm{\nabla^2 B(t)}_{H^3_x}
&\leq C_* \norm{B_0}_{H^3_x}
+ C_3't^{\frac{1}{ \alpha}}\norm{(u_0,B_0)}_{\cH^3_x}^{\frac 32}\norm{(u_0,B_0)}_{\cL^2_x}^{\frac12}
\int_{0}^{\frac{t}{2}} (t-s)^{-\frac{3}{2\alpha}}\d s\notag\\
&\quad +C_3'C_* t^{\frac{1}{\alpha}}\norm{(u_0,B_0)}_{\cH^3_x}^{\frac32}
\int_{\frac{t}{2}}^t  (t-s)^{-\frac{1}{2\alpha}}s^{-\frac{ 1}{\alpha}}
\(s^{\frac{3}{4\alpha}} \norm{(u_0,B_0)}_{\cH^3_x}^{\frac12} + \norm{(u_0,B_0)}_{\cL^2_x}^\frac12 \) \d s \notag\\
&\leq   C_* \norm{B_0}_{H^3_x}
+ C_4'C_* t^{1-\frac{1}{ 2\alpha}}  \norm{(u_0,B_0)}_{\cH^3_x}^{\frac 32}
\norm{(u_0,B_0)}_{\cL^2_x}^{\frac12}
+ C_4'C_*t^{1+\frac{1}{4\a}}  \norm{(u_0,B_0)}_{\cH^3_x}^{2}
\end{align}
for some universal constant $C_4'$.
Hence, taking $t_*$ possibly smaller such that
\begin{align}\label{assume-t}
	0<t_*\leq \min
	\left\{ 1, \frac{1}{16(C'_4)^2C_*^2 \norm{(u_0,B_0)}_{\cH^3_x} (1+ \norm{(u_0,B_0)}_{\cL^2_x})} \right\}
\end{align}
we deduce from \eqref{est-d2b2} that
for all $t\in [0,t_*]$,
\begin{align}
	t^{\frac{1}{ \alpha}}\norm{\nabla^2 B(t)}_{H^3_x}
	\leq \frac 32 C_* \|B_0\|_{H^3_x}.
\end{align}
The estimate of $\norm{\nabla^2 u}_{H^3_x}$
can be proved in a similar manner.
Hence, using bootstrap arguments
we obtain \eqref{claim-n2ub}, as claimed.
This verifies \eqref{pdvh3} when $N=2$ and $M=0$.

$(ii.3)$. The general case where $N\geq 3$ and $M=0$.
In order to treat the general case
we use the induction arguments.
Suppose that \eqref{pdvh3} is valid for all $N'\leq N$ and $M=0$.
By \eqref{eq-mild}, we have
\begin{align}\label{est-dnu}
		t^{\frac{N}{2\alpha}}
		\norm{(\nabla^N u(t), \na^N B(t))}_{\cH^3_x}
		\leq& C_*\norm{(u_0,B_0)}_{\cH^3_x}
		+ C_* t^{\frac{N}{2\alpha}} \int_{0}^{\frac t2}
		(t-s)^{-\frac{N+1}{2\a}}(\|u \otimes u\|_{H^3_x}
		+\|B\otimes B\|_{H^3_x} )\d s\notag\\
		&+ C_* t^{\frac{N}{2\alpha}}
		\int_{\frac t2}^t (t-s)^{-\frac{1}{2\a}}
		\bigg(\|\nabla^{N}(u\otimes u)\|_{H^3_x}
		+\|\nabla^{N}(B\otimes B)\|_{H^3_x} \notag\\
		&\qquad \qquad \quad
		+ \|\nabla^{N}(B\otimes u)\|_{H^3_x}
		       + \|\nabla^{N}(u\otimes B)\|_{H^3_x} \bigg)\d s.
\end{align}

As in \eqref{claim-n1ub} and \eqref{claim-n2ub},
we claim that
\begin{align}\label{claim-dnub}
	t^{\frac{N}{ 2\alpha}}\norm{(\nabla^N u(t),\nabla^N B(t))}_{\cH_x^3}
	\leq 2 C_* \norm{(u_0, B_0)}_{\cH_x^3}, \ \
	\forall\ t\in [0,t_*].
\end{align}

To this end, concerning the right-hand side of \eqref{est-dnu}, by the Leibniz rule,
\begin{align}\label{est-dnbu}
\norm{\nabla^{N} (B \otimes u)}_{ H^3_x}
    +\norm{\nabla^{N} (u \otimes B)}_{H^3_x}
    &\lesssim\sum_{j=0}^{N} (\norm{\nabla^j u}_{ H^3_x}\norm{\nabla^{N-j} B}_{L^\9_{ x}}
    +   \norm{\nabla^{N-j} B}_{ H^3_x}\norm{\nabla^j u}_{L^\9_{ x}}).
\end{align}
Using \eqref{esti-GN-H3}, \eqref{claim-dnub} and the induction assumption,
we get
\begin{align*}
\sum_{j=0}^{N} \norm{\nabla^j u}_{ H^3_x}\norm{\nabla^{N-j} B}_{L^\9_{ x}}
&\lesssim \sum_{j=0}^{N-3}  \norm{\nabla^j u}_{H^3_x}\norm{\nabla^{N-j} B}_{H^3_x}^{\frac12}\norm{\nabla^{N-j-3} B}_{H^3_x}^{\frac12} +
\norm{\nabla^{N-2} u}_{H^3_x}\norm{\nabla^{2} B}_{H^3_x}^{\frac12}\norm{\nabla^{2} B}_{L^2_x}^{\frac12}\notag\\
&\qquad\ +  \norm{\nabla^{N-1} u}_{H^3_x}
\norm{\nabla  B}_{H^3_x}^{\frac12}
\norm{\nabla  B}_{L^2_x}^{\frac12}
+ \norm{\nabla^N u}_{H^3_x}\norm{B}_{H^3_x}^{\frac12}
\norm{  B}_{L^2_x}^{\frac12}  \notag  \\
&\leq \wt C_1C_*\Big(s^{-\frac{N}{2\a}
+\frac{ 3}{4\a}}
\norm{ (u_0,B_0)}_{\cH^3_x}^2
+ s^{-\frac{N}{2\a}+\frac{1}{2\a}}
\norm{(u_0,B_0)}_{\cH^3_x}^{\frac{11}{6}}
 \norm{(u_0,B_0)}_{\cL^2_x}^{\frac16}   \notag  \\
&\quad + s^{-\frac{N}{2\a}+\frac{1}{4\a}}
\norm{(u_0,B_0)}_{\cH^3_x}^{\frac53}
\norm{(u_0,B_0)}_{\cL^2_x}^{\frac13}
     + s^{-\frac{N}{2\a}}
	 \norm{(u_0,B_0)}_{\cH^3_x}^{\frac32}
	 \norm{(u_0,B_0)}_{\cL^2_x}^{\frac12}\Big), \notag
\end{align*}
which yields
\begin{align}  \label{est-dnbu-end}
&\quad \norm{\nabla^{N} (B \otimes u)}_{H^3_x}+\norm{\nabla^{N} (u \otimes B)}_{H^3_x}\notag\\
&\leq 4 \wt C_1C_*\Big(  s^{-\frac{N}{2\alpha } + \frac{3}{4\alpha }}  \norm{(u_0,B_0)}_{\cH^3_x}^2
  +    s^{-\frac{N}{2\alpha} + \frac{1}{2\alpha}}  \norm{(u_0,B_0)}_{\cH^3_x}^{\frac{11}{6}}  \norm{(u_0,B_0)}_{\cL^2_x}^{\frac 16}\notag\\
&\qquad\qquad\quad + s^{-\frac{N}{2\alpha} + \frac{1}{4 \alpha}} \norm{(u_0,B_0)}_{\cH^3_x}^{\frac{5}{3}}  \norm{(u_0, B_0)}_{\cL^2_x}^{\frac 13}
+  s^{- \frac{N}{2\alpha}} \norm{(u_0,B_0)}_{\cH^3_x}^{\frac 32} \norm{(u_0,B_0)}_{\cL^2_x}^{\frac 12}\Big).
\end{align}
The other terms $ \norm{\nabla^{N} (u \otimes u)}_{H^3_x} $ and $\norm{\nabla^{N} (B \otimes B)}_{ H^3_x}$
can be estimated in the same manner.
Hence,
using Young's inequality, we obtain
\begin{align}   \label{est-dnbu-end*}
&\quad\norm{\nabla^{N} (u \otimes u)}_{\cH^3_x} +\norm{\nabla^{N} (B \otimes B)}_{\cH^3_x}
  +\norm{\nabla^{N} (B \otimes u)}_{\cH^3_x}+ \norm{\nabla^{N} (u \otimes B)}_{\cH^3_x}   \notag\\
&\leq  \wt C_2C_* s^{- \frac{N}{2\alpha}}\norm{(u_0,B_0)}_{\cH^3_x}^{\frac32}
      ( s^{\frac{3}{4\alpha}} \norm{(u_0,B_0)}_{\cH^3_x}^{\frac12}
        +  \norm{(u_0,B_0)}_{\cL^2_x}^{\frac12}).
\end{align}
Thus, we derive from
\eqref{vl2}, \eqref{vh3},  \eqref{esti-GN-H3},
\eqref{est-dnbu} and \eqref{est-dnbu-end*} that
\begin{align*}
  & t^{\frac{N}{2\alpha}} \norm{(\nabla^N u(t), \nabla^N B(t))}_{\cH^3_x}  \notag\\
	\leq& C_*\norm{(u_0,B_0)}_{H^3_x}
	+ \wt C_3 C_* t^{\frac{N}{2\alpha}} \norm{(u_0,B_0)}_{\cH^3_x}^{\frac 32} \norm{(u_0,B_0)}_{\cL^2_x} ^{\frac12} \int_{0}^{\frac{t}{2}}(t-s)^{-\frac{N+1}{2\alpha}}\d s\notag\\
	&+ \wt C_3 C_* t^{\frac{N}{2\alpha}} \norm{(u_0,B_0)}_{\cH^3_x}^{\frac32}
	\int_{\frac{t}{2}}^t (t-s)^{-\frac{1}{2\alpha}}s^{-\frac{N}{2\a}}
	(s^{\frac{3}{4\alpha}}\norm{(u_0,B_0)}_{\cH^3_x}^{\frac12} +  \norm{(u_0,B_0)}_{\cL^2_x}^{\frac12} ) \d s \notag\\
 \leq&  C_*\norm{(u_0,B_0)}_{\cH^3_x}
 (1 + \wt C_4C_* t^{1-\frac{1}{ 2\alpha}}
 \norm{(u_0,B_0)}_{\cH^3_x}^{\frac 32}
 \norm{(u_0,B_0)}_{\cL^2_x}^{\frac12}
 + \wt C_4C_*t^{1+\frac{1}{4\a}}  \norm{(u_0,B_0)}_{\cH^3_x}^{2}),
\end{align*}
which, along with the smallness condition \eqref{assume-t}
with $c$ small enough,
yields that \eqref{claim-dnub} still holds
but with the improved constant $3C_*/2$.
Then, using bootstrap arguments we prove \eqref{claim-dnub},
as claimed.
Finally, the induction arguments give \eqref{pdvh3}
for any $N\geq 3$ and $M=0$.

$(ii.4)$. The case where $N\geq 0$ and $M=1$.
We deduce from \eqref{equa-gMHD}, Lemma \ref{Lem-semi-est}
and estimate \eqref{est-dnbu-end*} that
\begin{align*}
   \|(\partial_t \na^N u, \partial_t \na^N B)\|_{\cH^3}
   \lesssim&  t^{-\frac{N+2\alpha}{2\alpha}} \|(u_0, B_0)\|_{\cH^3}
             + t^{-\frac{N+1}{2\alpha}} \|(u,B)\|_{\cH^3}^{\frac 32}
             (t^{\frac{3}{4\alpha}} (\|(u,B)\|_{\cH^3}^{\frac 32} )^{\frac 12}
              + \|(u,B)\|_{\cL^2}^{\frac 12} ),
\end{align*}
which, along with \eqref{con-pdvh3}, yields 
\begin{align*}
   \|(\partial_t \na^N u, \partial_t \na^N B)\|_{\cH^3}
   \lesssim t^{-\frac{N}{2\alpha}-1}  \|(u_0, B_0)\|_{\cH^3} .
\end{align*}
Therefore, the proof of Proposition \ref{Prop-LWP-Hyper-NLSE} is complete.
\end{proof}

\subsection{Stability estimates} \label{Subsec-Stability}

Next,
we will generate the local solutions $(v_i, D_i)$ to MHD equations
near every $t_i$, $0\leq i\leq m_{q+1}-1$.
Precisely,
for every $0\leq i\leq m_{q+1}-1$,
let $(v_i, D_i)$ solve the following MHD equations
on the small subinterval $[t_i, t_{i+1}+\theta_{q+1}]$,
\begin{equation}\label{equa-mhduq}
	\left\{\aligned
	&\p_tv_i +\nu_1 (-\Delta)^{\alpha} v_i+(v_i\cdot \nabla )v_i-(D_i\cdot \nabla )D_i + \nabla P_i=0,  \\
	&\p_tD_i +\nu_2 (-\Delta)^{\alpha} D_i+(v_i\cdot \nabla )D_i-(D_i\cdot \nabla )v_i =0,  \\
	& \div  v_i = 0,\quad \div  D_i = 0,\\
    & v_i|_{t=t_i}=u_q(t_i),\quad D_i|_{t=t_i}=B_q(t_i).
	\endaligned
	\right.
\end{equation}
Note that,
by the iterative estimate \eqref{ubh3} with $N=0$
and the choice of $m_{q+1}$ and $\theta_{q+1}$ in \eqref{def-mq-thetaq},
the conditions \eqref{t*-t0-v0H3} and \eqref{con-pdvh3} hold
with $t_*, t_0$ replaced by $t_{i+1}+\theta_{q+1}$ and $t_i$,
respectively. Thus, there exists a unique solution $(v_i,D_i)$ to \eqref{equa-mhduq}
on $[t_i, t_{i+1}+\theta_{q+1}]$.

Let $(w_i, H_i):= ( u_q-v_i, B_q-D_i)$ denote the difference between
the local solution $(v_i, D_i)$ and
the old well-prepared solution $(u_q, B_q)$
on $[t_i, t_{i+1}]$.
Then $(w_i,H_i)$ satisfies the linearized equations,
via \eqref{equa-mhdr} and \eqref{equa-mhduq},
\begin{equation}\label{equa-vi}
\left\{\begin{array}{l}
\partial_{t} w_{i}+ \nu_1 (-\Delta)^{\a} w_{i}+\div(v_i\otimes w_i+w_i\otimes u_q-D_i\otimes H_i-H_i\otimes B_q )+\nabla p_i =\div \rr^u_q, \\
\partial_{t} H_{i}+ \nu_2 (-\Delta)^{\a} H_{i}+\div(H_i\otimes u_q+D_i\otimes w_i-w_i\otimes B_q-v_i\otimes H_i  )  =\div \rr^B_q, \\
\div w_{i}=0,\quad \div H_{i}=0 \\
w_{i}|_{t=t_i}=0,\quad D_{i}|_{t=t_i}=0,
\end{array}\right.
\end{equation}
for some pressure $p_i:[t_i,t_{i+1}]\times \T^3\rightarrow \R$.

Proposition \ref{Prop-est-vi} below contains the key stability estimates
for the solutions to \eqref{equa-vi}.

\begin{proposition} [Stability estimates] \label{Prop-est-vi}
Let $\a\in [1,3/2)$, $1<\rho\leq 2$.
Let $(u_q,B_q,\rr^u_q,\rr^B_q)$ be the well-prepared solution to \eqref{equa-mhdr}
at level $q$,
let $(w_i, H_i)$ solve \eqref{equa-vi},
and set $s_{i+1}:= t_{i+1} + \theta_{q+1}$, $0\leq i\leq m_{q+1}-1$.
Then,
the following estimates hold:
\begin{align}
&\|(w_i,H_i)\|_{L^\9([t_i,t_{i+1}+\theta_{q+1}]; \cL^\rho_x)}
   \lesssim \int_{t_i}^{s_{i+1}} \| (\abs{\nabla} \rr^u_q(s),\abs{\nabla} \rr^B_q(s))\|_{\cL^\rho_x}\d s,\label{est-vilp}\\
&\|(w_i,H_i)\|_{L^\9([t_i,t_{i+1}+\theta_{q+1}];\cH^3_x)}
   \lesssim \int_{t_i}^{s_{i+1}} \| (\abs{\nabla} \rr^u_q(s),\abs{\nabla} \rr^B_q(s))\|_{\cH^3_x}\d s,\label{est-vih3}\\
&\|(\mathcal{R}^u w_i,\mathcal{R}^B H_i)\|_{L^\9([t_i,t_{i+1}+\theta_{q+1}];\cL^\rho_x)}
   \lesssim  \int_{t_i}^{s_{i+1}} \| (\rr^u_q(s),\rr^B_q(s))\|_{\cL^\rho_x}\d s,  \label{est-rvi}
\end{align}
where $0\leq i\leq m_{q+1}-1$,
$\mathcal{R}^u$ and $\mathcal{R}^B$ are the inverse divergence operators given by \eqref{calR-def}
in Appendix A,
and the implicit constants depend only on $\alpha$ and $\rho$.
\end{proposition}

\begin{proof}
Without loss of generality, we may consider the case $t_i=0$.
We reformulate \eqref{equa-vi} as 
\begin{subequations}
\begin{align}
 & w_i(t)=\int_{0}^{t} e^{-(t-s)\nu_1(-\Delta)^{\a}} \P_H
        \div (\rr^u_q-v_i\otimes w_i-w_i\otimes u_q+ D_i\otimes H_i+ H_i\otimes B_q  ) \d s,\label{e2.12}\\
  &H_i(t)=\int_{0}^{t} e^{-(t-s)\nu_2(-\Delta)^{\a}} \P_H
         \div (\rr^B_q-H_i\otimes u_q-D_i\otimes w_i+w_i\otimes B_q+v_i\otimes H_i  ) \d s.    \label{ne2.12}
\end{align}
\end{subequations}
Applying  Lemma~\ref{Lem-semi-est} to \eqref{e2.12}-\eqref{ne2.12}
yields
\begin{align}  \label{e2.13}
\|(w_i(t),H_i(t))\|_{\cL^\rho_x}
  & \leq C_* \int_{0}^{t} \|(\abs{\nabla}\rr^u_q ,\abs{\nabla}\rr^B_q) \|_{\cL^\rho_x} \d s
   \notag \\
 &+C_*\int_{0}^{t}(t-s)^{-\frac{1}{2\a}}
    ( \|(u_q, v_i) \|_{\cL^\9_x}  +\|(B_i, D_i) \|_{\cL^\9_x})
     (  \|w_i \|_{L^\rho_x} + \|H_i \|_{L^\rho_x}) \d s
\end{align}
for some universal constant $C_*$ depending only on $\rho$ and $\a$.

We claim that for all $t\in [0,t_1+\theta_{q+1}]$,
\begin{align}\label{e2.14}
 \|w_i(t)\|_{L^\rho_x}+ \|H_i(t)\|_{L^\rho_x}
 \leq 2 C_* \int_{0}^{t} \|(\abs{\nabla}\rr^u_q ,\abs{\nabla}\rr^B_q) \|_{\cL^\rho_x}\d s.
\end{align}

The proof of \eqref{e2.14} is based on bootstrap arguments.
First we note that,
\eqref{e2.14} is valid for $t=0$.
Moreover,
by assuming \eqref{e2.14},
we would prove that the same estimate holds
but  with the constant $2C_*$
replaced by a smaller constant $3C_*/2$.
To see this, plugging \eqref{e2.14} into \eqref{e2.13} we get
\begin{align}   \label{e2.15}
&\quad\ \|(w_i(t),H_i(t))\|_{\cL^\rho_x}\notag\\
& \leq 2 C_*
    \( \frac{1}{2}
     + C_* \int_0^t (t-s)^{-\frac{1}{2\alpha}} (\|(u_q,v_i)\|_{L^\infty_{t}\cL^\infty_x}
     +\|(B_i,D_i)\|_{L^\infty_{t}\cL^\infty_x}) \d s  \)
     \int_{0}^{t} \|(\abs{\nabla}\rr^u_q ,\abs{\nabla}\rr^B_q) \|_{\cL^\rho_x}\d s     \notag\\
&\leq   2 C_*
   \( \frac{1}{2} + 2 C_* t^{1-\frac{1}{2\a}}
  (\|(u_q,v_i)\|_{L^\infty_{t}\cL^\infty_x}+\|(B_i,D_i)\|_{L^\infty_{t}\cL^\infty_x}) \)
    \int_{0}^{t} \|(\abs{\nabla}\rr^u_q ,\abs{\nabla}\rr^B_q) \|_{\cL^\rho_x}\d s.
\end{align}
Thus, using the Sobolev embedding $H^3_x\hookrightarrow L^\9_x$,
\eqref{ubh3}, \eqref{def-mq-thetaq} and \eqref{vh3}
we bound
\begin{align}  \label{e2.15*}
  & 2 C_* (t_1+\theta_{q+1})^{1-\frac{1}{2\a}}
  (\|(u_q,v_i)\|_{L^\infty_{t}\cL^\infty_x}+\|(B_i,D_i)\|_{L^\infty_{t}\cL^\infty_x})  \notag \\
  \leq & 2CC_*(T/m_{q+1})^{1-\frac{1}{2\a}}
        \left( \norm{(u_q,v_i)}_{L^\infty_{t} \cH^3_x}+\|(B_i,D_i)\|_{L^\infty_{t} \cH^3_x}\right)\notag\\
  \leq& 6CC_* (\la^{-12})^{\frac12+\frac{\a-1}{2\a}} \norm{(u_q,B_q)}_{L^\infty_{t} \cH^3_x} \notag\\
  \leq& C'(\la^{-12})^{\frac12+\frac{\a-1}{2\a}}\la^5
  \leq \frac 14,
\end{align}
where $C'$ is some universal constant,
and $a$ is sufficiently large such that $C'\la^{-1}\leq 1/4$.
Hence, we obtain \eqref{e2.14} with the improved constant $3C_*/2$,
which,
via bootstrap arguments,
yields \eqref{e2.14} and so \eqref{est-vilp}.

Estimate \eqref{est-vih3} can be proved in the fashion
as the proof of \eqref{e2.14}.
We claim that
for all $t\in [0,t_1+\theta_{q+1}]$,
\begin{align}\label{claim-vih3}
	\|(w_i(t),H_i(t))\|_{\cH^3_x}
    \leq 2 C_* \int_{0}^{t} \|(\abs{\nabla}\rr^u_q,\abs{\nabla}\rr^B_q )\|_{\cH^3_x}\d s.
\end{align}
To this end,
we apply Lemma \ref{Lem-semi-est} to \eqref{equa-vi} to get
\begin{align}\label{vih3}
	\|(w_i(t),H_i(t))\|_{\cH^3_x}
	& \leq C_* \int_{0}^{t} \|(\abs{\nabla}\rr^u_q(s),\abs{\nabla}\rr^B_q(s))\|_{\cH^3_x}
	\d s   \notag \\
	&\quad +C_* \int_{0}^{t}  (t-s)^{-\frac{1}{2\a}}
         (\|(u_q,v_i)\|_{\cH^3_x}+\|(B_i,D_i)\|_{\cH^3_x})
         \|(w_i,H_i)\|_{\cH^3_x}\d s,
\end{align}
Plugging \eqref{claim-vih3} into \eqref{vih3}
and estimating as in \eqref{e2.15} and \eqref{e2.15*}
we come to
\begin{align}
 \norm{w_i(t),H_i(t)}_{\cH^3_x}
  \leq&  2 C_* \( \frac{1}{2} + 2C_* t^{1-\frac{1}{2\a}}
      (\|(u_q,v_i)\|_{L^\infty_{t} \cH^3_x}+\|(B_i,D_i)\|_{L^\infty_{t} \cH^3_x}) \)  \notag \\
 & \times \int_{0}^{t} \|(\abs{\nabla}\rr^u_q(s),\abs{\nabla}\rr^B_q(s))\|_{\cH^3_x}\d s
   \leq \frac 32 C_*.
\end{align}
This gives \eqref{claim-vih3} with the improved constant $3C_*/2$.
Hence,
using bootstrap arguments  once more
we prove  \eqref{claim-vih3} and so \eqref{est-vih3}
for all $t\in [0,t_1+\theta_{q+1}]$.

It remains to prove \eqref{est-rvi}.
Let
\begin{align}  \label{z-wi}
  y_i: = \Delta^{-1}\curl w_i, \quad  J_i: = \Delta^{-1}\curl H_i.
\end{align}
Note that $\curl y_i =-w_i$ and $\curl J_i=-H_i$,
as $w_i$ and $H_i$ are divergence free.
By the boundedness of Calder\'{o}n-Zygmund operators
$\mathcal{R}^u \curl$ and
$\mathcal{R}^B \curl$ in $L^\rho_x$,
$1<\rho<2$,
we have, for any $t\in [t_i,t_{i+1}]$,
\begin{align} \label{Ruw-RBH}
   \|\mathcal{R}^uw_i\|_{L^\rho_x}\leq C \|y_i\|_{L^\rho_x},
    \quad \|\mathcal{R}^B H_i\|_{L^\rho_x}\leq C \|J_i\|_{L^\rho_x},
\end{align}
where $C$ is a universal constant.

Moreover,
straightforward computations show that $y_1$ satisfies the equation
\begin{align}
\partial_t y_i + u_q \cdot \nabla y_i -B_q\cdot \nabla J_i  + \nu_1 (- \Delta)^\alpha  y_i
&= \Delta^{-1} \curl \div \rr^u_q +  \Delta^{-1} \curl \div \left( ( (y_i \times \nabla) v_i )^T \right)\notag\\
&\quad +  \Delta^{-1} \curl \div  \left( ( y_i \times \nabla) u_q\right) + \Delta^{-1} \nabla\div  \left( ( y_i \cdot \nabla) u_q\right) \notag\\
&\quad - \Delta^{-1} \curl \div \left( ( (J_i\times \nabla) D_i )^T \right)- \Delta^{-1} \curl \div  \left( (J_i \cdot \nabla) B_q\right)\notag\\
&\quad - \Delta^{-1} \nabla \div \left( (J_i\cdot \nabla) B_q \right),   \label{equa-y1}
\end{align}
and $J_i$ satisfies
\begin{align}
	\partial_t J_i + u_q \cdot \nabla J_i-B_q\cdot \nabla y_i +  \nu_2 (- \Delta)^{\alpha}  J_i
	&= \Delta^{-1} \curl \div \rr^B_q +  \Delta^{-1} \curl \div \left( ( (y_i\times \nabla) D_i )^T \right)\notag\\
	&\quad +  \Delta^{-1} \curl \div  \left( ( J_i \times \nabla) u_q\right) + \Delta^{-1} \nabla\div  \left( ( J_i \cdot \nabla) u_q\right) \notag\\
	&\quad - \Delta^{-1} \curl \div \left( ( (y_i\times \nabla) D_i )^T \right)- \Delta^{-1} \curl \div  \left( ( y_i \cdot \nabla) B_q\right)\notag\\
	&\quad - \Delta^{-1} \nabla \div \left( (y_i\cdot \nabla) B_q \right).  \label{equa-y2}
\end{align}

Then, we can apply Lemma~\ref{Lem-semi-est}
and use the boundedness of Calder\'{o}n-Zygmund operators
$\Delta^{-1} \curl \div $
and $\Delta^{-1} \nabla \div$ in $L^\rho_x$
to derive
\begin{align}
 \norm{(y_i(t),J_i(t))}_{\cL^\rho_x}
 &\leq C_* \int_0^t \norm{\rr^u_q(s),\rr^B_q(s)}_{\cL^\rho_x}\d s \notag\\
&\quad + C_* \norm{(u_q,B_q)}_{L^\infty_{t}\cL^\infty_x}
      \int_0^t (t-s)^{-\frac{1}{2\alpha}}
      \norm{(y_i(s),J_i(s))}_{\cL^\rho_x} \d s  \notag \\
&\quad  +C_* (  \norm{(\nabla u_q,\nabla v_i)}_{L^\infty_{t}\cL^\infty_x}
       +\norm{(\nabla B_q,\nabla D_i)}_{L^\infty_{t}\cL^\infty_x} )
      \int_0^t \norm{(y_i(s),J_i(s))}_{\cL^\rho_x}\d s,
\label{e2.20}
\end{align}
where $C_*$ is a universal constant depending only on $\rho$ and $\a$.

We claim that for any $t\in [0,t_1+\theta_{q+1}]$,
\begin{align}    \label{e2.21}
 \norm{(y_i(t),J_i(t))}_{\cL^\rho_x}
&\leq 2C_* \int_0^t \norm{(\rr^u_q(s),\rr^B_q(s))}_{\cL^\rho_x}\d s.
\end{align}

In order to prove \eqref{e2.21},
inserting \eqref{e2.21} into \eqref{e2.20}, we get
\begin{align*}
 \norm{(y_i(t),J_i(t))}_{\cL^\rho_x}
&\leq 2C_* \int_0^t \norm{(\rr^u_q(s),\rr^B_q(s))}_{\cL^\rho_x}\d s \notag\\
&\quad \times \( \frac{1}{2}
    +  2  C_* t^{1-\frac{1}{2\a}}
    \norm{(u_q,B_q)}_{L^\infty_{t}\cL^\infty_x}
    + C_*t\left(  \norm{(\nabla u_q,\nabla v_i)}_{L^\infty_{t}\cL^\infty_x}
    +\norm{(\nabla B_q, \nabla D_i)}_{L^\infty_{t}\cL^\infty_x}  \right)
    \) .
\end{align*}
Then, by the Sobolev embedding $H^2_x \hookrightarrow L^\infty_x$,
\eqref{ubh3}, \eqref{def-mq-thetaq} and \eqref{vh3},
\begin{align}\label{ne2.23}
&  2 C_*t^{1-\frac{1}{2\a}}
   + C_*t\left(  \norm{(\nabla u_q, \nabla v_i)}_{L^\infty_{t}\cL^\infty_x}
   +\norm{(\nabla B_q,\nabla D_i)}_{L^\infty_{t}\cL^\infty_x}  \right)
\norm{(u_q,B_q)}_{L^\infty_{t}\cL^\infty_x}\notag\\
\leq &\, 2C^*t^{1-\frac{1}{2\a}}\norm{(u_q,B_q)}_{L^\infty_{t} \cH^3_x}
         + C_*t\left( \norm{(u_q,v_i)}_{L^\infty_{t}\cH^3_x}
     +\norm{(B_q, D_i)}_{L^\infty_{t} \cH^3_x}  \right) \notag\\
\leq &\, CC_*(3(t_1+\theta_{q+1})+2
(t_1+\theta_{q+1})^{1-\frac{1}{2\a}})\norm{(u_q,B_q)}_{L^\infty_{t} \cH^3_x}\notag\\
 \leq &\, C'\(\la^{-12}+(\la^{-12})^{\frac12+\frac{2\a-1}{2\a}}\) \la^{5}
 \leq \frac 14
\end{align}
for $a$ sufficiently large.
It follows that
the constant in \eqref{e2.21} can be improved by $ 3C_*/2$.
This yields \eqref{e2.21} for all $t\in [0,t_1+\theta_{q+1}]$
by bootstrap arguments, as claimed.
In view of \eqref{Ruw-RBH},
we consequently prove \eqref{est-rvi}.
Therefore, the proof is complete.
\end{proof}

\subsection{Temporal gluing of local solutions}
From the previous section,
we see that,
for every $0\leq i\leq \mq-1$,
$(v_i,D_i)$ is exactly the solution to MHD equations \eqref{equa-gMHD}
on the small subinterval $[t_i,t_{i+1}+\theta_{q+1}]$,
and its difference from the old well-prepared solution $(u_q, B_q)$
can be controlled by using the stability estimates in Proposition \ref{Prop-est-vi}.

In this subsection,
we will glue all these local solutions together
by using a partition of unity $\{\chi_i\}$,
such that $(\sum_i \chi_i v_i,\sum_i \chi_i D_i)$ solves equations \eqref{equa-gMHD}
in a majority part of the time interval $[0,T]$.
The important outcome is that,
the new Reynolds and magnetic stresses would have disjoint temporal supports
with smaller length,
and, simultaneously, maintain the decay of amplitudes.

To be precise, let $\{\chi_i\}_{i=0}^{\mq-1}$ be a $C_0^\9$ partition of unity on $[0,T]$  such that
$$
0\leq  \chi_i(t) \leq 1, \quad\text{for} \quad t\in[0,T],
$$
for $0<i< \mq-1$,
\begin{align}\label{def-chi1}
\chi_{i}= \begin{cases}1 & \text { if } t_{i}+\thq \leq t \leq t_{i+1}, \\
 0 & \text { if } t\leq t_{i}, \text { or } t \geq t_{i+1}+\thq,\end{cases}
\end{align}
and for $i=0$,
\begin{align}\label{def-chi2}
\chi_{0}= \begin{cases}1 & \text { if } 0 \leq t \leq t_{i+1},
\\ 0 & \text { if } t \geq t_{i+1}+\thq,\end{cases}
\end{align}
and for $i= \mq-1$,
\begin{align}\label{def-chi3}
\chi_{\mq-1}= \begin{cases}1 & \text { if } t_{i}+\thq \leq t \leq T, \\
  0 & \text { if } t \leq t_{i}.\end{cases}
\end{align}
Furthermore,
\begin{align}\label{est-chi}
\|\partial_t^M \chi_i\|_{L^{\9}_t}\lesssim \thq^{-M},\ \ 0\leq i\leq \mq-1,
\end{align}
where the implicit constant is independent of $\thq$, $i$ and $M\geq 0$.

Then, we define the gluing solutions by
\begin{align}\label{def-wtu}
\wt u_q:=\sum_{i=0}^{\mq-1} \chi_i v_i,
\quad \wt B_q:=\sum_{i=0}^{\mq-1} \chi_i D_i.
\end{align}
Note that, $\wt u_q, \wt B_q: [0,T]\times \T^3\rightarrow \R^3$ are divergence free and mean free.
Moreover, we have
\begin{align*}
\wt u_q=(1-\chi_i)v_{i-1}+\chi_i v_i,
\quad \wt B_q=(1-\chi_i)D_{i-1}+\chi_i D_i,\ \ t\in [t_i, t_{i+1}],
\end{align*}
and the glued solutions $(\wt u_q,\wt B_q)$ satisfy the
following equations on $[t_i,t_{i+1}]$:
\begin{align}\label{equa-wtu}
\begin{cases}
&\partial_{t} \wt u_q+ \nu (-\Delta)^{\a} \wt u_q +\operatorname{div}(\wt u_q \otimes \wt u_q-\wt B_q\otimes \wt B_q) +\nabla \wt p=\div \tr,\\
&\partial_{t} \wt B_q+ \nu (-\Delta)^{\a} \wt B_q +\operatorname{div}(\wt B_q \otimes \wt u_q-\wt u_q\otimes \wt B_q) =\div \trb,
\end{cases}
\end{align}
for some pressure term $\wt p$
and the new stresses
\begin{subequations}  \label{def-nr}
\begin{align}
&\tr =\partial_t\chi_i\mathcal{R}^u(v_{i}-v_{i-1}) -\chi_{i}(1-\chi_{i})
    ((v_{i}-v_{i-1})\mathring\otimes (v_{i}-v_{i-1})-(D_i-D_{i-1})\mathring\otimes(D_i-D_{i-1})),\\
&\trb =\partial_t\chi_i\mathcal{R}^B(D_{i}-D_{i-1}) -\chi_{i}(1-\chi_{i})
((D_{i}-D_{i-1})\otimes (v_{i}-v_{i-1})-(v_i-v_{i-1})\otimes(D_i-D_{i-1})).
\end{align}
\end{subequations}

\subsection{Proof of well-preparedness in Theorem~\ref{Thm-nunr}}
Using the finite overlaps of the supports of $\{\chi_i\}$, \eqref{ubh3} and \eqref{vh3}, we get
\begin{align}\label{ver-nuh3}
\|\wt u_q\|_{L^\9_tH^3_x}& \leq \|\sum_i \chi_i v_i\|_{L^\9_tH^3_x}\notag\\
&\leq \sup_i \(\|(1-\chi_i)v_{i-1}\|_{L^\9_t(\supp(\chi_i\chi_{i-1});H^3_x)}+\|\chi_i v_i\|_{L^\9_t(\supp(\chi_i);H^3_x)}\)\notag\\
&\lesssim \|u_q\|_{L^\9_tH^3_x}\lesssim \la^{5},
\end{align}
and similarly,
\begin{align}\label{ver-nbh3}
	\|\wt B_q\|_{L^\9_tH^3_x}
    \leq \sup_i \(\|(1-\chi_i)D_{i-1}\|_{L^\9_t(\supp(\chi_i\chi_{i-1});H^3_x)}
    +\|\chi_i D_i\|_{L^\9_t(\supp(\chi_i);H^3_x)}\)\lesssim \la^{5},
\end{align}
which yields \eqref{nuh3}.

Regarding estimate \eqref{uuql2},
using \eqref{def-mq-thetaq}, \eqref{rhN} and \eqref{est-vilp},
we get
\begin{align}\label{ne2.42}
&\quad\ \|(\wt u_q-u_q, \wt B_q-B_q)\|_{L^\9_t \cL^2_{x}} \notag\\
& \leq \|(\sum_i \chi_i(v_i-u_q), \sum_i \chi_i(D_i-B_q))\|_{L^\9_t \cL^2_{x}}  \notag\\
&\leq \sup_i\(\|(v_i-u_q, D_i-B_i)\|_{L^\9_t(\supp(\chi_i); \cL^2_x)}
    +\|(v_{i-1}-u_q, D_{i-1}-B_q)\|_{L^\9_t(\supp(\chi_i\chi_{i-1}); \cL^2_x)}\)\notag\\
&\lesssim \mq^{-1} \| (\rr^u_q, \rr^B_q)\|_{L^\9_t \cH^3_x}
 \lesssim \mq^{-1}\la^9 \lesssim \la^{-3}.
\end{align}
Hence, estimate \eqref{uuql2} is verified.

Concerning the $L^1_{t,x}$-estimate of the new Reynolds stress $\rr^u_q$,
the expression \eqref{def-nr} yields that
\begin{align}
\|\tr\|_{L_{t,x}^{1}} & \leq\|\partial_t\chi_i\mathcal{R}(v_{i}-v_{i-1})\|_{L_{t,x}^1 }\notag\\
&\quad\ \   +\|\chi_{i}(1-\chi_{i})((v_{i}-v_{i-1})\mathring\otimes (v_{i}-v_{i-1})
    -(D_{i}-D_{i-1})\mathring\otimes (D_{i}-D_{i-1}))\|_{L_{t,x}^{1}} \notag\\
&=:K_1+K_2.\label{e2.33}
\end{align}

In order to estimate the right-hand-side above,
for the first term $K_1$, we choose
\begin{align}   \label{rho-ve}
1<\rho<\frac{4\ve_R+36+8\beta b}{\ve_R+36+8\beta b},
\end{align}
and use H\"older's inequality to obtain
\begin{align}
 K_1 \leq&  \sum_i \|\p_t \chi_i\|_{L^1_t}\|\mathcal{R}(v_i-v_{i-1})\|_{L^\9_tL^\rho_x}  \notag \\
 \lesssim& \sum_i \|\mathcal{R}(w_i-w_{i-1})\|_{L^\9_t(\supp(\chi_i\chi_{i-1});L^\rho_x)}.
\end{align}
Note that, the uniform bound $\|\p_t \chi_i\|_{L^1_t} \lesssim 1$
was used in the last step.
Then, by the Gagliardo-Nirenberg inequality,
\eqref{rhN}, \eqref{rl1}, \eqref{est-rvi} and \eqref{rho-ve},
\begin{align}   \label{est-j1}
    K_1 \lesssim& \sum_i\int_{t_i}^{t_{i+1}+\theta_{q+1}}
        \|(\rr^u_q(s),\rr^B_q(s))\|_{\cL^\rho_x}\d s \notag\\
	\lesssim&  \|(\rr^u_q(s),\rr^B_q(s))\|_{L^1_{t}\cL^1_x}^{1-\frac{2(\rho-1)}{3\rho}}
              \|(\rr^u_q(s),\rr^B_q(s))\|_{L^\9_t\cH^3_x}^{\frac{2(\rho-1)}{3\rho}} \notag \\
    \lesssim&  \delta_{q+1}^{1-\frac{2(\rho-1)}{3\rho}}  \la^{-\ve_R+(\ve_R+9)\frac{2(\rho-1)}{3\rho}} \notag \\
    \lesssim& \lambda_q^{-\frac{\ve_R}{2}} \delta_{q+1}
    \leq \la^{-\frac{3\ve_R}{8}}\delta_{q+1},
\end{align}
where in the last step we  used $\la^{-\ve_R/8}$
(or $a$ sufficiently large) to absorb the implicit constant.

For the second term $K_2$ of \eqref{e2.33},
we use \eqref{est-vilp} to estimate
\begin{align}
 K_2  &\lesssim \sum_i |\supp_{t} \chi_i(1-\chi_i)|\|\chi_{i}-\chi_{i}^{2} \|_{L^\9_t}
  \(\|(w_{i}-w_{i-1},  H_{i}- H_{i-1})\|_{L^\9_t(\supp(\chi_i\chi_{i-1});\cL^2_x)}^2 \)\notag \\
  & \lesssim \sum_i \thq \(\int_{t_i}^{t_{i+1}}\|(|\nabla|\rr^u_q(s),|\nabla|\rr^B_q(s))\|_{\cL^2_x}  \d s\)^2  \notag\\
  &\lesssim \thq  \|(|\nabla|\rr^u_q,|\nabla|\rr^B_q)\|_{L^1_t \cL^2_x}^2.
\end{align}
Using the interpolation estimate, \eqref{rhN} with $N=3$, \eqref{rl1}
and the facts that $\la^{-1}\ll \delta_{q+1}^{1/9}$ and $0<\eta<1$,
we obtain
\begin{align}  \label{est-j2}
	K_2 \lesssim& \mq^{-\frac{1}{\eta}}
 \|(\rr^u_q(s),\rr^B_q(s))\|_{L^1_{t}\cL^1_x}^{\frac{8}{9}}
\|(\rr^u_q(s),\rr^B_q(s))\|_{L^1_t\cH^3_x}^{\frac{10}{9}} \notag\\
         \lesssim& \la^{-12/\eta}\la^{-\frac89\ve_R}\delta_{q+1}^{\frac89} \la^{10}\leq \la^{-\frac12\ve_R}\delta_{q+1},
\end{align}
where we also chose $a$ sufficient large and used $\la^{-\ve_R/4}$ to absorb the implicit constant.

Thus, it follows from \eqref{e2.33}, \eqref{est-j1} and \eqref{est-j2} that
\begin{align}\label{est-trq}
\|\tr\|_{L^1_{t,x}}\leq \la^{-\frac{\ve_R}{4}}\delta_{q+1}.
\end{align}
The $L^1_{t,x}$-estimate of the magnetic stress $\trb$
can be obtained by using analogous arguments.
Hence, \eqref{nrl1} is verified.

For estimate \eqref{nrh3},
we infer from \eqref{def-nr} that,
for any $t\in[t_i,t_{i+1}]$, $0\leq i\leq \mq-1$,
\begin{align}\label{e2.45}
\|\partial_t^M \nabla^N \tr\|_{L^\infty_tH^3_x}
\leq& \| \partial_t^M \nabla^N(\partial_t\chi_i\mathcal{R}(v_{i}-v_{i-1}))\|_{L^\infty_tH^3_x }   \notag \\
&  +\|\partial_t^M \nabla^N(\chi_{i}(1-\chi_{i})((v_{i}-v_{i-1}) \mathring\otimes (v_{i}-v_{i-1})
   - (D_{i}-D_{i-1})\mathring\otimes (D_{i}-D_{i-1})))\|_{L^\infty_tH^3_x} \notag \\
 =:&\wt K_1 + \wt K_2.
\end{align}
Note that,
by \eqref{ubh3} with $N=3$, \eqref{pdvh3}
and \eqref{est-chi},
\begin{align}\label{e2.46}
 \wt K_1\lesssim &\sum_{M_1+M_2=M}\norm{\partial_t^{M_1+1} \chi_i}_{L^\infty_t} \left(\norm{\partial_t^{M_2} \nabla^N v_i}_{L^\infty_t(\supp(\chi_i);H^3_x)} + \norm{\partial_t^{M_2} \nabla^N v_{i-1}}_{L^\infty_t(\supp(\chi_{i-1});H^3_x)}\right) \notag\\
\lesssim& \sum_{M_1+M_2=M} \thq^{-M_1-1}\mq^{\frac{N}{2\alpha} + M_2} \lambda_q^{5} \lesssim \thq^{-M-1} \mq^{\frac{N}{2\alpha} } \lambda_q^{5} \,,
\end{align}
where the last step is due to $m_{q+1} \leq \theta_{q+1}^{-1}$
and the implicit constants are independent of $i$ and $q$.
Moreover, by \eqref{ubh3}, \eqref{vh3} and \eqref{pdvh3},
\begin{align}\label{e2.47}
 \wt K_2
 \lesssim & \sum_{M_1+M_2=M}  \(\norm{\partial_t^{M_1}( \chi_i (1-\chi_i)}_{L^\9_t}\norm{\partial_t^{M_2} \nabla^N((v_i-v_{i-1}) \mathring\otimes(v_i-v_{i-1}))}_{L^\infty_t(\supp(\chi_i\chi_{i-1});H^3_x)}\right. \notag\\
&\qquad\qquad\quad \left.+ \norm{\partial_t^{M_1}( \chi_i (1-\chi_i)}_{L^\9_t}\norm{\partial_t^{M_2} \nabla^N((D_i-D_{i-1}) \mathring\otimes(D_i-D_{i-1}))}_{L^\infty_t(\supp(\chi_i\chi_{i-1});H^3_x)}\) \notag\\
\lesssim & \sum_{M_1+M_2=M}  \thq^{-M_1}\Big( \norm{\partial_t^{M_2} \nabla^N ( (v_i-v_{i-1}) \otimes(v_i-v_{i-1}))}_{L^\infty_t(\supp(\chi_i\chi_{i-1});H^3_x)} \notag\\
&\qquad\qquad\qquad\qquad + \norm{\partial_t^{M_2} \nabla^N((D_i-D_{i-1}) \mathring\otimes(D_i-D_{i-1}))}_{L^\infty_t(\supp(\chi_i\chi_{i-1});H^3_x)}\Big) \notag\\
 \lesssim & \sum_{M_1+M_2=M}  \thq^{-M_1}\mq^{\frac{N}{2\alpha}+M_2 } \lambda_q^{10}
 \lesssim \thq^{-M-1} \mq^{\frac{N}{2\alpha} }\lambda_q^{5} ,
\end{align}
where the last step is due to
$\theta^{-1}_{q+1} \geq m_{q+1} \geq \lbb_q^5$,
and the implicit constants are independent of $i$ and $q$.
Thus,
it follows from \eqref{e2.45}, \eqref{e2.46} and \eqref{e2.47} that
\begin{align}
\|\partial_t^M \nabla^N \tr\|_{L^\infty_tH^3_x}
  \lesssim \theta_{q+1}^{-M-1} m_{q+1}^{\frac{N}{2\alpha}} \lbb_q^5
 \lesssim \thq^{-M-N-1} \lambda_q^{5}.
\end{align}
Similar arguments also yield the same upper bound
of the magnetic stress $\trb$.
Thus, \eqref{nrh3} is verified.

Finally, we are left to prove the remaining well-preparedness of
$(\wt u_q,\wt B_q,\tr,\trb)$,
\eqref{suppnr} and \eqref{suppwtuq},
which can be proved by using analogous arguments in \cite{lqzz22}.
For the reader's convenience,
we sketch the main arguments below.

Let
\begin{align}\label{def-indexsetb}
\mathcal{C}:=\left\{ i\in \mathbb{Z}: 1\leq i\leq m_{q+1}-1\  \text{and}\ ( \rr^u_q,\rr^B_q)\not\equiv 0\ \text{on}\ [t_{i-1},t_{i}+\thq]\cap [0,T] \right\},
\end{align}
and
\begin{align}   \label{Iq1-C-def}
  I_{q+1} := \bigcup_{i\in \mathcal{C}} \left[t_i-2\thq,t_i+3\thq \right].
\end{align}
Such constructions guarantee that, for any $q\geq 0$,
\begin{align}\label{iq1}
I_{q+1}\subseteq I_q.
\end{align}

In order to prove  \eqref{suppnr}
when $ \operatorname{dist}(t,I_{q+1}^c)\leq {3}\thq/2$,
we see that
if  $t\in I_{q+1}$, then by \eqref{Iq1-C-def},
we have for some $i\in \mathcal{C}$,
\begin{align}\label{e3.57}
t\in [t_i-2\thq,t_i-\frac{\thq}{2}]\quad \text{or}\quad t\in [t_i+\frac{3\thq}{2},t_i+3\thq].
\end{align}
But in both cases,
$\partial_t \chi_j$ and $1-\chi_j$ vanish,
$j=i-1$ or $i$.
In view of \eqref{def-nr},
\eqref{suppnr} then follows.

Moreover, if $t\in I_{q+1}^c$, then
$t\in [t_j,t_{j+1}]$ for some $0\leq j\leq \mq-1$.
If $t\in [t_j+\thq,t_{j+1}]$,
reasoning as above yields \eqref{suppnr}.
If $t\in [t_j,t_{j}+\thq]$,
then $j\notin \mathcal{C}$,
otherwise $t\in [t_j-2\thq,t_j+3\thq]\subseteq I_{q+1}$.
Hence, by the definition of $\mathcal{C}$ in \eqref{def-indexsetb},
\begin{align*}
(\rr^u_q,\rr^B_q)\equiv0 \quad \text{on}\quad [t_{j-1},t_j+\thq],
\end{align*}
and so $(u_q,B_q)$ solves equations \eqref{equa-gMHD} on $[t_{j-1},t_j+\thq]$.
But since $(v_{j-1}, D_{j-1})$
also solves \eqref{equa-gMHD}
with the same condition $(u_q(t_{j-1}),B_q(t_{j-1}))$ at $t_{j-1}$,
so does $(v_j, D_j)$ on $[t_j, t_j + \theta_{q+1}]$.
Thus,
in the overlapped regime $[t_j, t_j+\theta_{q+1}]$,
the uniqueness in Proposition  \ref{Prop-LWP-Hyper-NLSE} then yields
\begin{align*}
     (v_{j-1}, D_{j-1}) = (v_j, D_j) = (u_q, B_q)\ \  on\ [t_j, t_j+\theta_{q+1}].
\end{align*}
Plugging this into \eqref{def-nr} with $j$ replacing $i$
we thus obtain \eqref{suppnr}.

Regarding \eqref{suppwtuq},
we take any $t\in [0,T]$ such that $(\wt u_q(t),\wt B_q(t))\neq 0$.
Then, $t\in [t_i,t_{i+1}]$ for some $0\leq i\leq m_{q+1}-1$.
If $t\in [t_i+\thq,t_{i+1}]$, then we have $(u_q(t_i),B_q(t_i))\neq 0$,
otherwise, by \eqref{equa-mhduq},
$(\wt u_q(t),\wt B_q(t))=(v_i(t),H_i(t))=0$.
Since $|t-t_i|\leq T/\mq$,
we obtain \eqref{suppwtuq}.
Moreover, if $t\in [t_i, t_i+\thq]$,
we have $(u_q(t_i),B_q(t_i))\neq 0$ or $(u_q(t_{i-1}),B_q(t_{i-1}))\neq 0$.
Otherwise by the uniqueness of \eqref{equa-mhduq},
$(v_i,D_i) = (v_{i-1}, D_{i-1}) = 0$,
which by \eqref{def-wtu} leads to
$(\wt u_q(t), \wt B_q(t)) = 0$,
a contradiction.
Hence,
we get  $|t-t_i|\leq T/\mq$ and $|t-t_{i-1}|\leq 2(T/\mq)$,
and so \eqref{suppwtuq} follows.

Therefore, the proof of Theorem~\ref{Thm-nunr} is complete.
\hfill $\square$

\section{Velocity and magnetic perturbations} \label{Sec-perturb-S1}

We start to treat the supercritical regime $\mathcal{S}_1$,
whose borderline in particular includes the endpoint $(s,\gamma,p)=(2\alpha/\gamma+1-2\alpha, \gamma, \infty)$.

The aim of this section is to first construct suitable
velocity and magnetic perturbations
and then to verify the inductive estimates
\eqref{ubh3}, \eqref{ubpth2},
\eqref{u-B-L2tx-conv}-\eqref{u-B-Lw-conv} for the new velocity and magnetic fields
at level $q+1$.
For this purpose,
let us first introduce the appropriate intermittent velocity
and magnetic flows adapted to the supercritical regime $\mathcal{S}_1$,
which in particular feature both the spatial and temporal intermittency.

\subsection{Intermittent velocity and magnetic flows} \label{Subsec-Flow-S1}

The intermittent flows mainly contain the
spatial building blocks and the temporal building blocks,
indexed by four parameters $\rs$, $\lambda$, $\tau$ and $\sigma$:
\begin{equation}\label{larsrp-endpt2}
	\rs := \lambda_{q+1}^{2-2\a-10\varepsilon},\
	\lambda := \lambda_{q+1},\   \tau:=\lambda_{q+1}^{2\a}, \ \sigma:=\lambda_{q+1}^{2\varepsilon},
\end{equation}
where $\varepsilon$ is given by \eqref{ne3.1}.

{\bf $\bullet$ Spatial building blocks.}
The appropriate spatial building blocks in the supercritical regime $\mathcal{S}_1$
would be the intermittent shear flows in \cite{bbv20}.
In the following we mainly recall
from \cite{bbv20} their
constructions and properties.

Let $\Phi : \mathbb{R} \to \mathbb{R}$ be a smooth cut-off function supported on
the interval $[-1,1]$
and normalize $\Phi$ such that $\phi := - \frac{d^2}{dx^2}\Phi$ satisfies
\begin{equation}\label{e4.91}
	\frac{1}{2 \pi}\int_{\mathbb{R}} \phi^2(x)\d x = 1.
\end{equation}
The corresponding rescaled cut-off functions are defined by
\begin{equation*}
	\phi_{\rs}(x) := {\rs^{-\frac{1}{2}}}\phi\left(\frac{x}{\rs}\right), \quad
	\Phi_{\rs}(x):=   {\rs^{-\frac{1}{2}}} \Phi\left(\frac{x}{\rs}\right).
\end{equation*}
Note that,  $\phi_{\rs}$ is supported in the ball of radius $\rs$, in $\bbr$.
By an abuse of notation,
we periodize $\phi_{\rs}$ and $\Phi_{\rs}$ so that
they are treated as periodic functions defined on $\mathbb{T}$.

The \textit{intermittent velocity shear flows} in \cite{bbv20} are defined by
\begin{equation*}
	W_{(k)} := \phi_{\rs}( \lambda \rs N_{\Lambda}k\cdot x)k_1,\ \  k \in \Lambda_u \cup \Lambda_B  ,
\end{equation*}
and the \textit{intermittent magnetic shear flows} are defined by
\begin{equation*}
	D_{(k)} := \phi_{\rs}( \lambda \rs N_{\Lambda}k\cdot x)k_2, \ \ k \in \Lambda_B .
\end{equation*}
Here, $N_{\Lambda}$ is given by \eqref{NLambda},
$(k,k_1,k_2)$ are the orthonormal bases in $\R^3$ in
Geometric Lemmas~\ref{geometric lem 1} and \ref{geometric lem 2},
$\Lambda_u, \Lambda_B$ are the wave vector sets of finite cardinality,
and the parameter $\rs$  parameterizes the concentration of the flows.
In particular,
$\{W_{(k)}, D_{(k)}\}$ are $(\mathbb{T}/(\lbb \rs))^3$-periodic,
supported on thin plane with thickness $\sim {1}/{\lbb}$ in each periodic domain.
See \cite{bbv20}.

For brevity of notations, we set
\begin{align}\label{snp-endpt2}
	\phi_{(k)}(x) := \phi_{\rs}(\lambda \rs N_{\Lambda}k\cdot x),  \ \
	\Phi_{(k)}(x) := \Phi_{\rs}(\lambda \rs N_{\Lambda}k\cdot x),
\end{align}
and rewrite
\begin{subequations}
	\begin{align}
		&	W_{(k)} = \phi_{(k)} k_1,\quad  k\in \Lambda_u\cup \Lambda_B, \label{snwd-endpt2}\\
		&	D_{(k)} = \phi_{(k)} k_2,\quad  k\in \Lambda_B.  \label{snd-endpt2}
	\end{align}
\end{subequations}
Then,
$W_{(k)}$ and $D_{(k)}$ are  mean zero on $\T^3$.
Moreover,
the corresponding potentials are defined by
\begin{subequations} \label{corrector vector-endpt2}
	\begin{align}
		& W_{(k)}^c := \frac{1}{\lambda^2N_{ \Lambda }^2}\Phi_{(k)} k_1,  \ \ k\in \Lambda_u\cup \Lambda_B, \\
		& D_{(k)}^c := \frac{1}{\lambda^2N_{ \Lambda }^2}\Phi_{(k)} k_2,\ \ k\in \Lambda_B.
	\end{align}
\end{subequations}

The intermittent flows
$W_{(k)}$ and $D_{(k)}$ will be used in the construction
of the crucial principal parts of the perturbations
$w^{(p)}_{q+1}$ and $d^{(p)}_{q+1}$ in \eqref{wp-def-A2} and \eqref{dp-def-A2} below,
and the corresponding potentials
will be used in the construction of the incompressible correctors
$w^{(c)}_{q+1}$ and $d^{(c)}_{q+1}$ in \eqref{wqc-endpt2} and \eqref{dqc-endpt2}.

Lemma \ref{buildingblockestlemma-endpt2} below contains their
crucial analytic estimates,
which in particular provide 1D spatial intermittency.

\begin{lemma} [\cite{bbv20} Estimates of intermittent shear flows] \label{buildingblockestlemma-endpt2}
	For any $p \in [1,\infty]$ and $N \in \mathbb{N}$, we have
	\begin{align}
		&\left\|\nabla^{N} \phi_{(k)}\right\|_{L^{p}_{x}}+\left\|\nabla^{N} \Phi_{(k)}\right\|_{L^{p}_{x}}
		\lesssim r_{\perp}^{\frac 1p- \frac12}  \lambda^{N}.  \label{intermittent estimates2-endpt2}
	\end{align}
	In particular,
	\begin{align}
		&\displaystyle \|\nabla^{N}  W_{(k)}\|_{C_t  L^{p}_{x}}
		+\lambda^{2} \|\nabla^{N}  W_{(k)}^c\|_{C_t L^{p}_{x}}\lesssim r_{\perp}^{\frac 1p- \frac12} \lambda^{N},
         \ \ k\in \Lambda_u \cup \Lambda_B, \label{ew-endpt2} \\
         &\displaystyle \|\nabla^{N}  D_{(k)}\|_{C_t  L^{p}_{x}}
         +\lambda^{2} \|\nabla^{N}  D_{(k)}^c\|_{C_t L^{p}_{x}}\lesssim r_{\perp}^{\frac 1p- \frac12} \lambda^{N},
         \ \ k\in \Lambda_B. \label{ed-endpt2}
	\end{align}
	Moreover, for every $k \neq k'\in \Lambda_{u}\cup \Lambda_{B}$,
$N\in \mathbb{N}$ and $p \in [1, \infty]$,
	the following product estimate holds
	\begin{equation}  \label{intersect-phik1}
		\|\na^N (\phi_{(k)}\phi_{(k')}) \|_{C_tL^p_x}\lesssim \lbb^N \rs^{\frac{2}{p}-1}.
	\end{equation}
	The implicit constants above are independent of the parameters $\rs$ and $\lambda$.
\end{lemma}

{\bf $\bullet$ Temporal building blocks.}
The intermittent shear flows provide 1D spatial intermittency
and so permit to control the velocity and resistivity
$(-\Delta)^{\alpha}$ with $\alpha \in [0,1/2)$.
In order to control the stronger velocity and resistivity,
particularly, when $\alpha$ is beyond the Lions exponent $5/4$,
it is crucial to introduce the  temporal intermittency in the building blocks.

We adopt the notations as in \cite{cl21,cl20.2,lzz21}.
Let $g\in C_c^\infty([0,T])$ be any cut-off function such that
\begin{align*}
	\aint_{0}^T g^2(t) \d t=1,
\end{align*}
and rescale the cut-off function $g$ by
\begin{align}\label{gk1}
	g_\tau(t) :=\tau^{\frac 12} g(\tau t),
\end{align}
where $\tau \in \mathbb{N}_+$.
By an abuse of notation,
we periodize $g_\tau$ such that
it is treated as a periodic function defined on $[0,T]$.
Moreover, we define
$h_\tau: [0,T] \to \bbr$ by
\begin{align} \label{hk}
	h_\tau(t):= \int_{0}^t \left(g_\tau^2(s)  - 1\right)\ ds,
\end{align}
and set
\begin{align}\label{gk}
	\g:=g_\tau(\sigma t),\ \
	h_{(\tau)}(t):= h_\tau(\sigma t).
\end{align}

The function $h_{(\tau)}$ will be used later in the construction of the
temporal correctors $w^{(o)}_{q+1}$ and $d^{(o)}_{q+1}$
(see \eqref{wo.2}-\eqref{do.2}),
which permit to balance the high temporal oscillations caused by $g_{(\tau)}$.

We have the following  estimates of the temporal building blocks.

\begin{lemma} [\cite{lzz21} Estimates of temporal intermittency]   \label{Lem-gk-esti}
	For  $\gamma \in [1,+\infty]$, $M  \in \mathbb{N}$,
	we have
	\begin{align}  \label{gk estimate}
		\left\|\partial_{t}^{M}\g \right\|_{L^{\gamma}_t} \lesssim \sigma^{M}\tau^{M+\frac12-\frac{1}{\gamma}},
	\end{align}
	where the implicit constants are independent of $\tau$ and $\sigma$.
	Moreover, we have
	\begin{align}\label{hk-est}
		\|h_{(\tau)}\|_{L_t^\infty}\leq 1.
	\end{align}
\end{lemma}

In the next \S \ref{Subsec-Amplitude-S1},
we define the amplitudes of the velocity and magnetic flows
adapted to the geometry of MHD equations.

\subsection{Amplitudes of velocity and magnetic perturbations}   \label{Subsec-Amplitude-S1}
The amplitudes of the magnetic and velocity perturbations
are constructed mainly to decrease the effects of
the old Reynolds and magnetic stresses,
based on Geometric Lemmas \ref{geometric lem 1} and \ref{geometric lem 2},
so that the desirable $L^1_t\cL^1_x$-decay property \eqref{rl1}
can be achieved for the new stresses at level $q+1$.

{\bf $\bullet$ The magnetic amplitudes.}
Let $\chi: [0, +\infty) \to \mathbb{R}$ be a smooth cut-off function such that
\begin{equation}\label{e4.0}
	\chi (z) =
	\left\{\aligned
	& 1,\quad 0 \leq z\leq 1, \\
	& z,\quad z \geq 2,
	\endaligned
	\right.
\end{equation}
and
\begin{equation}\label{e4.1}
	\frac 12 z \leq \chi(z) \leq 2z \quad \text{for}\quad z \in (1,2).
\end{equation}

Set
\begin{equation}\label{rhob}
	\varrho_B(t,x) := 2 \varepsilon_B^{-1}  \lbb_q^{-\frac{\ve_R}{4}}
      \delta_{q+ 1} \chi\left( \frac{|\trb(t, x) |}{\lbb_q^{-\frac{\ve_R}{4}}  \delta_{q+1} } \right),
\end{equation}
where $\varepsilon_{B}$ is the small radius in Geometric Lemma \ref{geometric lem 2},
and $\trb$ is the well-prepared magnetic stress in the last section.
Then, one has
\begin{equation}\label{rhor}
	\left|  \frac{\trb}{\varrho_B} \right|
	= \left| \frac{\trb}{2 \varepsilon_B^{-1} \lbb_q^{-\frac{\ve_R}{4}}  \delta_{q+ 1}\chi
		( \lbb_q^{\frac{\ve_R}{4}} \delta_{q+1} ^{-1} |\trb | )} \right| \leq \varepsilon_B,
\end{equation}

Moreover, choose the smooth temporal cut-off function $f_B$,
adapted to the support of $\trb$,  such that
\begin{itemize}
	\item $0\leq f_B \leq 1$ and $f_B \equiv 1$ on $\supp_t \trb$;
	\item $\supp_t f_B \subseteq N_{\thq/2}(\supp_t \trb )$;
	\item $\|f_B\|_{C_t^N}\lesssim \thq^{-N}$,\ \  $1\leq N\leq 9$.
\end{itemize}

Then, the amplitudes of the magnetic perturbations are defined by
\begin{equation}\label{akb}
	a_{(k)}(t,x):= a_{k, B}(t,x)
	=  \varrho_B^{\frac{1}{2} } (t,x) f_B(t)\gamma_{(k)}
	\left(\frac{-\trb(t,x)}{\varrho_B(t,x)}\right), \quad k \in \Lambda_B,
\end{equation}
where $\gamma_{(k)}$ is the smooth function in Geometric Lemma~\ref{geometric lem 2}.

The important algebraic and analytic properties of
the amplitudes $\{a_{(k)}, k\in \Lambda_B\}$ are summarized in Lemma \ref{Lem-mae-S1} below.
The proof of \eqref{a-mag-S1} below is similar to
that in \cite{lzz21} with slight modifications and replacing $\ell$ by $\theta_{q+1}$, we omit the details here.

\begin{lemma} [Magnetic amplitudes] \label{Lem-mae-S1}
	We have that
	\begin{align}\label{magcancel}
		&  \sum\limits_{ k \in  \Lambda_B} a_{(k)}^2 \g^2
		( D_{(k)} \otimes W_{(k)} - W_{(k)} \otimes  D_{(k)} ) \notag\\
		= & -\trb
		+  \sum\limits_{ k \in \Lambda_B}  a_{(k)}^2\g^2\P_{\neq 0}(  D_{(k)} \otimes W_{(k)} -W_{(k)} \otimes  D_{(k)} ) \notag\\
		& + \sum_{k \in \Lambda_B}  a_{(k)}^2 (\g^2-1) \aint_{\T^3}D_{(k)}\otimes W_{(k)}-W_{(k)}\otimes D_{(k)}\d x,
	\end{align}
	where $\P_{\neq 0}$ denotes the spatial projection onto nonzero Fourier modes.
	Moreover, for $1\leq N\leq 9$, $k\in \Lambda_B$,
	\begin{align}    \label{a-mag-S1}
		\norm{a_{(k)}}_{L^2_{t,x}} \lesssim \delta_{q+1}^{\frac{1}{2}} ,\ \
		& \norm{ a_{(k)} }_{C_{t,x}} \lesssim \thq^{-1},\ \ \norm{ a_{(k)} }_{C_{t,x}^N} \lesssim \thq^{-7N}.
	\end{align}
\end{lemma}

{\bf $\bullet$ The velocity amplitudes.}
In contrast with the NSE,
because of the strong coupling   between the velocity and magnetic fields,
the construction of the velocity amplitudes would reply on
the magnetic shear flows and wave vectors.
Moreover, we also would need an additional matrix $\mathring{G}^{B}$ defined by
\begin{equation}
	\label{def:G}
	\mathring{G}^{B}: = \sum_{k \in \Lambda_B}a_{(k)}^2 \aint_{\mathbb{T}^3} W_{(k)} \otimes W_{(k)} - D_{(k)} \otimes D_{(k)}\d x.
\end{equation}
It holds that, for $N\geq 1$,
\begin{align}
	\label{G estimates}
	\norm{\mathring{G}^{B}}_{C_{t,x}} \lesssim \thq^{-2},\quad
	\norm{\mathring{G}^{B}}_{C_{t,x}^N} \lesssim \ell^{-7N-1} ,\quad
	\norm{\mathring{G}^{B}}_{L^1_{t,x}} &\lesssim \delta_{q+1}.
\end{align}

Set
\begin{equation}\label{defrho}
	\begin{aligned}
		&  \varrho_u(t,x):= 2 \varepsilon_u^{-1}  \delta_{q+ 1}\chi
		\left( \frac{|\tr(t,x) + \mathring{G}^{B}(t,x)|}{\delta_{q+1}} \right).
	\end{aligned}
\end{equation}
By \eqref{e4.0} and \eqref{e4.1}, it holds that
\begin{align}  \label{R-G-veu}
	& \left|  \frac{\tr + \mathring{G}^{B}}{\varrho_u} \right| \leq \varepsilon_u,
\end{align}
where $\varepsilon_{u}$ is the small radius in Geometric Lemma \ref{geometric lem 1},
and $\tr$ is the well-prepared Reynolds stress in the last section.
Then, we choose the smooth temporal cut-off function $f_u$,
adapted to the support of $(\tr, \mathring{G}^{B})$, such that
\begin{itemize}
	\item $0\leq f_u\leq 1$, $f_u\equiv 1$ on $\supp_t (\tr, \mathring{G}^B)$;
	\item $\supp_t f_u\subseteq  N_{\thq/2}(\supp_t (\tr, \mathring{G}^B) )$;
	\item $\|f_u\|_{C_t^N}\lesssim \thq^{-N}$, $1\leq N\leq 9$.
\end{itemize}

The amplitudes of the velocity perturbations are defined by
\begin{equation}\label{velamp}
	\begin{aligned}
		&a_{(k)} (t,x) := a_{k, u}(t, x)
		= \varrho_u^{\frac{1}{2}} f_u(t)\gamma_{(k)}
		\left(\Id - \frac{\tr(t,x) +  \mathring{G}^{B}(t,x)}{\varrho_u(t,x)} \right),
		\quad k \in \Lambda_u.
	\end{aligned}
\end{equation}

We collect the algebraic and analytic properties of
the velocity amplitudes $\{a_{(k)}, k\in \Lambda_u\}$
in the following lemma.

\begin{lemma} [\cite{lzz21} Estimates of velocity amplitudes]  \label{Lem-vae-S1}
	We have that
	\begin{align}\label{velcancel}
		\sum\limits_{ k \in  \Lambda_u} a_{(k)}^2 \g^2
		W_{(k)} \otimes W_{(k)}
		& = \rho_u f^2_u \Id - \tr -  \mathring{G}^{B}
		+  \sum\limits_{ k \in \Lambda_u}  a_{(k)}^2 \g^2 \P_{\neq 0}(  W_{(k)} \otimes  W_{(k)} )\notag\\
		& \quad  +\sum_{k\in \Lambda_u}a_{(k)}^{2}\left(\g^2-1 \right)\aint_{\T^3}W_{(k)}\otimes W_{(k)}\d x .
	\end{align}
	Moreover, for $1\leq N\leq 9$, $k\in \Lambda_{u} $,
	we have
	\begin{align}\label{a-vel-S1}
		\norm{a_{(k)}}_{L^2_{t,x}} \lesssim \delta_{q+1}^{\frac{1}{2}} ,\ \
		\norm{ a_{(k)} }_{C_{t,x}} \lesssim \thq^{-1},\quad\norm{a_{(k)}  }_{C_{t,x}^N} \lesssim \thq^{-14N}.
	\end{align}
\end{lemma}

\subsection{Velocity and magnetic perturbations}   \label{Subsec-perturb-S1}

We are now in stage to define the crucial velocity and magnetic perturbations,
mainly consisting of three parts:
the principal parts, the incompressible parts
and the temporal correctors.

{\bf $\bullet$ Principal parts.}
Let us first define the principal parts
$w_{q+1}^{(p)}$ and $d_{q+1}^{(p)}$, respectively,
of the velocity and magnetic perturbations by
\begin{subequations}
	\begin{align}
		w_{q+1}^{(p)} &:= \sum_{k \in \Lambda_u\cup \Lambda_{B} } a_{(k)}\g W_{(k)},  \label{wp-def-A2} \\
		d_{q+1}^{(p)} &:= \sum_{k \in \Lambda_B } a_{(k)}\g D_{(k)}.    \label{dp-def-A2}
	\end{align}
\end{subequations}

The key fact here is that,
the MHD type nonlinearity of the principal parts
would decrease the effects of the old stresses
in the convex integration scheme,
as shown in the following algebraic identities:
\begin{align} \label{mag oscillation cancellation calculation}
	& d_{q+ 1}^{(p)} \otimes w_{q+ 1}^{(p)} - w_{q+1}^{(p)}\otimes d_{q+ 1}^{(p)}\notag + \trb   \nonumber \\
	=& \sum_{k \in \Lambda_B}  a_{(k)}^2 \g^2 \P_{\neq 0}(D_{(k)}\otimes W_{(k)}-W_{(k)}\otimes D_{(k)}) \nonumber \\
	&+ \sum_{k \in \Lambda_B}  a_{(k)}^2 (\g^2-1) \aint_{\T^3}D_{(k)}\otimes W_{(k)}-W_{(k)}\otimes D_{(k)}\d x \nonumber \\
	&+ \(\sum_{k \neq k' \in \Lambda_B}+ \sum_{k \in \Lambda_u, k' \in \Lambda_B}\)a_{(k)}a_{(k')}\g^2(D_{(k')}\otimes W_{(k)}-W_{(k)}\otimes D_{(k')}),
\end{align}
and
\begin{align}  \label{vel oscillation cancellation calculation}
	& w_{q+ 1}^{(p)} \otimes  w_{q+ 1}^{(p)} - d_{q+ 1}^{(p)} \otimes d_{q+1}^{(p)} + \tr \nonumber  \\
	=&  \rho_{u} f^2_u Id   +\sum_{k \in \Lambda_u} a_{(k)}^2 \g^2\P_{\neq0}( W_{(k)}\otimes W_{(k)}) \nonumber \\
	& + \sum_{k \in \Lambda_B} a_{(k)}^2 \g^2\P_{\neq0} (W_{(k)}\otimes W_{(k)}-D_{(k)}\otimes D_{(k)})\nonumber \\
	& +\sum_{k\in \Lambda_u}a_{(k)}^{2}\left(\g^2-1 \right)\aint_{\T^3}W_{(k)}\otimes W_{(k)}\d x \nonumber \\
	&+ \sum_{k\in \Lambda_B}a_{(k)}^{2}\left(\g^2-1 \right)\aint_{\T^3}W_{(k)}\otimes W_{(k)}-D_{(k)}\otimes D_{(k)}\d x \nonumber \\
	& + \sum_{k \neq k' \in \Lambda_u\cup\Lambda_B }a_{(k)}a_{(k')}\g^2W_{(k)}\otimes W_{(k')}-\sum_{k \neq k' \in \Lambda_B}a_{(k)}a_{(k')}\g^2D_{(k)}\otimes D_{(k')}.
\end{align}

{\bf $\bullet$ Incompressibility correctors.}
The corresponding incompressibility correctors are defined by
\begin{subequations} \label{wqc-dqc-endpt2}
	\begin{align}
		w_{q+1}^{(c)}
		&:=   \sum_{k\in \Lambda_u\cup\Lambda_{B} }\g (\nabla a_{(k)} \times \curl W_{(k)}^c+ \curl (\nabla a_{(k)} \times W_{(k)}^c)), \label{wqc-endpt2}\\
		d_{q+1}^{(c)}
		&:=   \sum_{k\in \Lambda_B }\g (\nabla a_{(k)} \times \curl D_{(k)}^c+ \curl (\nabla a_{(k)} \times D_{(k)}^c)). \label{dqc-endpt2}
	\end{align}
\end{subequations}

One has that (see \cite[(5.28), (5.29)]{bbv20})
\begin{subequations}
	\begin{align}
		&  w_{q+1}^{(p)} + w_{q+1}^{(c)} = \sum_{k \in \Lambda_u\cup\Lambda_{B}} \curl \curl ( a_{(k)} \g W_{(k)}^c) , \label{div free velocity.1} \\
		&   d_{q+1}^{(p)} + d_{q+1}^{(c)} = \sum_{k \in \Lambda_{B}} \curl \curl ( a_{(k)} \g D_{(k)}^c) , \label{div free velocity.2}
	\end{align}
\end{subequations}
and so
\begin{align*}
	\div (w_{q+1}^{(p)} + w_{q +1}^{(c)})= \div (d_{q+1}^{(p)} + d_{q +1}^{(c)})= 0,
\end{align*}
thereby justifying the definition of the incompressible correctors.

{\bf $\bullet$ Temporal correctors.}
We also need the temporal correctors $w_{q+1}^{(o)}$ and $d_{q+1}^{(o)}$
in order to balance the high temporal frequency oscillations
in \eqref{mag oscillation cancellation calculation} and \eqref{vel oscillation cancellation calculation}:
\begin{subequations} \label{wo-do-def.2}
	\begin{align}
		& w_{q+1}^{(o)}:= -\sigma^{-1}\sum_{k\in\Lambda_u }\P_{H}\P_{\neq 0}\(h_{(\tau)}\aint_{\T^3} W_{(k)}\otimes W_{(k)}\d x\nabla (a_{(k)}^2) \)\notag\\
		&\qquad\quad\  -\sigma^{-1}\sum_{k\in\Lambda_B }\P_{H}\P_{\neq 0}\(h_{(\tau)}\aint_{\T^3} W_{(k)}\otimes W_{(k)}-D_{(k)}\otimes D_{(k)}\d x\nabla (a_{(k)}^2)\)     ,\label{wo.2}\\
		& d_{q+1}^{(o)}:= -\sigma^{-1}\sum_{k\in\Lambda_B }\P_{H}\P_{\neq 0}\(h_{(\tau)} )\aint_{\T^3} D_{(k_)}\otimes W_{(k)}-W_{(k)}\otimes D_{(k)}\d x\nabla (a_{(k)}^2)\).  \label{do.2}
	\end{align}
\end{subequations}
Recall that $h_{(\tau)}$ is given by \eqref{hk}.

The effects of these temporal correctors to
balance  high temporal oscillations are encoded in
the following key algebraic identities:
\begin{align} \label{utemcom}
	&\partial_{t} w_{q+1}^{(o)}+
	\sum_{k\in \Lambda_u}
	\P_{\neq 0}\(\left(\g^2-1 \right)\aint_{\T^3}W_{(k)}\otimes W_{(k)}\d x \nabla(a_{(k)}^{2}) \) \nonumber  \\
	&+ \sum_{k\in \Lambda_B}\P_{\neq 0}\( \left(\g^2-1 \right)\aint_{\T^3}W_{(k)}\otimes W_{(k)}-D_{(k)}\otimes D_{(k)}\d x\nabla (a_{(k)}^2)\)  \nonumber  \\
	=&\left(\nabla\Delta^{-1}\div\right) \sigma^{-1}  \sum_{k \in \Lambda_u} \P_{\neq 0} \partial_{t}\(h_{(\tau)}\aint_{\T^3}W_{(k)}\otimes W_{(k)}\d x\nabla (a_{(k)}^2)\) \nonumber  \\
	&+\left(\nabla\Delta^{-1}\div\right) \sigma^{-1}  \sum_{k \in \Lambda_B} \P_{\neq 0} \partial_{t}\(h_{(\tau)}\aint_{\T^3}W_{(k)}\otimes W_{(k)}-D_{(k)}\otimes D_{(k)}\d x\nabla (a_{(k)}^2)\) \nonumber  \\
	&-\sigma^{-1}\sum_{k\in \Lambda_u}\P_{\neq 0}\(h_{(\tau)}\aint_{\T^3}W_{(k)}\otimes W_{(k)}\d x\p_t\nabla (a_{(k)}^2)\) \nonumber  \\
	&-\sigma^{-1}\sum_{k\in \Lambda_B}\P_{\neq 0}\(h_{(\tau)}\aint_{\T^3}W_{(k)}\otimes W_{(k)}-D_{(k)}\otimes D_{(k)}\d x\p_t\nabla (a_{(k)}^2)\),
\end{align}
and
\begin{align} \label{btemcom}
	&\partial_{t} d_{q+1}^{(o)}+\sum_{k\in \Lambda_B}
	\P_{\neq 0}\(\left(\g^2-1 \right)\aint_{\T^3}D_{(k)}\otimes W_{(k)}-W_{(k)}\otimes D_{(k)}\d x\nabla(a_{(k)}^{2})\)  \nonumber  \\
	=&\left(\nabla\Delta^{-1}\div\right) \sigma^{-1}  \sum_{k \in \Lambda_B} \P_{\neq 0} \partial_{t}\(h_{(\tau)}\aint_{\T^3}D_{(k)}\otimes W_{(k)}-W_{(k)}\otimes D_{(k)}\d x\nabla(a_{(k)}^{2})\) \nonumber  \\
	&-\sigma^{-1}\sum_{k\in \Lambda_B}\P_{\neq 0}\(h_{(\tau)}\aint_{\T^3}D_{(k)}\otimes W_{(k)}-W_{(k)}\otimes D_{(k)}\d x\p_t\nabla(a_{(k)}^{2})\) .
\end{align}

{\bf $\bullet$ Velocity and magnetic perturbations.}
Now we are ready to define the velocity and magnetic perturbations $w_{q+1}$ and $d_{q+1}$ at level $q+1$ by
\begin{subequations} \label{vq1-dq1-S1}
	\begin{align}
		w_{q+1} &:= w_{q+1}^{(p)} + w_{q+1}^{(c)}+\wo
		\label{velocity perturbation-endpt2},    \\
		d_{q+1} &:= d_{q+1}^{(p)} + d_{q+1}^{(c)}+\wo
		\label{magnetic perturbation-endpt2},
	\end{align}
\end{subequations}
and the velocity and  magnetic  fields at level $q+1$  by
\begin{subequations} \label{uq1-Bq1-S1}
	\begin{align}
		& u_{q+1}:= \wt u_q + w_{q+1}, \label{q+1 velocity-endpt2}\\
		& B_{q+1}:= \wt B_q + d_{q+1}, \label{q+1 magnetic-endpt2}
	\end{align}
\end{subequations}
where $(\wt u_q,\wt B_q)$ is
the well-prepared velocity and magnetic fields  in
the previous gluing stage in \S \ref{Sec-Concen-Rey}.
Note that,
by the above constructions, $w_{q+1}$ and $d_{q+1}$ are mean free and divergence free.

The main estimates of the velocity and magnetic perturbations
are summarized in Lemma \ref{Lem-perturb-S1} below.

\begin{lemma}  [Estimates of perturbations] \label{Lem-perturb-S1}
	For any $\rho \in(1,\infty), \gamma \in [1,\infty]$ and
	every integer $0\leq N\leq 7$,
	we have the following estimates:
	\begin{align}
		&\norm{\na^N w_{q+1}^{(p)} }_{L^ \gamma_tL^\rho_x } +\norm{\na^N d_{q+1}^{(p)} }_{L^ \gamma_tL^\rho_x }
       \lesssim \thq^{-1} \lbb^N\rs^{\frac{1}{\rho}- \frac 12}\tau^{\frac12-\frac{1}{ \gamma}},\label{uprinlp-endpt2}\\
		&\norm{\na^N w_{q+1}^{(c)} }_{L^\gamma_tL^\rho_x   }+\norm{\na^N d_{q+1}^{(c)} }_{L^ \gamma_tL^\rho_x }
       \lesssim \thq^{-14}\lbb^{N-1}\rs^{\frac{1}{\rho}- \frac 12}\tau^{\frac12-\frac{1}{\gamma}}, \label{ucorlp-endpt2} \\
		&\norm{\na^N \wo }_{L^\gamma_tL^\rho_x  } +  \norm{\na^N \dqo }_{L^\gamma_tL^\rho_x  }\lesssim \thq^{-14N-16}\sigma^{-1},   \label{dcorlp-endpt2}
	\end{align}
	where the implicit constants depend only on $N$, $\gamma$ and $\rho$. In particular, for integers $1\leq N\leq 7$,
	\begin{align}
		& \norm{ w_{q+1}^{(p)} }_{L^\9_tH^N_x }  + \norm{ w_{q+1}^{(c)} }_{L^\9_tH^N_x}+\norm{ \wo }_{L^\9_tH^N_x}
		\lesssim \lambda^{N+2},\label{principal h3 est.2}\\
		& \norm{ d_{q+1}^{(p)} }_{L^\9_tH^N_x }  + \norm{ d_{q+1}^{(c)} }_{L^\9_tH^N_x}+\norm{ \dqo }_{L^\9_tH^N_x}
		\lesssim \lambda^{N+2}.     \label{bprincipal h3 est.2}
	\end{align}
	Moreover, for the temporal derivatives,
	we have, for $1\leq N\leq 7$,
	\begin{align}
		& \norm{\p_t w_{q+1}^{(p)} }_{L^\9_tH^N_x }  + \norm{\p_t w_{q+1}^{(c)} }_{L^\9_tH^N_x}+\norm{\p_t \wo }_{L^\9_tH^N_x}
		\lesssim \lambda^{N+5},\label{pth2 est.2}\\
		& \norm{\p_t d_{q+1}^{(p)} }_{L^\9_tH^N_x }  + \norm{\p_t d_{q+1}^{(c)} }_{L^\9_tH^N_x}+\norm{\p_t \dqo }_{L^\9_tH^N_x}
		\lesssim \lambda^{N+5},\label{bpth2 est.2}
	\end{align}
	where the implicit constants are independent of $\lbb$.
\end{lemma}

\begin{proof}
	First, using  \eqref{ew-endpt2}, \eqref{gk estimate},
	\eqref{wp-def-A2}, \eqref{dp-def-A2} and Lemmas~\ref{Lem-mae-S1} and \ref{Lem-vae-S1}
	we get, for any $\rho \in (1,\infty)$, that 
	\begin{align}\label{uplp}
		 & \norm{\nabla^N w_{q+1}^{(p)} }_{L^\gamma_tL^\rho_x }
		+ \norm{\nabla^N d_{q+1}^{(p)} }_{L^\gamma_tL^\rho_x }  \notag \\
		\lesssim&  \sum_{k \in \Lambda_u\cup\Lambda_B}
		\sum\limits_{N_1+N_2 = N}
		\|a_{(k)}\|_{C^{N_1}_{t,x}}\|\g\|_{L_t^\gamma}
		\norm{ \nabla^{N_2}   W_{(k)} }_{C_tL^\rho_x }
          +\sum_{k \in \Lambda_B}
          \sum\limits_{N_1+N_2 = N}
          \|a_{(k)}\|_{C^{N_1}_{t,x}}\|\g\|_{L_t^\gamma}
          \norm{ \nabla^{N_2}   D_{(k)} }_{C_tL^\rho_x }  \notag \\
		\lesssim&	 \thq^{-1}\lbb^N\rs^{\frac{1}{\rho}-\frac12}\tau^{\frac12-\frac{1}{\gamma}},
	\end{align}
where we also used $\theta^{-14}_{q+1} \ll \lbb$ in the last step,
due to \eqref{b-beta-ve} and \eqref{def-mq-thetaq}.
Thus, \eqref{uprinlp-endpt2} follows.
	
	Similarly, we deduce
	\begin{align*}
		& \norm{\na^N w_{q+1}^{(c)} }_{L^\gamma_tL^\rho_x}
		+ \norm{\na^N d_{q+1}^{(c)} }_{L^\gamma_tL^\rho_x} \notag \\
		\lesssim&
		\sum\limits_{k\in\Lambda_u\cup\Lambda_B }\|\g\|_{L^\gamma_t}  \sum_{N_1+N_2=N}
		\( \norm{ a_{(k)} }_{C_{t,x}^{N_1+1}} \norm{\na^{N_2} W^c_{(k)}}_{C_tW^{1,\rho}_x }+ \norm{ a_{(k)} }_{C_{t,x}^{N_1+2}} \norm{\na^{N_2} W^c_{(k)}}_{C_tL^{\rho}_x } \)  \nonumber   \\
		&\quad +\sum\limits_{k\in \Lambda_B }\|\g\|_{L^\gamma_t}  \sum_{N_1+N_2=N}
		\( \norm{ a_{(k)} }_{C_{t,x}^{N_1+1}} \norm{\na^{N_2} D^c_{(k)}}_{C_tW^{1,\rho}_x }+ \norm{ a_{(k)} }_{C_{t,x}^{N_1+2}} \norm{\na^{N_2} D^c_{(k)}}_{C_tL^{\rho}_x } \)  \nonumber   \\
		\lesssim & \thq^{-14}\lambda^{N-1}\rs^{\frac{1}{\rho}-\frac12} \tau^{\frac12-\frac{1}{\gamma}},
	\end{align*}
	which yields \eqref{ucorlp-endpt2}.
	
	For the temporal correctors,
	using \eqref{wo-do-def.2}, \eqref{hk-est} and Lemmas \ref{Lem-mae-S1} and \ref{Lem-vae-S1}, we obtain
	\begin{align*}
	&	\norm{ \na^N \wo }_{L^\gamma_tL^\rho_x }
	   + \norm{ \na^N \dqo }_{L^\gamma_tL^\rho_x }
		\lesssim \sigma^{-1}\sum_{k \in \Lambda_u\cup\Lambda_B}\|h_{(\tau)}\|_{C_{t}} \|\nabla^{N+1} (a^2_{(k)})\|_{C_{t,x}}
		\lesssim  \thq^{-14N-16} \sigma^{-1},
	\end{align*}
	which yields \eqref{dcorlp-endpt2}.
	
	The $L^\infty_t H^N_x$-estimates \eqref{principal h3 est.2} and \eqref{bprincipal h3 est.2}
	then follow directly from estimates \eqref{uprinlp-endpt2}-\eqref{dcorlp-endpt2}
with $\gamma=\infty$ and $\rho=2$,
and the choice of parameters in \eqref{b-beta-ve}, \eqref{def-mq-thetaq} and \eqref{larsrp-endpt2}.
	
	Concerning estimate \eqref{pth2 est.2},
	by virtue of \eqref{larsrp-endpt2} and Lemmas
	\ref{buildingblockestlemma-endpt2}-\ref{Lem-vae-S1}, we get
	\begin{align} \label{wprincipal h2 est.2}
		\norm{\p_t w_{q+1}^{(p)} }_{L^\9_tH^N_x }
		\lesssim&   \sum_{k \in \Lambda_u\cup\Lambda_B }
		\|a_{(k)}\|_{C_{t,x}^{N+1} }
		\norm{\g}_{W^{1,\9}_t}\norm{ W_{(k)} }_{L^\9_tH^N_x}
		\lesssim \thq^{-14N-14} \lbb^{N} \sigma  \tau^{\frac 32}
	\end{align}
	and
	\begin{align} \label{uc h2 est.2}
		\norm{\p_t w_{q+1}^{(c)} }_{L^\9_tH^N_x  }
		& \lesssim   \sum_{k \in \Lambda_u\cup\Lambda_B}
		\|a_{(k)}\|_{C_{t,x}^{N+3}}
		\norm{\g}_{W^{1,\infty}_{t}}
		(\norm{ W^c_{(k)} }_{L^\9_tH^N_x} + \norm{ \nabla W^c_{(k)} }_{L^\9_tH^N_x} )   \nonumber \\
		& \lesssim \thq^{-14N-42} \sigma  \tau^{\frac 32}  (\lbb^{N-2}+\lbb^{N-1})  \notag \\
		& \lesssim \thq^{-14N-42} \lbb^{N-1}\sigma  \tau^{\frac 32}   .
	\end{align}
	Since $\mathbb{P}_H \mathbb{P}_{\not =0}$ is bounded in $H^N_x$,
	and $\partial_t h(\tau) = \sigma(g^2(\tau)-1)$,
we deduce
	\begin{align} \label{wo h2 est.2}
		\norm{\p_t w_{q+1}^{(o)} }_{L^\9_tH^N_x  }
		\lesssim \sigma^{-1} \sum_{k \in \Lambda_u\cup\Lambda_B } \|\p_t (h_{(\tau)} \na (a_{(k)}^2) )\|_{L^\9_tH^N_x}
		\lesssim  \thq^{-14N-16}\tau .
	\end{align}
	Thus, taking into account \eqref{larsrp-endpt2},
$\theta_{q+1}^{-14N-42} \lesssim \lbb^{2\ve}$ for $1\leq N\leq 7$,
$0<\ve\leq 1/20$ and $\alpha<3/2$, we conclude  that
	\begin{align*}
		& \norm{\p_t w_{q+1}^{(p)} }_{L^\9_tH^N_x }  + \norm{\p_t  w_{q+1}^{(c)} }_{L^\9_tH^N_x}+\norm{ \p_t \wo }_{L^\9_tH^N_x}\notag \\
		\lesssim &\, \thq^{-14N-14}\lambda^N\sigma \tau^{\frac32}
		+\thq^{-14N-42}\lambda^{N-1} \sigma\tau^{\frac32}
		+ \thq^{-14N-16}\tau \notag\\
		\lesssim &\,\thq^{-14N-14}\lambda^{3\a+N+2\ve} +\thq^{-14N-42}\lambda^{3\a+N-1+2\ve}+ \thq^{-14N-16} \lambda^{2\a} \lesssim \lambda^{N+5},
	\end{align*}
	thereby yielding \eqref{pth2 est.2}.
	Analogous arguments also apply to the temporal derivatives of the magnetic perturbations
	and give estimate \eqref{bpth2 est.2}.
	Therefore, the proof of Lemma \ref{Lem-perturb-S1} is complete.
\end{proof}

\subsection{Verification of inductive estimates for velocity and magnetic fields.}
We are now in stage to verify the inductive estimates \eqref{ubh3}, \eqref{ubpth2},
\eqref{u-B-L2tx-conv}-\eqref{u-B-Lw-conv} for the velocity and magnetic fields.

First, in view of  \eqref{ubh3},
\eqref{pdvh3}, \eqref{uq1-Bq1-S1}
and \eqref{principal h3 est.2}, for $0\leq  N\leq 4$
we have that,
\begin{align}
\norm{(u_{q+1},B_{q+1})}_{L^\9_tH^{  N+3}_x}
& \lesssim \norm{(\wt u_q,\wt B_q)}_{L^\9_tH^{ N+3}_x}+\norm{(w_{q+1},d_{q+1})}_{L^\9_tH^{  N+3}_x}\notag \\
&\lesssim \sup_i\|(v_i,H_i)\|_{L^\9(\supp (\chi_i); H^{  N+3}_x)} +\laq^{ N+5} \notag\\
&\lesssim m_{q+1}^{\frac{N}{2\alpha}} \sup_i \|(u_q, B_q)\|_{\cH^3_x}  +\laq^{ N+5}  \notag \\
&\lesssim \mq^{\frac{N}{2\a}}\la^5+ \laq^{N+5} \lesssim \lambda_{q+1}^{ N+5}, \label{verifyuc1-endpt2}
\end{align}
where we also used $m_{q+1}^{\frac{N}{2\alpha}}\la^5 \lesssim \lbb_q^{5+6N} \lesssim \lbb_{q+1}^{5+N}$,
due to \eqref{la}, \eqref{b-beta-ve} and \eqref{def-mq-thetaq}.
Similarly,
\begin{align}
&\quad \norm{(\p_t u_{q+1},\p_t B_{q+1})}_{L^\9_tH^{N}_x} \notag\\
& \lesssim \norm{(\p_t \wt u_q,\p_t \wt B_q)}_{L^\9_tH^{N}_x}+\norm{(\p_t w_{q+1},\p_t d_{q+1})}_{L^\9_tH^{N}_x}\notag \\
&\lesssim \sup_i\|(\p_t (\chi_iv_i), \p_t (\chi_iH_i))\|_{L^\9(\supp( \chi_i); \cH^{N}_x)} +\laq^{N+5}\notag\\
&\lesssim \sup_i( \|\p_t\chi_i\|_{C_t}\| (v_i,H_i) \|_{L^\9(\supp (\chi_i); \cH^{N}_x)}
+\|\chi_i\|_{C_t}\| (\p_t v_i,\p_t H_i) \|_{L^\9(\supp( \chi_i); \cH^{N}_x)})+ \laq^{ N+5} \notag\\
& \lesssim  \theta_{q+1}^{-1} m_{q+1}^{\frac{1}{2\alpha}}
           \|(u_q, B_q)\|_{L^\infty_t \cH^3_x}
           + m_{q+1}^{1+ \frac{1}{2\alpha}}
           \|(u_q, B_q)\|_{L^\infty_t \cH^3_x}  + \laq^{ N+5}  \notag \\
& \lesssim \theta_{q+1}^{-\frac 32} \lbb_q^{5}
           +\lambda_{q+1}^{N+5}\lesssim \lambda_{q+1}^{N+5}.  \label{verifyupth2-endpt2}
\end{align}
Hence, estimates \eqref{ubh3} and \eqref{ubpth2} are verified at level $q+1$.

Next, we consider estimates  \eqref{u-B-L2tx-conv} and \eqref{u-B-L1L2-conv}.
It should be mentioned that,
the derivation towards $L^2_{t}\cL^2_x$-decay estimates of principal parts
requires to exploit the decoupling between the high and low frequency parts.
To be precise,
we apply the $L^p$ decorrelation Lemma~\ref{Decorrelation1}
with $f= a_{(k)}$, $g = \g\phi_{(k)}$ and $\sigma = \lambda^{2\ve}$
and then using \eqref{la}, \eqref{b-beta-ve}
and Lemmas \ref{buildingblockestlemma-endpt2}-\ref{Lem-vae-S1}
to derive
\begin{align} \label{Lp decorr vel-endpt2}
	\norm{(w^{(p)}_{q+1},d^{(p)}_{q+1})}_{L^2_t\cL^2_{x}}
	&\lesssim \sum\limits_{k\in \Lambda_u\cup\Lambda_B}
	\Big(\|a_{(k)}\|_{L^2_{t,x}}\norm{ \g }_{L^2_{t}} \norm{ \phi_{(k)}}_{C_tL^2_{x}} +\sigma^{-\frac12}\|a_{(k)}\|_{C^1_{t,x}}\norm{ \g }_{L^2_{t}} \norm{ \phi_{(k)}}_{C_tL^2_{x}}\Big) \notag\\
	&\lesssim  \delta_{q+1}^{\frac{1}{2}}+\thq^{-14}\lambda^{-\ve}_{q+1}   \lesssim \delta_{q+1}^{\frac{1}{2}}.
\end{align}
Taking into account \eqref{ucorlp-endpt2} and \eqref{dcorlp-endpt2},
we obtain
\begin{align}  \label{e3.41.2}
	\norm{(w_{q+1},d_{q+1})}_{L^2_{t}\cL^2_x}
	\lesssim \delta_{q+1}^{\frac{1}{2}} +\thq^{-14}\lambda^{-1}+ \thq^{-16}\sigma^{-1}\lesssim \delta_{q+1}^{\frac{1}{2}},
\end{align}
which, along with \eqref{uuql2}, yields 
\begin{align}  \label{e3.43}
	\norm{(u_{q+1}-u_{q},B_{q+1}-B_{q}) }_{L^2_{t}\cL^2_x}
	&\lesssim  \norm{(\wt  u_q - u_q,\wt B_q - B_q)}_{L^\9_t \cL^2_x}+ \norm{(w_{q+1},d_{q+1})}_{L^2_{t}\cL^2_x}\nonumber  \\
	&\lesssim  \lambda_q^{-3}+\delta_{q+1}^{\frac{1}{2}}   \leq M^*\delta_{q+1}^{\frac{1}{2}}
\end{align}
for $M^*$ sufficiently large.
This verifies the $L^2_{t}\cL^2_x$-decay estimate \eqref{u-B-L2tx-conv} at level $q+1$.

The $L^1_t\cL^2_x$-estimate \eqref{u-B-L1L2-conv} can be verified easier as follows:
by Lemma \ref{Lem-perturb-S1} and \eqref{larsrp-endpt2},
\begin{align} \label{wql1.2}
	\norm{(w_{q+1},d_{q+1})}_{L^1_t \cL^2_x}
	&\lesssim\norm{(w_{q+1}^{(p)},d_{q+1}^{(p)}) }_{L^1_t \cL^2_x}
       + \norm{(w_{q+1}^{(c)},d_{q+1}^{(c)}) }_{L^1_t \cL^2_x}
       +\norm{(w_{q+1}^{(o)},d_{q+1}^{(o)}) }_{L^1_t \cL^2_x}\notag \\
	&\lesssim \thq^{-1}\tau^{-\frac12}+\thq^{-14} \lbb^{-1} \tau^{-\frac12} + \thq^{-16}\sigma^{-1}\lesssim \lambda_{q+1}^{-\ve},
\end{align}
which along with \eqref{uuql2} yields that
\begin{align}  \label{uql1l2}
	\norm{(u_{q+1}-u_{q},B_{q+1}-B_{q}) }_{L^1_t \cL^2_x} 	
	&\lesssim \norm{(\wt  u_q -u_q, \wt B_q - B_q) }_{L^1_t\cL^2_x}+ \norm{(w_{q+1},d_{q+1})}_{L^1_t\cL^2_x} \nonumber \\
	&\lesssim \lambda_q^{-3}+\lambda_{q+1}^{-\ve} \leq \delta_{q+2}^{\frac{1}{2}}.
\end{align}
Thus, the $L^1_t\cL^2_x$-estimate \eqref{u-B-L1L2-conv} is verified at level $q+1$.

At last, concerning estimate \eqref{u-B-Lw-conv},
since $(s,\gamma,p)\in \mathcal{S}_1$,
one has the embedding (see \cite[(6.32)]{lqzz22})
\begin{align}\label{sobolevem2}
	H^3_x\hookrightarrow W^{s,p}_x.
\end{align}
By virtue of \eqref{rhN} and \eqref{est-vih3}, we have
\begin{align*}
	\norm{(\wt u_q-u_q, \wt B_q-B_q)}_{L^\gamma_t \cW^{s,p}_x}
	& \lesssim \norm{(\sum_i\chi_i(v_i-u_q), \sum_i\chi_i(H_i-B_q))}_{L^\9_t \cH^3_x}\notag\\
	&\lesssim \sup_i \int_{t_i}^{s_{i+1}} \| (\abs{\nabla} \rr^u_q(s),\abs{\nabla} \rr^B_q(s))\|_{\cH^3_x}ds \notag\\
	&\lesssim m_{q+1}^{-1} \| (\abs{\nabla} \rr^u_q(s),\abs{\nabla} \rr^B_q(s))\|_{L^\9_t\cH^3_x} \notag\\
	&\lesssim m_{q+1}^{-1} \lambda_{q}^{10}\lesssim \la^{-2},
\end{align*}
where $s_{i+1}=t_{i+1}+\theta_{q+1}$.
Taking into account \eqref{b-beta-ve}, \eqref{larsrp-endpt2}, \eqref{sobolevem2}
and Lemma \ref{Lem-perturb-S1},
we derive
\begin{align}\label{lw-est.2}
	\norm{(u_{q+1} - u_q, B_{q+1} - B_q)}_{L^\gamma_t \cW^{s,p}_x}
	&\lesssim  \norm{(\wt u_q-u_q, \wt B_q - B_q)}_{L^\gamma_t \cW^{s,p}_x}
	+\norm{(w_{q+1}, d_{q+1})}_{L^\gamma_t \cW^{s,p}_x} \notag\\
	&\lesssim  \la^{-2}+ \thq^{-1}\laq^{s}\rs^{\frac{1}{p}- \frac 12}\tau^{\frac12-\frac{1}{\gamma}}
	+\thq^{-44}\sigma^{-1} \notag\\
	&\lesssim   \la^{-2}+ \lambda_{q+1}^{s+2\a-1-\frac{2\a}{\gamma}-\frac{2\a-2}{p}+\ve(6- \frac{10}{p}) }
	+ \lbb_{q+1}^{-\ve} .
\end{align}
Since by \eqref{ne3.1},
\begin{align*}
	s+2\a-1-\frac{2\a}{\gamma}-\frac{2\a-2}{p}+\ve(6-\frac{16}{p})
	   <-10\ve,
\end{align*}
we thus obtain
\begin{align*}
	\norm{(u_{q+1} - u_q, B_{q+1} - B_q)}_{L^\gamma_t \cW^{s,p}_x} \leq \delta_{q+2}^{\frac12},
\end{align*}
which verifies the $L^\gamma_t \cW^{s,p}_x$-estimate \eqref{u-B-Lw-conv} at level $q+1$.

\section{Reynolds and magnetic stresses}    \label{Sec-stress-S1}

In this section we aim to choose suitable velocity and magnetic stresses
in the new relaxation system \eqref{equa-mhdr} at level $q+1$
and to verify the inductive estimates \eqref{rhN} and \eqref{rl1}
in the supercritical regime $\mathcal{S}_1$.

\subsection{Decomposition of magnetic and Reynolds stresses}
Let us first consider the magnetic stress.
Using \eqref{equa-mhdr} with $q+1$ replacing $q$,
\eqref{vq1-dq1-S1} and \eqref{uq1-Bq1-S1}
we derive the equation for the magnetic stress:
\begin{align}
	\displaystyle\div\mathring{R}_{q+1}^B
	&=\underbrace{\partial_t (d^{(p)}_{q+ 1}+d^{(c)}_{q+ 1})+\nu_2(-\Delta)^{\alpha} d_{q+1}  +\div (d_{q + 1} \otimes \wt u_q - \wt u_q \otimes d_{q+1}+ \wt B_q\otimes w_{q+1} -w_{q +1} \otimes \wt B_q )}_{ \div\mathring{R}_{lin}^B  }   \notag\\
	&\quad+\underbrace{\div (d_{q+ 1}^{(p)} \otimes w_{q+1}^{(p)} -w_{q+ 1}^{(p)} \otimes  d_{q+ 1}^{(p)} + \trb)+ \partial_t \dqo}_{\div\mathring{R}_{osc}^B }  \notag\\
	&\quad+\div\Big( d_{q+1}^{(p)} \otimes (w_{q+1}^{(c)}+\wo) -(w_{q+1}^{(c)} +\wo) \otimes d_{q+1}  \notag  \\
	&\qquad \underbrace{\qquad\quad+(d_{q+1}^{(c)}+\dqo)\otimes w_{q+1}-w_{q+1}^{(p)} \otimes (d_{q+1}^{(c)} +\dqo)\Big) }_{\div\mathring{R}_{cor}^B }.  \label{rb.2}
\end{align}
Based on this fact,
we define the new magnetic stress by
\begin{align}\label{rbcom.2}
	\mathring{R}_{q+1}^B := \mathring{R}_{lin}^B +   \mathring{R}_{osc}^B+ \mathring{R}_{cor}^B,
\end{align}
where $\mathcal{R}^B$ is the inverse divergence operator
given by \eqref{calRB-def} in the Appendix,
the linear error
\begin{align}
	\mathring{R}_{lin}^B
	:= &  \mathcal{R}^B\(\partial_t (d^{(p)}_{q+1}+d^{(c)}_{q+1})\)
	+ \nu_2 \mathcal{R}^B (-\Delta)^{\a} d_{q+1} \nonumber \\
	&  +\mathcal{R}^B\P_{H}\div\( d_{q + 1} \otimes \wt u_q - \wt u_q \otimes d_{q+1}+ \wt B_q\otimes w_{q+1} -w_{q +1} \otimes \wt B_q\),\label{rbp.2}
\end{align}
the oscillation error
\begin{align}
	\mathring{R}_{osc}^B &:=  \sum_{k \in \Lambda_B}\mathcal{R}^B\P_{H}\P_{\neq 0}\left (\g^2 \P_{\neq 0}(D_{(k)}\otimes W_{(k)}-W_{(k)}\otimes D_{(k)} )\nabla (a_{(k)}^2)\right)\notag\\
	&\quad-\sigma^{-1}\sum_{k\in \Lambda_B}\mathcal{R}^B\P_{H}\P_{\neq 0}\(h_{(\tau)}\aint_{\T^3}D_{(k)}\otimes W_{(k)}-W_{(k)}\otimes D_{(k)} \d x\p_t\nabla(a_{(k)}^{2})\)\notag\\
	&\quad+ \(\sum_{k \neq k' \in \Lambda_B}+ \sum_{k \in \Lambda_u, k' \in \Lambda_B}\)\mathcal{R}^B\P_{H}\div\(a_{(k)}a_{(k')}\g^2(D_{(k')}\otimes W_{(k)}-W_{(k)}\otimes D_{(k')})\), \label{rob.2}
\end{align}
and the corrector error
\begin{align}    \label{rbp2.2}
	\mathring{R}_{cor}^B := &\,\mathcal{R}^B\P_{H}\div\bigg( d_{q+1}^{(p)} \otimes (w_{q+1}^{(c)}+ \wo) -(w_{q+1}^{(c)}+\wo) \otimes d_{q+1}\notag\\
	&\qquad \qquad\qquad + (d_{q+1}^{(c)}+\dqo)\otimes w_{q+1} -w_{q+1}^{(p)} \otimes (d_{q+1}^{(c)}+\dqo) \bigg).
\end{align}

Moreover, we also compute for the Reynolds stress:
\begin{align}
	\displaystyle\div\mathring{R}_{q+1}^u
	&\displaystyle = \underbrace{\partial_t (w_{q+1}^{(p)}+w_{q+1}^{(c)}) +\nu_1(-\Delta)^{\alpha} w_{q+1} +\div\big(\wt u_q \otimes w_{q+1} + w_{q+ 1} \otimes \wt u_q - \wt B_q \otimes d_{q+1} - d_{q+1} \otimes \wt B_q\big) }_{ \div\mathring{R}_{lin}^u +\nabla P_{lin} }   \notag\\
	&\displaystyle\quad+ \underbrace{\div (w_{q+1}^{(p)} \otimes w_{q+1}^{(p)} - d_{q+1}^{(p)} \otimes d_{q+1}^{(p)} +  \tr) +\partial_t \wo}_{\div\mathring{R}_{osc}^u +\nabla P_{osc}}  \notag\\
	&\displaystyle\quad+\div\Big((w_{q+1}^{(c)}+\wo)\otimes w_{q+1}+ w_{q+1}^{(p)} \otimes (w_{q+1}^{(c)}+\wo)  \notag \\
	&\qquad \underbrace{\qquad - (d_{q+1}^{(c)}+\dqo)\otimes d_{q+1}
		- d_{q+1}^{(p)} \otimes (d_{q+1}^{(c)}+\dqo) \Big)}_{\div\mathring{R}_{cor}^u +\nabla P_{cor}}
	+\nabla P_{q+1}.  \label{ru.2}
\end{align}
Hence, using the inverse divergence operator $\mathcal{R}^u$ given by \eqref{calR-def}
we define the new Reynolds stress by
\begin{align}\label{rucom.2}
	\mathring{R}_{q+1}^u := \mathring{R}_{lin}^u +   \mathring{R}_{osc}^u+ \mathring{R}_{cor}^u,
\end{align}
where the linear error
\begin{align}
	\mathring{R}_{lin}^u & := \mathcal{R}^u\(\partial_t (w_{q+1}^{(p)} +w_{q+1}^{(c)}  )\)
	+ \nu_1 \mathcal{R}^u (-\Delta)^{\a} w_{q+1} \nonumber \\
	&\quad  + \mathcal{R}^u\P_H \div \(\wt u_q \mathring{\otimes} w_{q+1} + w_{q+ 1}
	\mathring{\otimes} \wt u_q- \wt B_q \mathring{\otimes} d_{q+1} - d_{q+1} \mathring{\otimes} \wt B_q\), \label{rup.2}
\end{align}
the oscillation error
\begin{align}\label{rou.2}
	\mathring{R}_{osc}^u :=& \sum_{k \in \Lambda_u} \mathcal{R}^u \P_H\P_{\neq 0}\left(\g^2 \P_{\neq 0}(W_{(k)}\otimes W_{(k)})\nabla (a_{(k)}^2)\right) \notag\\
	&+ \sum_{k \in \Lambda_B} \mathcal{R}^u \P_H\P_{\neq 0}\left(\g^2 \P_{\neq 0}(W_{(k)}\otimes W_{(k)}-D_{(k)}\otimes D_{(k)})\nabla (a_{(k)}^2)\right)\notag\\
	&-\sigma^{-1}\sum_{k\in \Lambda_u}\mathcal{R}^u \P_H \P_{\neq 0}\(h_{(\tau)}\aint_{\T^3}W_{(k)}\otimes W_{(k)}\d x\p_t\nabla(a_{(k)}^{2})\)\notag\\
	&-\sigma^{-1}\sum_{k\in \Lambda_B}\mathcal{R}^u \P_H \P_{\neq 0}\( h_{(\tau)}\aint_{\T^3}W_{(k)}\otimes W_{(k)}-D_{(k)}\otimes D_{(k)}\d x\p_t\nabla(a_{(k)}^{2})\)\notag\\
	&+\sum_{k \neq k' \in \Lambda_u\cup\Lambda_B }\mathcal{R}^u \P_H \div\(a_{(k)}a_{(k')}\g^2W_{(k)}\otimes W_{(k')}\)\notag\\
	&  -\sum_{k \neq k' \in \Lambda_B}\mathcal{R}^u \P_H \div\(a_{(k)}a_{(k')}\g^2D_{(k)}\otimes D_{(k')}\),
\end{align}
and the corrector error
\begin{align}
	\mathring{R}_{cor}^u &
	:= \mathcal{R}^u \P_H \div \bigg( w^{(p)}_{q+1} \mathring{\otimes} (w_{q+1}^{(c)}+\wo)
	+ (w_{q+1}^{(c)}+\wo) \mathring{\otimes} w_{q+1}  \nonumber \\
	&\qquad \qquad  \qquad \ \  - d^{(p)}_{q+1} \mathring{\otimes} (d_{q+1}^{(c)}+\dqo)- (d_{q+1}^{(c)}+\dqo) \mathring{\otimes} d_{q+1} \bigg). \label{rup2.2}
\end{align}

\begin{remark}  \label{Rem-Ru-RB}
We note that,
by the algebraic identities \eqref{mag oscillation cancellation calculation}-\eqref{vel oscillation cancellation calculation},
\eqref{utemcom}-\eqref{btemcom},
the new magnetic and Reynolds stresses at level $1+1$
satisfy equations \eqref{rb.2} and \eqref{ru.2}, respectively.
Moreover,  one also has  (see, e.g., \cite{bbv20,lzz21})
\begin{subequations}
	\begin{align}
        &   \mathring{R}^u_{q+1}  = \mathcal{R}^u  \mathbb{P}_H \div \mathring{R}^u_{q+1},   \label{calRuPHdiv-Ru.2}  \\
		&  \mathring{R}_{q+1}^B  = \mathcal{R}^B  \mathbb{P}_H \div \mathring{R}_{q+1}^B.    \label{calRuPHdiv-Rb.2} 	
	\end{align}
\end{subequations}
\end{remark}

\subsection{Verification of inductive estimates of magnetic stress}

In the following we verify the inductive estimates
\eqref{rhN} and \eqref{rl1} for the new Reynolds
and magnetic stresses.

\subsubsection{\bf Verification of $L^\9_tH^{\wt N}_x$-estimates of magnetic stress}
Using the identity \eqref{calRuPHdiv-Rb.2} and
equations \eqref{equa-mhdr} at level $q+1$
we get that for $\wt N=3,4$,
\begin{align}
	\norm{\mathring R^B_{q+1}}_{L^\9_tH^{\wt N}_x}&\leq  \norm{\mathcal R^B \P_H (\div \rr^B_{q+1})}_{L^\9_tH^{\wt N}_x}\notag \\
	&\lesssim \norm{\partial_t B_{q+1}+\div(B_{q+1}\otimes u_{q+1}-u_{q+1}\otimes B_{q+1}) +\nu(-\Delta )^{\alpha}B_{q+1}}_{L^\9_tH^{\wt N-1}_x}   \label{ine-rq1h3} \\
	&\lesssim  \norm{\partial_t B_{q+1}}_{L^\9_tH^{\wt N-1}_x}
               +\norm{B_{q+1}\otimes u_{q+1}-u_{q+1}\otimes B_{q+1}}_{L^\9_tH^{\wt N}_x}
               + \norm{B_{q+1}}_{L^\9_tH^{\wt N+2\alpha-1}_x}\notag \\
	&\lesssim   \norm{\partial_t B_{q+1}}_{L^\9_tH^{\wt N-1}_x}
                +\norm{u_{q+1}}_{L^\9_tH^{\wt N}_x} \norm{B_{q+1}}_{L^\infty_{t,x}}
                +\norm{B_{q+1}}_{L^\9_tH^{\wt N}_x} \norm{u_{q+1}}_{L^\infty_{t,x}}
                 + \norm{B_{q+1}}_{L^\9_tH^{\wt N+3}_x}. \notag
\end{align}

Concerning the $L^\9_{t,x}$ estimates of $u_{q+1}$ and $B_{q+1}$ in \eqref{ine-rq1h3},
we use \eqref{nuh3}, \eqref{uq1-Bq1-S1},
the Sobolev embedding $H^2_x\hookrightarrow L^\9_x$
and Lemma \ref{Lem-perturb-S1} to derive
\begin{align}\label{ul9}
	\norm{(u_{q+1},B_{q+1})}_{L^\infty_{t}\cL^\infty_x}
	\leq& \norm{(\wt u_{q},\wt B_{q})}_{L^\infty_{t}\cL^\infty_x}
	+ \norm{(w_{q+1},d_{q+1})}_{L^\infty_{t}\cH^2_x}    \notag\\
	\lesssim&  \norm{(\wt u_{q},\wt B_{q})}_{L^\infty_{t} \cH^3_x}
	+ \norm{(w_{q+1},d_{q+1})}_{L^\infty_{t} \cH^2_x} \notag \\
	\lesssim& \la^{5}+ \lambda_{q+1}^{4} \lesssim \lambda_{q+1}^{4},
\end{align}
where the implicit constant is independent of $q$.

Thus, inserting \eqref{ubh3}, \eqref{ubpth2} and \eqref{ul9} into \eqref{ine-rq1h3},
we obtain
\begin{align} \label{RB-HN-A1}
	\norm{\mathring R^B_{q+1}}_{L^\9_tH^{\wt N}_x}
	&\lesssim  \lambda_{q+1}^{\wt N+4}+ \lambda_{q+1}^{\wt N+6} + \lambda_{q+1}^{\wt N+5}
	\lesssim \lambda_{q+1}^{\wt N+6}
\end{align}
for some universal constant,
which verifies \eqref{rhN} of $\mathring R^B_{q+1}$ at level $q+1$.

\subsubsection{\bf Verification of $L^1_{t,x}$-decay of  magnetic stress}
We aim to verify the $L^1_{t,x}$-decay \eqref{rl1} of the
magnetic stress $\mathring{R}^B_{q+1}$ at level $q+1$.
For this purpose, we choose
\begin{align}\label{defp}
	\rho: =\frac{2\a-2+10\varepsilon}{2\a-2+9\varepsilon}\in (1,2),
\end{align}
where $\ve$ is given by \eqref{ne3.1},
such that
\begin{equation}\label{setp}
	(2-2\a-10\ve)(\frac{1}{\rho}-\frac12)=1-\a-4\ve,
\end{equation}
and
\begin{align}  \label{rs-rp-p-ve-endpt2}
	\rs^{\frac 1\rho-\frac12} =  \lambda^{1-\a-4\ve}.
\end{align}

Below we estimate the three parts of
the magnetic stress $\mathring{R}^B_{q+1}$ separately.

\paragraph{\bf (i) Linear error.}
Note that, by Lemmas \ref{buildingblockestlemma-endpt2},
\ref{Lem-gk-esti}, \ref{Lem-mae-S1}
and \ref{Lem-perturb-S1}, \eqref{ne3.1}, \eqref{larsrp-endpt2}, \eqref{div free velocity.2}
and \eqref{rs-rp-p-ve-endpt2},
\begin{align}
	& \| \mathcal{R}^B \partial_t( d_{q+1}^{(p)}+ d_{q+1}^{(c)})\|_{L_t^1L_x^\rho}  \nonumber \\
	\lesssim& \sum_{k \in \Lambda_B}
        \Big(\| \g\|_{L^1_t}\|  a_{(k)} \|_{C_{t,x}^2}
          \| D^c_{(k)} \|_{C_t W_x^{1,\rho}}
          +\| \p_t\g\|_{L_t^1}\| a_{(k)} \|_{C_{t,x}^1}\| D^c_{(k)} \|_{C_t W_x^{1,\rho}}\Big)\nonumber \\
	\lesssim& \thq^{-14}\tau^{-\frac12}\rs^{\frac{1}{\rho}-\frac12}\lambda^{-1}
            + \thq^{-7}\sigma\tau^{\frac12}\rs^{\frac{1}{\rho}-\frac12}\lambda^{-1}
	\lesssim  \thq^{-14}\lambda^{-2\ve}.\label{time derivative-endpt2}
\end{align}

Next, we control the hyper-resistivity $(-\Delta)^{\alpha}$.
It is important here to exploit the temporal intermittency
in order to control the resistivity beyond the Lions exponent $5/4$.

More precisely,
using the interpolation estimate, Lemma \ref{Lem-perturb-S1},
\eqref{larsrp-endpt2},
\eqref{uprinlp-endpt2}, \eqref{rs-rp-p-ve-endpt2}
and the fact that $2-\alpha \geq  5\va$,
we estimate
\begin{align}
	\norm{\mathcal{R}(-\Delta)^{\alpha} d_{q+1}^{(p)} }_{L_t^1L^\rho_x}
	& \lesssim \norm{ |\na|^{2\a-1} d_{q+1}^{(p)} }_{L_t^1L^\rho_x}\notag\\
	& \lesssim  \norm{d_{q+1}^{(p)}}_{L_t^1L^\rho_x} ^{\frac{4-2\a}{3}} \norm{d_{q+1}^{(p)}}_{L_t^1W^{3,\rho}_x} ^{\frac{2\a-1}{3}}\notag\\
	& \lesssim \thq^{-1}\lbb^{2\alpha-1}\rs^{\frac{1}{\rho}-\frac12}\tau^{-\frac12}\lesssim \thq^{-1}\lambda^{-4\ve},\label{e5.18.2}
\end{align}
and
\begin{align}
	&\norm{\mathcal{R}(-\Delta)^{\alpha} d_{q+1}^{(c)} }_{L_t^1L^\rho_x}
	\lesssim \thq^{-14}\lbb^{2\alpha-2}\rs^{\frac{1}{\rho}-\frac12}\tau^{-\frac12}\lesssim \thq^{-14}\lambda^{-1-4\ve},\label{e5.19.2}\\
	&\norm{\mathcal{R}(-\Delta)^{\alpha} \dqo }_{L_t^1L^\rho_x}   \lesssim \thq^{-44}\sigma^{-1}\lesssim \thq^{-44}\lambda^{-2\ve}.\label{e5.21.2}
\end{align}
Hence, it follows that
\begin{align}  \label{mag viscosity-endpt2}
	\norm{\mathcal{R}(-\Delta)^{\alpha} d_{q+1} }_{L_t^1L^\rho_x}
	\lesssim& \norm{\mathcal{R}(-\Delta)^{\alpha} d_{q+1}^{(p)} }_{L_t^1L^\rho_x}
     +\norm{\mathcal{R}(-\Delta)^{\alpha} d_{q+1}^{(c)} }_{L_t^1L^\rho_x}
     +\norm{\mathcal{R}(-\Delta)^{\alpha} \dqo }_{L_t^1L^\rho_x} \notag \\
	\lesssim& \thq^{-44}\lambda^{-2\ve} .
\end{align}

At last,
for the remaining terms in \eqref{rbp.2},
using \eqref{nuh3}, Lemma \ref{Lem-perturb-S1}
and \eqref{rs-rp-p-ve-endpt2}, we have
\begin{align} \label{magnetic linear estimate1.2}
	&\norm{ \mathcal{R}^B\P_H\div\(d_{q + 1} \otimes \wt u_q - \wt u_q \otimes d_{q+1}
      + \wt B_q\otimes w_{q+1} -w_{q +1} \otimes \wt B_q\) }_{L_t^1L^\rho_x}  \nonumber \\	
	\lesssim& \norm{\wt u_q}_{L^\9_{t}H^3_x} \norm{d_{q+1}}_{L_t^1L^\rho_x} +\norm{\wt B_q}_{L^\9_{t}H^3_x} \norm{w_{q+1}}_{L_t^1L^{\rho}_x}  \nonumber \\
	\lesssim&\lambda^5_q (\thq^{-1} \rs^{\frac{1}{p}-\frac12}\tau^{-\frac 12}  + \thq^{-16} \sigma^{-1} )
	\lesssim \thq^{-17}\lambda^{-2\ve}.
\end{align}

Therefore,
we conclude from \eqref{time derivative-endpt2}, \eqref{mag viscosity-endpt2} and \eqref{linear estimate1-endpt2} that
the linear error can be bounded by
\begin{align}   \label{linear estimate-endpt2}
	\norm{\mathring{R}_{lin}^B }_{L_t^1L^\rho_x}
	& \lesssim \thq^{-14}\lambda^{-2\ve} +\thq^{-44}\lambda^{-2\ve}+\thq^{-17}\lambda^{-2\ve}
	\lesssim \thq^{-44}\lambda^{-2\ve}.
\end{align}

\paragraph{\bf (ii) Oscillation error.}
In view of \eqref{rob.2},
we decompose the oscillation error into three parts
\begin{align*}
	\mathring{R}_{osc}^B = \mathring{R}^B_{osc.1} +  \mathring{R}^B_{osc.2}+\mathring{R}^B_{osc.3},
\end{align*}
where the high-low spatial oscillation error
\begin{align*}
	\mathring{R}_{osc.1}^B
	&:=    \sum_{k \in \Lambda_B}\mathcal{R}^B \P_{H}\P_{\neq 0}\left(\g^2 \P_{\neq 0}(D_{(k)}\otimes W_{(k)}-W_{(k)}\otimes D_{(k)})\nabla (a_{(k)}^2) \right),
\end{align*}
the low frequency error
\begin{align*}
	\mathring{R}_{osc.2}^B &
	:= -\sigma^{-1}\sum_{k\in \Lambda_B}\mathcal{R}^B\P_{H}\P_{\neq 0}
	\(h_{(\tau)}\aint_{\T^3}D_{(k)}\otimes W_{(k)}-W_{(k)}\otimes D_{(k)}\d x\p_t\nabla(a_{(k)}^{2})\),
\end{align*}
and the interaction error
\begin{align*}
	\mathring{R}_{osc.3}^B &
	:=\(\sum_{k \neq k' \in \Lambda_B}+ \sum_{k \in \Lambda_u, k' \in \Lambda_B}\)\mathcal{R}^B\P_{H}\div\(a_{(k)}a_{(k')}\g^2(D_{(k')}\otimes W_{(k)}-W_{(k)}\otimes D_{(k')})\).
\end{align*}

First, the control of  $\mathring{R}_{osc.1}^B $
relies on the key fact that
the velocity and magnetic flows are of high oscillations
\begin{align*}
   \P_{\not=0} (D_{(k)} \otimes W_{(k)} - W_{(k)} \otimes D_{(k)})
    = \P_{\geq \frac 12 \lbb \rs} (D_{(k)} \otimes W_{(k)} - W_{(k)} \otimes D_{(k)}),
\end{align*}
while the amplitudes $a_{(k)}$ are slowly varying.
Hence, one may gain an extra factor $(\lambda \rs)^{-1}$ by using
the inver-divergence operator $\mathcal{R}^B$.
This leads to, via Lemma \ref{commutator estimate1}
with $a = \nabla (a_{(k)}^2)$, $f =  \phi_{(k)}^2$
and $k=\lbb \rs /2$,
\begin{align}  \label{I1-esti.2}
	\norm{\mathring{R}_{osc.1}^B }_{L^1_tL^p_x}
	&\lesssim  \sum_{ k \in \Lambda_B}
	\|\g\|_{L^2_t}^2\norm{|\nabla|^{-1} \P_{\not =0}
		\left(\P_{\geq (\lambda \rs/2)}(D_{(k)}\otimes W_{(k)}-W_{(k)}\otimes D_{(k)} )\nabla (a_{(k)}^2)\right)}_{C_tL^p_x} \notag \nonumber  \\
	& \lesssim  \sum_{ k \in \Lambda_B} \lambda^{-1}  \rs^{-1} \norm{ \na^3(a^2_{(k)})}_{C_{t,x}}
	\norm{\phi_{(k)}^2}_{L^p_x}  \nonumber  \\
	& \lesssim \thq^{-23}  \lambda^{-1}  \rs^{\frac{1}{p}-2},
\end{align}
where we also used Lemmas \ref{buildingblockestlemma-endpt2} and \ref{Lem-mae-S1}
in the last step.

Moreover, the low frequency part $\mathring{R}_{osc.2}^B $
can be estimated by using \eqref{hk-est} and \eqref{a-mag-S1}:
\begin{align}  \label{I2-esti.2}
	\norm{\mathring{R}_{osc.2} }_{L^1_tL^\rho_x}
	\lesssim \sigma^{-1} \sum_{k\in\Lambda_B} \|h_{(\tau)}\|_{C_t}\( \norm{a_{(k)} }_{C_{t,x}} \norm{a_{(k)} }_{C_{t,x}^2} +\norm{a_{(k)} }_{C_{t,x}^1}^2\)
	\lesssim \thq^{-15} \sigma^{-1}.
\end{align}

Finally, thanks to the small interactions between different intermittent spatial building blocks,
the interaction error $\mathring{R}_{osc.3}^B$
can be controlled by the product estimate in Lemma~\ref{buildingblockestlemma-endpt2},
\begin{align}\label{I3-esti.2}
	\norm{\mathring{R}_{osc.3}^B}_{L^1_tL^p_x}
   & \lesssim \( \sum_{ k \neq k' \in \Lambda_B}
     +\sum_{ k \in \Lambda_u ,  k' \in \Lambda_B }\) \norm{a_{(k)}a_{(k')}\g^2(D_{(k')}\otimes W_{(k)}-W_{(k)}\otimes D_{(k')})}_{L^1_tL^p_x} \notag\\	
	& \lesssim \( \sum_{ k \neq k' \in \Lambda_B}+\sum_{ k \in \Lambda_u ,  k' \in \Lambda_B }\)\norm{a_{(k)} }_{C_{t,x}}\norm{a_{(k')} }_{C_{t,x}}  \norm{\g^2 }_{L^1_t}\|\phi_{(k)}\phi_{(k')}\|_{C_tL^p_x}\notag\\
	& \lesssim  \thq^{-2}\rs^{\frac{2}{p}-1} \, .
\end{align}

Therefore, combing \eqref{I1-esti.2}, \eqref{I2-esti.2} and \eqref{I3-esti.2}
altogether and using \eqref{larsrp-endpt2} and \eqref{rs-rp-p-ve-endpt2}
we arrive at
\begin{align}
	\label{oscillation estimate-endpt2}
	\norm{\mathring{R}_{osc}}_{L_t^1L^\rho_x}
	&\lesssim   \thq^{-23}  \lambda^{-1}  \rs^{\frac{1}{\rho}-2}
             +\thq^{-15} \sigma^{-1}+\thq^{-2}\rs^{\frac{2}{\rho}-1}\notag \\
	&\lesssim \thq^{-23} \lbb^{2\a-3+11\ve} +\thq^{-15} \lbb^{-2\va}+\thq^{-2}\lbb^{2-2\a-8\ve} \notag\\
	&\lesssim \thq^{-15} \lbb^{-2\va},
\end{align}
where the last step is due to
$2\alpha-3 \leq -20\ve$ implied by \eqref{ne3.1}.

\paragraph{\bf (iii) Corrector error.}
It is relatively easier to control the corrector error
$\mathring{R}_{cor}^B$.
Using H\"older's inequality, Lemma \ref{Lem-perturb-S1},
\eqref{ne3.1}, \eqref{larsrp-endpt2}
and \eqref{rs-rp-p-ve-endpt2}, we get
\begin{align}
	\norm{\mathring{R}_{cor}^B }_{L^1_{t}L^{\rho}_x}
	\lesssim& \norm{w_{q+1}^{(c)} +\wo }_{L^2_{t}L^{\9}_x} (\norm{d^{(p)}_{q+1} }_{L^2_{t}L^\rho_x} + \norm{d_{q+1} }_{L^2_{t}L^\rho_x}) \notag \\
	&  +  (\norm{ w_{q+1}^{(p)}}_{L^2_{t}L^\rho_x} + \norm{ w_{q+1}}_{L^2_{t}L^\rho_x}) \norm{ d_{q+1}^{(c)}+  \dqo }_{L^2_{t}L^{\9}_x}\notag  \\
	\lesssim&  \( \thq^{-14}\lbb^{-1}\rs^{-\frac12 }+\thq^{-16}\sigma^{-1}\) \(\thq^{-1} \rs^{\frac{1}{\rho}-\frac12}  +\thq^{-14}\lbb^{-1} \rs^{\frac{1}{\rho}-\frac12}  + \thq^{-16} \sigma^{-1}\) \notag \\
	\lesssim&  \( \thq^{-14}\lbb^{\a-2+5\ve}+\thq^{-16}\lambda^{-2\ve}\) \(\thq^{-1} \lbb^{-\a+1-4\ve} + \thq^{-16}\lambda^{-2\ve}\) \notag \\
	\lesssim& \, \thq^{-32} \lambda^{-4\ve}. \label{corrector estimate-endpt2}
\end{align}

Therefore,
from  estimates \eqref{linear estimate-endpt2},
\eqref{oscillation estimate-endpt2},
\eqref{corrector estimate-endpt2} of the three parts of magnetic stress
we conclude that
\begin{align} \label{rq1bl1}
	\|\mathring{R}_{q+1}^B \|_{L^1_{t,x}}
	&\leq \| \mathring{R}^B_{lin} \|_{L^1_tL^\rho_{x}} +  \| \mathring{R}^B_{osc}\|_{L^1_tL^\rho_{x}}
	+  \|\mathring{R}^B_{cor} \|_{L^1_tL^\rho_{x}}  \nonumber  \\
	&\lesssim   \thq^{-44}\lambda^{-2\varepsilon}
	+\thq^{-15}\lambda^{-2\ve} + \thq^{-32}\lambda^{-4\ve} \nonumber  \\
	& \leq \lambda^{-\ve_R} \delta_{q+2},
\end{align}
where the last step is due to \eqref{b-beta-ve}
and $\ve_R<\ve/ 10$.
This justifies the inductive estimate \eqref{rl1}
for the $L^1_{t,x}$-decay of the magnetic stress $\mathring{R}^B_{q+1}$.

\subsection{Verification of inductive estimates of Reynolds stress}

We now verify the inductive estimates \eqref{rhN} and \eqref{rl1}
for the Reynolds stress $\mathring{R}^u_{q+1}$ given by \eqref{rucom.2}
at level $q+1$.

The verification of estimate \eqref{rhN} can be proved in a similar fashion as
the proof of \eqref{RB-HN-A1},
by using estimates \eqref{principal h3 est.2}-\eqref{bpth2 est.2}.
Below we will focus on the proof of
$L^1_{t,x}$-decay estimate \eqref{rl1}
of the three parts  $\mathring{R}_{lin}^u$, $\mathring{R}_{osc}^u$ and $\mathring{R}_{cor}^u$
for the Reynolds stress  $\mathring{R}_{q+1}^u$ at level $q+1$.

\paragraph{\bf (i) Linear error.}
Arguing as in the proof of \eqref{time derivative-endpt2},
but with \eqref{div free velocity} instead,
we get
\begin{align}
	\| \mathcal{R}^u\partial_t( w_{q+1}^{(p)}+ w_{q+1}^{(c)})\|_{L_t^1L_x^\rho}
	\lesssim  \sum_{k \in \Lambda_u\cup \Lambda_{B}}\| \mathcal{R}^u \curl\curl\partial_t(\g a_{(k)} W^c_{(k)}) \|_{L_t^1L_x^\rho}
	\lesssim  \thq^{-28}\lambda^{-2\ve}.\label{time derivative.3}
\end{align}

For the hyper-viscosity term, similarly to \eqref{mag viscosity-endpt2},
we use the temporal and spatial intermittency to derive
\begin{align}  \label{vel viscosity-endpt2}
	\norm{\mathcal{R}^u(-\Delta)^{\alpha} w_{q+1} }_{L_t^1L^\rho_x} \lesssim \thq^{-44}\lambda^{-2\ve} .
\end{align}

Moreover,
using \eqref{nuh3}, Lemma \ref{Lem-perturb-S1}
and \eqref{rs-rp-p-ve-endpt2}, we get
\begin{align} \label{linear estimate1-endpt2}
	&\norm{ \mathcal{R}^u\P_H\div (\wt u_q \otimes w_{q+1} + w_{q+ 1} \otimes \wt u_q - \wt B_q \otimes d_{q+1} - d_{q+1} \otimes \wt B_q) }_{L_t^1L^\rho_x}
	\lesssim \thq^{-17}\lambda^{-2\ve}.
\end{align}

Therefore,
estimates \eqref{time derivative.3}, \eqref{vel viscosity-endpt2} and \eqref{linear estimate1-endpt2}
together yield that
\begin{align}   \label{linear estimate-endpt2.3}
	\norm{\mathring{R}_{lin}^u }_{L_t^1L^\rho_x}
	& \lesssim \thq^{-28}\lambda^{-4\ve} +\thq^{-44}\lambda^{-2\ve}+\thq^{-17}\lambda^{-2\ve}
	\lesssim \thq^{-44}\lambda^{-2\ve}.
\end{align}

\paragraph{\bf (ii) Oscillation error.}
We infer from \eqref{rou.2} that the oscillation error consists of three parts
\begin{align*}
	\mathring{R}_{osc}^u = \mathring{R}_{osc.1}^u + \mathring{R}_{osc.2}^u + \mathring{R}_{osc.3}^u,
\end{align*}
where the high-low spatial oscillation error
\begin{align*}
	\mathring{R}_{osc.1}^u
	:=&   \sum_{k \in \Lambda_u }\mathcal{R}^u  \P_{H}\P_{\neq 0}\left(\g^2 \P_{\neq 0}(W_{(k)}\otimes W_{(k)} )\nabla (a_{(k)}^2) \right)   \notag \\
	&+ \sum_{k \in \Lambda_B} \mathcal{R}^u \P_H\P_{\neq 0}\left(\g^2 \P_{\neq 0}(W_{(k)}\otimes W_{(k)}-D_{(k)}\otimes D_{(k)})\nabla (a_{(k)}^2)\right),
\end{align*}
the low frequency error
\begin{align*}
	\mathring{R}_{osc.2}^u
	:=& -\sigma^{-1}\sum_{k\in \Lambda_u}\mathcal{R}^u \P_{H}\P_{\neq 0}
	\(h_{(\tau)}\aint_{\T^3}W_{(k)}\otimes W_{(k)} \d x\, \p_t\nabla(a_{(k)}^{2})\)    \notag \\
	&-\sigma^{-1}\sum_{k\in \Lambda_B}\mathcal{R}^u \P_H \P_{\neq 0}\( h_{(\tau)}\aint_{\T^3}W_{(k)}\otimes W_{(k)}-D_{(k)}\otimes D_{(k)}\d x\p_t\nabla(a_{(k)}^{2})\)
\end{align*}
and the interaction error
\begin{align*}
	\mathring{R}_{osc.3}^u &
	:=\sum_{k \neq k' \in \Lambda_u\cup\Lambda_B }\mathcal{R}^u \P_H \div\(a_{(k)}a_{(k')}\g^2W_{(k)}\otimes W_{(k')}\)\notag\\
	& \qquad -\sum_{k \neq k' \in \Lambda_B}\mathcal{R}^u \P_H \div\(a_{(k)}a_{(k')}\g^2D_{(k)}\otimes D_{(k')}\).
\end{align*}

Then, applying the decoupling Lemma \ref{commutator estimate1} again
with $a = \nabla (a_{(k)}^2)$ and $f =  \phi_{(k)}^2$
and using  Lemmas \ref{buildingblockestlemma-endpt2},
\ref{Lem-mae-S1} and \ref{Lem-vae-S1},
we estimate
\begin{align}  \label{I1-esti.3}
		\norm{\mathring{R}_{osc.1}^u }_{L^1_tL^\rho_x}
		&\lesssim  \sum_{ k \in \Lambda_u }
		\|\g\|_{L^2_t}^2\norm{|\nabla|^{-1} \P_{\not =0}
			\left(\P_{\geq (\lambda \rs/2)}(W_{(k)}\otimes W_{(k)} )\nabla (a_{(k)}^2)\right)}_{C_tL^\rho_x} \notag  \\
		&\quad + \sum_{ k \in  \Lambda_B }
		\|\g\|_{L^2_t}^2\norm{|\nabla|^{-1} \P_{\not =0}
			\left(\P_{\geq (\lambda \rs/2)}(W_{(k)}\otimes W_{(k)}-D_{(k)}\otimes D_{(k)} )\nabla (a_{(k)}^2)\right)}_{C_tL^\rho_x} \notag  \\
		& \lesssim \sum_{ k \in  \Lambda_u\cup\Lambda_B }
		\||\na|^3 (a^2_{(k)})\|_{C_{t,x}}  \lambda^{-1} \rs^{-1}\norm{\phi^2_{(k)} }_{C_tL^{\rho}_x}  \nonumber  \\
		& \lesssim \thq^{-44}  \lambda^{-1}  \rs^{\frac{1}{\rho}-2}.
\end{align}
Moreover,  \eqref{hk} and Lemma \ref{Lem-perturb-S1} yield that
\begin{align}  \label{I3-esti.3}
	\norm{\mathring{R}_{osc.2} }_{L^1_tL^\rho_x}
	\lesssim \sigma^{-1} \sum_{k\in\Lambda_u\cup\Lambda_B} \|h_{(\tau)}\|_{C_t}\( \norm{a_{(k)} }_{C_{t,x}} \norm{a_{(k)} }_{C_{t,x}^2} +\norm{a_{(k)} }_{C_{t,x}^1}^2\)
	\lesssim \thq^{-29} \sigma^{-1}.
\end{align}
Regarding the interaction error $\mathring{R}_{osc.3}^u$,
we use the product estimate \eqref{intersect-phik1} to estimate
\begin{align}\label{T3-esti.3}
	\norm{\mathring{R}_{osc.3}^u}_{L^1_tL^\rho_x} &\lesssim \sum_{k \neq k' \in \Lambda_u\cup \Lambda_B}\norm{a_{(k)}a_{(k')} \g^2W_{(k)}\otimes W_{(k')}}_{L^1_tL^\rho_x}   \notag\\
	&\quad+\sum_{k \neq k' \in \Lambda_B}\norm{a_{(k)}a_{(k')}\g^2D_{(k)}\otimes D_{(k')}}_{L^1_tL^\rho_x}\notag\\
	&\lesssim \sum_{k \neq k' \in \Lambda_u\cup \Lambda_B}
	\norm{a_{(k)}}_{C_{t,x}}\norm{a_{(k')}}_{C_{t,x}}\norm{\g^2}_{L^1_t}\|\phi_{(k)}\phi_{(k')}\|_{C_tL^\rho_x}\notag\\
	&\lesssim  \thq^{-2}\rs^{\frac{2}{\rho}-1}.
\end{align}

Therefore, combing \eqref{I1-esti.3}, \eqref{I3-esti.3} and \eqref{T3-esti.3}
altogether and using \eqref{ne3.1} and \eqref{larsrp-endpt2},
we arrive at
\begin{align}
	\label{oscillation estimate-endpt2.3}
	\norm{\mathring{R}_{osc}^u}_{L_t^1L^\rho_x}
	&\lesssim   \thq^{-44}  \lambda^{-1} \rs^{\frac{1}{\rho}-2}+\thq^{-29} \sigma^{-1} + \thq^{-2}\rs^{\frac{2}{\rho}-1}\notag \\
	&\lesssim \thq^{-44} \lbb^{2\a-3+11\ve} +\thq^{-29} \lbb^{-2\va}+\thq^{-2}\lbb^{2-2\a-8\ve}  \notag\\
	&\lesssim \thq^{-29} \lbb^{-2\va}.
\end{align}

\paragraph{\bf (iii) Corrector error.}
Arguing as in the proof of \eqref{corrector estimate-endpt2},
we also have
\begin{align}
	\norm{\mathring{R}_{cor}^u}_{L^1_{t}L^{\rho}_x}
	\lesssim& \norm{ w_{q+1}^{(p)} \otimes (w_{q+1}^{(c)}+\wo) +(w_{q+1}^{(c)} +\wo) \otimes w_{q+1}\notag\\
		&\quad- (d_{q+1}^{(c)}+\dqo)\otimes d_{q+1} -d_{q+1}^{(p)} \otimes (d_{q+1}^{(c)} +\dqo) }_{L^1_{t}L^{\rho}_x} \notag \\
	\lesssim& \, \thq^{-32} \lambda^{-4\ve}. \label{corrector estimate-endpt2.3}
\end{align}

Therefore,
we conclude from  estimates \eqref{linear estimate-endpt2.3},
\eqref{oscillation estimate-endpt2.3} and
\eqref{corrector estimate-endpt2.3} that
\begin{align} \label{rq1bl1.2}
	\|\mathring{R}_{q+1} \|_{L^1_{t,x}}
	&\leq \| \mathring{R}_{lin} \|_{L^1_tL^\rho_{x}} +  \| \mathring{R}_{osc}\|_{L^1_tL^\rho_{x}}
	+  \|\mathring{R}_{cor} \|_{L^1_tL^\rho_{x}}  \nonumber  \\
	&\lesssim   \thq^{-44}\lambda^{-2\varepsilon}
	+\thq^{-29}\lambda^{-2\ve} + \thq^{-32}\lambda^{-4\ve} \nonumber  \\
	& \leq \lambda^{-\ve_R} \delta_{q+2},
\end{align}
which verifies the inductive $L^1_{t,x}$-decay  estimate \eqref{rl1}
of the new Reynolds stress $\mathring{R}^u_{q+1}$
at level $q+1$.

\section{The supercritical regime $\mathcal{S}_2$} \label{Sec-S2}

In this section,
we treat the other supercritical regime $\mathcal{S}_2$ when $\alpha\in [1,3/2)$.
Quite differently from $\mathcal{S}_1$,
we shall use the building blocks with stronger spatial intermittency,
in order to achieve the possibly low spatial integrability in $\mathcal{S}_2$.
Hence, unlike in \S \ref{Sec-perturb-S1},
we choose the intermittent flows
constructed in \cite{lzz21},
which in particular provide  2D spatial intermittency.

The main parameters here for the intermittent flows
will be indexed by six parameters $\rs$, $\rp$, $\lambda$, $\mu$, $\tau$ and $\sigma$,
chosen in the following way:
\begin{equation}\label{larsrp}
	\rs := \lambda_{q+1}^{-1+2\varepsilon},\ \rp := \lambda_{q+1}^{-1+6\varepsilon},\
	 \lambda := \lambda_{q+1},\ \tau:=\lambda_{q+1}^{4\a-4+12\varepsilon}, \
      \mu:=\lambda_{q+1}^{2\a-1+3\ve}, \ \sigma:=\lambda_{q+1}^{2\varepsilon},
\end{equation}
where $\varepsilon$ is the small constant satisfying \eqref{e3.1}.
Note that,
two new parameters $\rp$ and $\mu$, respectively,
are introduced here,
in order to further concentrate the flows along direction $k_1$
and to balance high temporal oscillations.

\subsection{Spatial-temporal building blocks.}
Let $\psi: \mathbb{R} \rightarrow \mathbb{R}$ be a smooth and mean-free function,
supported on the interval $[-1,1]$, satisfying
\begin{equation}\label{e4.92}
	\frac{1}{2 \pi} \int_{\mathbb{R}} \psi^{2}\left(x\right) \d x=1.
\end{equation}
The corresponding rescaled cut-off functions are defined by
\begin{equation*}
	\psi_{\rp}\left(x\right) := {r_{\|}^{- \frac 1  2}} \psi\left(\frac{x}{r_{\|}}\right).
\end{equation*}
Note that, $\psi_{\rp}$ is supported in the ball of radius $\rp$ in $\bbr$.
By an abuse of notation,
we periodize $\psi_{\rp}$ so that it is treated as a periodic function defined on $\mathbb{T}$.
We also keep using the same rescaled function  $\psi_{\rs}$ and $\Psi_{\rp}$
as in \S \ref{Subsec-Flow-S1}

The new \textit{intermittent velocity and magnetic flows} are defined by
\begin{align*}
	& W_{(k)} :=  \psi_{\rp}(\lambda \rs N_{\Lambda}(k_1\cdot x+\mu t))\phi_{\rs}( \lambda \rs N_{\Lambda}k\cdot x)k_1,\ \  k \in \Lambda_u \cup \Lambda_B  , \notag \\
	& D_{(k)} :=  \psi_{\rp}(\lambda \rs N_{\Lambda}(k_1\cdot x+\mu t))\phi_{\rs}( \lambda \rs N_{\Lambda}k\cdot x)k_2, \ \ k \in \Lambda_B .
\end{align*}
Here, $N_\Lambda$, $(k_1,k_2,k)$,
$\Lambda_u$ and $\Lambda_B$ are as in \S \ref{Subsec-Flow-S1}.
We note that,
compared to the previous spatial building blocks in \S \ref{Subsec-Flow-S1},
the current building blocks
$\{W_{(k)}, D_{(k)}\}$ are
supported on thinner cuboids with length $\sim {1}/{(\lbb \rs)}$,
width $\sim {\rp}/{(\lbb \rs)}$ and height $\sim {1}/{\lbb}$.
See \cite[Figure 2]{lzz21}.
Moreover,
by choosing $k_2\neq k_2'$ if $k\neq k'$,
one has much smaller volume of
the intersections of distinct intermittent flows,
see Lemma \ref{Lem-build-S2} below.

For brevity of notations,
we let
\begin{equation}\label{snp}
	\begin{array}{ll}
		&\phi_{(k)}(x) := \phi_{\rs}( \lambda \rs N_{\Lambda}k\cdot x), \ \
		\Phi_{(k)}(x) := \Phi_{\rs}( \lambda \rs N_{\Lambda}k\cdot x),  \\
		&\psi_{(k_1)}(x) :=\psi_{\rp}(\lambda \rs N_{\Lambda}(k_1\cdot x+\mu t)),
	\end{array}
\end{equation}
and rewrite
\begin{subequations} \label{snwd}
\begin{align}
	W_{(k)} =& \psi_{(k_1)}\phi_{(k)} k_1,\ \ k\in \Lambda_u\cup \Lambda_B, \label{Wk-def-A1} \\
    D_{(k)} =& \psi_{(k_1)}\phi_{(k)} k_2,\ \ k\in \Lambda_B.    \label{Dk-def-A1}
\end{align}
\end{subequations}

The corresponding incompressible correctors are defined by
\begin{subequations} \label{wtWkc-wtDkc-def-A1}
	\begin{align}
		&\wt W_{(k)}^c := \frac{1}{\lambda^2N_{ \Lambda }^2} \nabla\psi_{(k_1)}\times\curl(\Phi_{(k)} k_1), \ \ k\in \Lambda_u \cup \Lambda_B,  \label{wtWkc-def-A1} \\
        &\wt D_{(k)}^c:=-\frac{1}{ \lbb^2N_\Lambda^2}\Delta \psi_{(k_1)} \Phi_{(k)} k_2,\ \  k\in \Lambda_B.  \label{wtDkc-def-A1}
	\end{align}
\end{subequations}
Let
\begin{subequations} \label{WcDc-def-A1}
\begin{align} \
	& W^c_{(k)} := \frac{1}{\lambda^2N_{\Lambda}^2 } \psi_{(k_1)}\Phi_{(k)} k_1, \ \ k\in \Lambda_u \cup \Lambda_B, \label{Wc-def-A1} \\
    & D^c_{(k)} :=\frac{1}{\lambda^2N_{\Lambda}^2}\psi_{(k_1)}\Phi_{(k)}k_2,  \ \ k\in  \Lambda_B. \label{Dc-def-A1}
\end{align}
\end{subequations}
Then, it holds that (see \cite[(3.18), (3.22)]{lzz21})
\begin{subequations} \label{WcDc-curl}
   \begin{align}
	 & W_{(k)} + \wt W_{(k)}^c
	=\curl \curl W^c_{(k)}, \ \ k\in \Lambda_u \cup \Lambda_B,  \label{W-wtW-Wc} \\
     & D_{(k)} + \wt D_{(k)}^c
	=\curl \curl D^c_{(k)}, \ \  k\in \Lambda_B, \label{D-wtD-Dc}
   \end{align}
\end{subequations}
and thus
\begin{align}
    & \div (W_{(k)}+ \wt W^c_{(k)}) =0, \ \ k\in \Lambda_u \cup \Lambda_B,  \label{div-Wck-Wk-0}   \\
    & \div (D_{(k)}+ \wt D^c_{(k)}) =0, \ \ k\in \Lambda_B. \label{div-Dck-Dk-0}
\end{align}

Besides the above algebraic identities adapted to the geometry of MHD equations,
another nice feature of the current spatial building blocks is that,
they provide the  2D spatial intermittency and permit to control the
hypo-dissipativity and hypo-resistivity $(-\Delta)^{\alpha_i}$
for any $\alpha_i \in [0,1)$, $i=1,2$.
Moreover,
stronger intermittency also can be gained for
the interactions between different spatial building blocks.
This is the content of Lemma \ref{Lem-build-S2} below.

\begin{lemma} [\cite{lzz21} Estimates of spatial intermittency] \label{Lem-build-S2}
	For $p \in [1,+\infty]$, $N,\,M \in \mathbb{N}$, we have
	\begin{align}
		&\left\|\nabla^{N} \partial_{t}^{M} \psi_{(k_1)}\right\|_{C_t L^{p}_{x}}
		\lesssim r_{\|}^{\frac 1p- \frac 12}\left(\frac{r_{\perp} \lambda}{r_{\|}}\right)^{N}
          \left(\frac{r_{\perp} \lambda \mu}{r_{\|}}\right)^{M}, \label{intermittent estimates} \\
		&\left\|\nabla^{N} \phi_{(k)}\right\|_{L^{p}_{x}}+\left\|\nabla^{N} \Phi_{(k)}\right\|_{L^{p}_{x}}
		\lesssim r_{\perp}^{\frac 1p- \frac 12}  \lambda^{N}, \label{intermittent estimates2}
	\end{align}
	where the implicit constants are independent of $\rs,\,\rp,\,\lambda$ and $\mu$.
	In particular, we have
	\begin{align}
		&\displaystyle\left\|\nabla^{N} \partial_{t}^{M} W_{(k)}\right\|_{C_t  L^{p}_{x}}
          +\frac{r_{\|}}{r_{\perp}}\left\|\nabla^{N} \partial_{t}^{M} \wt W_{(k)}^{c}\right\|_{C_t L^{p}_{x}}
          +\lambda^{2}\left\|\nabla^{N} \partial_{t}^{M} W_{(k)}^c\right\|_{C_t L^{p}_{x}}\displaystyle \nonumber \\
		&\qquad \lesssim r_{\perp}^{\frac 1p- \frac 12} r_{\|}^{\frac 1p- \frac 12} \lambda^{N}
          \left(\frac{r_{\perp} \lambda \mu}{r_{\|}}\right)^{M}, \ \ k\in \Lambda_u \cup \Lambda_B, \label{ew}  \\
		&\displaystyle\left\|\nabla^{N} \partial_{t}^{M} D_{(k)}\right\|_{C_t  L^{p}_{x}}
        +\frac{r_{\|}}{r_{\perp}}\left\|\nabla^{N} \partial_{t}^{M} \wt D_{(k)}^{c}\right\|_{C_t L^{p}_{x}}
         +\lambda^{2}\left\|\nabla^{N} \partial_{t}^{M} D_{(k)}^c\right\|_{C_t L^{p}_{x}}   \nonumber \\
		&\qquad  \lesssim r_{\perp}^{\frac 1p- \frac 12} r_{\|}^{\frac 1p- \frac 12}
          \lambda^{N}\left(\frac{r_{\perp} \lambda \mu}{r_{\|}}\right)^{M},\ \ k\in \Lambda_B.    \label{ed}
	\end{align}
Moreover, for every $k \neq k'\in \Lambda_{u}\cup \Lambda_{B}$ and $p \in [1, \infty]$, we have
\begin{equation}  \label{intersect-phik1-phik1'}
\|\psi_{(k_1)}\phi_{(k)}\psi_{(k'_1)}\phi_{(k')} \|_{C_tL^p_x}\lesssim \rs^{\frac{1}{p}-1}\rp^{\frac{2}{p}-1},
\end{equation}
where the implicit constant is independent of the parameters $\rs,\,\rp$ and $\lambda$.
\end{lemma}

At last,
in order to control the hyper-viscosity and hyper-resistivity $(-\Delta)^{\alpha}$
when $\alpha \geq 1$,
again it is crucial to use the temporal intermittency.
We shall use the temporal building blocks as in \S \ref{Subsec-Flow-S1},
that is, the same temporal building blocks $g_{(\tau)}$ and $h_{(\tau)}$
given by \eqref{gk},
but with the new parameters $\tau, \sigma$ given by \eqref{larsrp},
which turns out to be effective to treat the current supercritical regime $\mathcal{S}_2$.

\subsection{Velocity and magnetic perturbations}   \label{subsec-Perturbation-S2}

Let us first define the amplitude functions.

{\bf $\bullet$ Amplitudes.}
As in \S \ref{Subsec-Amplitude-S1},
the amplitudes of the magnetic perturbations
are defined by
\begin{align}
	&	a_{(k)}(t,x):= \varrho^{\frac{1}{2} }_B (t,x) f_B (t)\gamma_{(k)}
	\(-\frac{\trb(t,x)}{\varrho_B(t,x)} \), \quad k \in \Lambda_B, \label{ak-endpt2}
\end{align}
and the amplitudes of the velocity perturbations are defined by
\begin{align}
	&	a_{(k)}(t,x):=  \varrho^{\frac{1}{2} }_u (t,x) f_u (t)\gamma_{(k)}
	\({\rm Id}-\frac{\tr(t,x)+\mathring{G}^B(t,x)}{\varrho_u(t,x)}\), \quad k \in \Lambda_u, \label{akb-endpt2}
\end{align}
where $\varrho_u, \varrho_B, f_u, f_B$, $\gamma_{(k)}$
and $\mathring{G}^B$ are defined as in \S \ref{Subsec-Flow-S1}.
Note that,
the amplitudes obey the same estimates as in Lemmas \ref{Lem-mae-S1} and \ref{Lem-vae-S1}.
Namely, one has

\begin{lemma} \label{Lem-a-S2}
	For $1\leq N\leq 9$ we have
	\begin{align}  \label{a-mag-S2}
		\norm{a_{(k)}}_{L^2_{t,x}} \lesssim \delta_{q+1}^{\frac{1}{2}} ,  \ \
		\norm{ a_{(k)} }_{C_{t,x}} \lesssim \thq^{-1},\ \ \norm{ a_{(k)} }_{C_{t,x}^N} \lesssim \thq^{-7N},
        \ \ k\in \Lambda_B,
	\end{align}
	and
	\begin{align}  \label{a-vel-S2}
		\norm{a_{(k)}}_{L^2_{t,x}} \lesssim \delta_{q+1}^{\frac{1}{2}} , \ \
		\norm{ a_{(k)} }_{C_{t,x}} \lesssim \thq^{-1},\ \ \norm{ a_{(k)} }_{C_{t,x}^N} \lesssim \thq^{-14N},
     \ \ k\in \Lambda_u,
	\end{align}
	where the implicit constants are independent of $q$.
\end{lemma}

Next, we are going to construct the velocity and magnetic perturbations,
which consist of the principal parts,
the incompressible correctors
and the temporal correctors.
It should be mentioned that,
the incompressibility correctors in \eqref{wqc-dqc} below
are different from the previous ones in \eqref{wqc-dqc-endpt2}
in the supercritical regime $\mathcal{S}_1$.
Moreover,
the current supercritical regime  $\mathcal{S}_2$
also would require a new type of temporal correctors
$w^{(t)}_{q+1}$ and $d^{(t)}_{q+1}$,
in order to balance the high spatial oscillations
caused by the concentration function $\psi_{\rp}$,
that did not appear in the previous supercritical regime $\mathcal{S}_1$.

{\bf $\bullet$ Principal parts.}
The principal parts
of the velocity
and  magnetic perturbations are defined by
\begin{subequations}\label{pp}
	\begin{align}
		w_{q+1}^{(p)} &:= \sum_{k \in \Lambda_u \cup \Lambda_B } a_{(k)}\g W_{(k)},
		\label{pv}\\
		d_{q+1}^{(p)} &:= \sum_{k \in \Lambda_B} a_{(k)}\g D_{(k)}.
		\label{ph}
	\end{align}
\end{subequations}
Note that,
the algebraic identities \eqref{mag oscillation cancellation calculation}
and \eqref{vel oscillation cancellation calculation} still hold
(see \cite{lzz21}).

{\bf $\bullet$ Incompressibility correctors.}
The corresponding incompressibility correctors are defined by
\begin{subequations} \label{wqc-dqc}
	\begin{align}
		w_{q+1}^{(c)}
		&:=   \sum_{k\in \Lambda_u \cup\Lambda_B  } \g\left(\curl (\nabla a_{(k)} \times W^c_{(k)})
		+ \nabla a_{(k)} \times \curl W^c_{(k)} +a_{(k)} \wt W_{(k)}^c \right) , \label{wqc} \\
		d_{q+1}^{(c)} &:=   \sum_{k\in \Lambda_B } \g\left(\curl (\nabla a_{(k)} \times D_{(k)}^c)
		+ \nabla a_{(k)} \times \curl D_{(k)}^c
		+a_{(k)}\wdc\right ), \label{dqc}
	\end{align}
\end{subequations}
where $W^c_{(k)}$ and  $D_k^c $  are given by \eqref{WcDc-def-A1}
and $\wt W_{(k)}^c$and $\wt D_{(k)}^c$ are as in \eqref{wtWkc-wtDkc-def-A1}.
One has  that
(see \cite[(4.35a), (4.35b)]{lzz21})
\begin{subequations}
	\begin{align}
		&  w_{q+1}^{(p)} + w_{q+1}^{(c)}
		=\curl \curl \left(  \sum_{k \in \Lambda_u \cup \Lambda_B} a_{(k)} \g W^c_{(k)} \right), \label{div free velocity} \\
		&  d_{q+1}^{(p)} + d_{q+1}^{(c)}=\curl  \curl \left(  \sum_{k \in \Lambda_B} a_{(k)}\g D_{(k)}^c \right).  \label{div free magnetic}
	\end{align}
\end{subequations}
In particular,
\begin{align} \label{div-wpc-dpc-0}
	\div (w_{q+1}^{(p)} + w_{q +1}^{(c)}) = \div (d_{q+1}^{(p)} + d_{q +1}^{(c)}) =  0.
\end{align}

{\bf $\bullet$ Two types of temporal correctors to balance spatial-temporal oscillations.}
As in the regime $\mathcal{S}_1$,
the temporal correctors $w_{q+1}^{(o)}$ and $d_{q+1}^{(o)}$
of the same expressions as in \eqref{wo.2} and \eqref{do.2},
respectively,
also will be used in the current situation,
in order to balance the high temporal oscillations.
We note that
the algebraic identities \eqref{utemcom} and \eqref{btemcom} still hold.

However, unlike in the previous supercritical regime $\mathcal{S}_1$
in \S \ref{Subsec-perturb-S1},
we also need to introduce a new type of temporal corrections,
particularly, to balance the high spatial oscillations.

More precisely, we define the temporal correctors $w_{q+1}^{(t)}$ and $d_{q+1}^{(t)}$ by
\begin{subequations}  \label{veltemcor-magtemcor}
	\begin{align}
		&w_{q+1}^{(t)} := -{\mu}^{-1}  \sum_{k\in \Lambda_u\cup \Lambda_B} \P_{H}\P_{\neq 0}(a_{(k)}^2\g^2 \psi_{(k_1)}^2 \phi_{(k)}^2 k_1),\label{veltemcor}\\
		\label{magtemcor}
		&d_{q+1}^{(t)} := -{\mu}^{-1}  \sum_{k\in \Lambda_B}\P_{H}\P_{\neq 0}(a_{(k)}^2\g^2 \psi_{(k_1)}^2 \phi_{(k)}^2 k_2),
	\end{align}
\end{subequations}
where $\P_{H}$ denotes the Helmholtz-Leray projector, i.e.,
$\P_{H}=\Id-\nabla\Delta^{-1}\div. $

The important algebraic identities to
balance high spatial oscillations are
stated below
(see \cite[(4.38), (4.39)]{lzz21}):
\begin{align} \label{utem}
	&\partial_{t} w_{q+1}^{(t)}+    \sum_{k \in \Lambda_u \cup \Lambda_B}  \P_{\neq 0}
	\(a_{(k)}^{2}\g^2 \div(W_{(k)} \otimes W_{(k)})\)  \nonumber  \\
	=&(\nabla\Delta^{-1}\div)  {\mu}^{-1}  \sum_{k \in \Lambda_u \cup \Lambda_B}  \P_{\neq 0} \partial_{t}
	\(a_{(k)}^{2}\g^2 \psi_{(k_1)}^{2} \phi_{(k)}^{2}  k_1\)   \nonumber  \\
	& - {\mu}^{-1}    \sum_{k \in \Lambda_u \cup \Lambda_B}   \P_{\neq 0}
(\partial_{t}\( a_{(k)}^{2}\g^2)  \psi_{(k_1)}^{2} \phi_{(k)}^{2} k_1\),
\end{align}
and
\begin{align} \label{btem}
	&\partial_{t} d_{q+1}^{(t)}+ \sum_{k\in \Lambda_B} \P_{\neq 0}
	\(a_{(k)}^{2}\g^2 \div(D_{(k)} \otimes  W_{(k)}-W_{(k)} \otimes D_{(k)} )\)   \nonumber  \\
	=&(\nabla\Delta^{-1}\div) {\mu}^{-1}   \sum_{k \in \Lambda_B} \P_{\neq 0} \partial_{t}
	\(a_{(k)}^{2}\g^2 \psi_{(k_1)}^{2} \phi_{(k)}^{2}  k_2\)  \nonumber  \\
	&  - {\mu}^{-1}  \sum_{k \in \Lambda_B}
	\P_{\neq 0}(\partial_{t} \(a_{(k)}^{2}\g^2) \psi_{(k_1)}^{2}\phi_{(k)}^{2} k_2 \).
\end{align}

{\bf $\bullet$ Velocity and magnetic perturbations.}
Now we define the velocity and magnetic perturbations $w_{q+1}$ and $d_{q+1}$
at level $q+1$:
\begin{subequations}   \label{perturbation}
	\begin{align}
		w_{q+1} &:= w_{q+1}^{(p)} + w_{q+1}^{(c)}+ w_{q+1}^{(t)}+\wo,
		\label{velocity perturbation} \\
		d_{q+1} &:= d_{q+1}^{(p)} + d_{q+1}^{(c)}+ d_{q+1}^{(t)}+\dqo.
		\label{magnetic perturbation}
	\end{align}
\end{subequations}
The corresponding velocity and magnetic fields at level $q+1$ are defined by
\begin{subequations} \label{uB-q+1-A1}
	\begin{align}
		& u_{q+1}:=\wt u_{q} + w_{q+1},
		\label{q+1 velocity}\\
		& B_{q+1}:=\wt B_{q}+ d_{q+1},
		\label{q+1 magnetic}
	\end{align}
\end{subequations}
where $\wt u_q$, $\wt B_q$ are the well-prepared
velocity and magnetic fields in \S \ref{Sec-Concen-Rey}.

Lemma \ref{Lem-perturb-A1} below contains
the key analytic estimates of the velocity and magnetic perturbations.

\begin{lemma}  [Estimates of perturbations] \label{Lem-perturb-A1}
Let $\rho \in(1,\9), \gamma \in [1,+\infty]$ and integer $0\leq N\leq 7$.
Then, the following estimates hold:
	\begin{align}
		&\norm{\na^N w_{q+1}^{(p)} }_{L^ \gamma_tL^\rho_x } + \norm{\na^N d_{q+1}^{(p)} }_{L^ \gamma_tL^\rho_x   } \lesssim \thq^{-1} \lbb^N\rs^{\frac{1}{\rho}-\frac12}\rp^{\frac{1}{\rho}-\frac12}\tau^{\frac12-\frac{1}{ \gamma}},\label{uprinlp}\\
		&\norm{\na^N w_{q+1}^{(c)} }_{L^\gamma_tL^\rho_x   } +\norm{\na^N d_{q+1}^{(c)} }_{L^\gamma_tL^\rho_x  } \lesssim \thq^{-1}\lbb^N\rs^{\frac{1}{\rho}+\frac12}\rp^{\frac{1}{\rho}-\frac{3}{2}}\tau^{\frac12-\frac{1}{\gamma}}, \label{ucorlp}\\
		&\norm{ \na^Nw_{q+1}^{(t)} }_{L^\gamma_tL^\rho_x   }+ \norm{\na^N d_{q+1}^{(t)} }_{L^\gamma_tL^\rho_x  }\lesssim \thq^{-2}\lbb^N\mu^{-1}\rs^{\frac{1}{\rho}-1}\rp^{\frac{1}{\rho}-1}\tau^{1-\frac{1}{\gamma}} ,\label{dco rlp}\\
		&\norm{\na^N \wo }_{L^\gamma_tL^\rho_x  }
         + \norm{\na^N \dqo }_{L^\gamma_tL^\rho_x  }\lesssim \thq^{-14N-16}\sigma^{-1} ,\label{dcorlp}
	\end{align}
where the implicit constants depend only on $N$, $\gamma$ and $\rho$.
In particular, for integrals $1\leq N\leq 7$,
\begin{align}
	& \norm{ w_{q+1}^{(p)} }_{L^\9_tH^N_x }  + \norm{ w_{q+1}^{(c)} }_{L^\9_tH^N_x}+\norm{ w_{q+1}^{(t)} }_{L^\9_tH^N_x}+\norm{ \wo }_{L^\9_tH^N_x}
	\lesssim \lambda^{N+2},\label{principal h3 est-endpt1}\\
		& \norm{ d_{q+1}^{(p)} }_{L^\9_tH^N_x }  + \norm{ d_{q+1}^{(c)} }_{L^\9_tH^N_x}+\norm{ d_{q+1}^{(t)} }_{L^\9_tH^N_x}+\norm{ \dqo }_{L^\9_tH^N_x}
	\lesssim \lambda^{N+2}.   \label{bprincipal h3 est-endpt1}
\end{align}
Moreover, for the temporal derivatives,
we have that for integers $1\leq N\leq 7$,
\begin{align}
	& \norm{\p_t w_{q+1}^{(p)} }_{L^\9_tH^N_x }  + \norm{\p_t w_{q+1}^{(c)} }_{L^\9_tH^N_x}+\norm{\p_t w_{q+1}^{(t)} }_{L^\9_tH^N_x}+\norm{\p_t \wo }_{L^\9_tH^N_x}
	\lesssim \lambda^{N+5},\label{pth2 est-endpt1}\\
	& \norm{\p_t d_{q+1}^{(p)} }_{L^\9_tH^N_x }  + \norm{\p_t d_{q+1}^{(c)} }_{L^\9_tH^N_x}+\norm{\p_t d_{q+1}^{(t)} }_{L^\9_tH^N_x}+\norm{\p_t \dqo }_{L^\9_tH^N_x}
	\lesssim \lambda^{N+5},\label{bpth2 est-endpt1}
\end{align}
where the implicit constants are independent of $\lbb$.
\end{lemma}

\begin{proof}
Estimates \eqref{uprinlp}, \eqref{ucorlp} and \eqref{dcorlp}
can be proved in the same fashion of \eqref{uprinlp-endpt2}-\eqref{dcorlp-endpt2}.
Regarding the new temporal correctors $w^{(t)}_{q+1}$ and $d^{(t)}_{q+1}$,
by \eqref{veltemcor}, \eqref{magtemcor},
Lemmas \ref{Lem-gk-esti}, \ref{Lem-build-S2}, \ref{Lem-a-S2}
and the boundedness of operators $P_{\not =0}$ and $\mathbb{P}_H$ in $L^\rho$,
$\rho\in (1,\infty)$,
\begin{align}\label{lp vel time estimate}
		 \norm{\na^N w_{q+1}^{(t)} }_{L^\gamma_tL^\rho_x}
			\lesssim & \,\mu^{-1}    \sum_{k \in \Lambda_u \cup \Lambda_B}
			\norm{ \g }_{L^{2\gamma}_t }^2
              \sum_{N_1+N_2+N_3=N}
              \|\nabla^{N_1}(a_{(k)}^2)\|_{C_{t,x} }\norm{  \na^{N_2} (\psi_{(k_1)}^2) }_{C_tL^{\rho}_x }\norm{  \na^{N_3}(\phi_{(k)}^2) }_{L^{\rho}_x }  \notag  \\
			\lesssim & \,\mu^{-1}  \tau^{1-\frac{1}{\gamma}}
             \sum_{N_1+N_2+N_3=N} \thq^{-14N_1-2} \lambda^{N_2} r_{\parallel}^{\frac{1}{\rho} -1}\lambda^{N_3} r_{\perp}^{\frac{1}{\rho}-1}    \notag   \\
			\lesssim & \, \thq^{-2} \lambda^{N}\mu^{-1}r_{\perp}^{\frac{1}{\rho}-1} r_{\parallel}^{\frac{1}{\rho} -1}\tau^{1-\frac{1}{\gamma}},
\end{align}
	and similarly,
	\begin{equation}
		\begin{aligned}
			\label{lp mag time estimate}
			\norm{\na^N d_{q+1}^{(t)} }_{L^\gamma_tL^\rho_x}
           &\lesssim  \thq^{-2} \lambda^{N}\mu^{-1}r_{\perp}^{\frac{1}{\rho}-1} r_{\parallel}^{\frac{1}{\rho} -1}\tau^{1-\frac{1}{\gamma}},
		\end{aligned}
	\end{equation}
	which yield \eqref{dco rlp}.

The $L^\infty_tH^N_x$-estimates of the velocity and magnetic perturbations in
\eqref{principal h3 est-endpt1} and  \eqref{bprincipal h3 est-endpt1}
then follow from estimates \eqref{uprinlp}-\eqref{dcorlp}
with $\gamma =\infty$ and $\rho=2$.

It remains to prove estimates \eqref{pth2 est-endpt1} and \eqref{bpth2 est-endpt1}
of the temporal derivatives.
Using \eqref{b-beta-ve}, \eqref{larsrp}
and Lemmas \ref{Lem-gk-esti}, \ref{Lem-build-S2} and
\ref{Lem-a-S2}
we get
\begin{align*}
	& \norm{\p_t w_{q+1}^{(p)} }_{L^\9_tH^N_x }  + \norm{\p_t  w_{q+1}^{(c)} }_{L^\9_tH^N_x}+\norm{\p_t  w_{q+1}^{(t)} }_{L^\9_tH^N_x}+\norm{ \p_t \wo }_{L^\9_tH^N_x}\notag \\
	\lesssim &\, \thq^{-14N-14}\lambda^{N+1}\mu \tau^{\frac12} +\thq^{-14N-42}\lambda^{N+1}\mu r_{\perp} r_{\parallel}^{-1}\tau^{\frac12}
	+ \thq^{-14N-16} \lambda^{N+1} r_{\perp}^{-\frac 12} r_{\parallel}^{-\frac12} \tau + \thq^{-14N-16}\tau \notag\\
	\lesssim &\,\thq^{-14N-14}\lambda^{N+4\a- 2+9\ve} +\thq^{-14N-42}\lambda^{N+4\a- 2+5\ve}+ \thq^{-14N-16}
      \lambda^{N+4\a- 2+8\ve}+ \thq^{-14N-16} \lambda^{4\a-4+12\ve} \notag \\
	\lesssim&\, \lambda^{N+5},
\end{align*}
where we also used the fact that $\thq^{-14N-42}\ll \lbb^{3\ve}$
and $0<\ve\leq 1/10$.
This yields \eqref{pth2 est-endpt1}.

Arguing in a similar manner,
we see that the derivative of magnetic perturbations
obey the same upper bound
and so get \eqref{bpth2 est-endpt1}.
Therefore, the proof of Lemma~\ref{Lem-perturb-A1} is complete.	
\end{proof}

\subsection{Verification of inductive estimates for velocity and magnetic fields}  \label{Subsec-induc-vel-mag}

We now verify the inductive estimates \eqref{ubh3}, \eqref{ubpth2}, \eqref{u-B-L2tx-conv}-\eqref{u-B-Lw-conv}
for the velocity and magnetic fields.

First, we derive from \eqref{ubh3},  \eqref{principal h3 est-endpt1}
and \eqref{bprincipal h3 est-endpt1} that, for $0\leq N\leq 4$,
\begin{align}  \label{verifyuc1}
	 \norm{(u_{q+1},B_{q+1})}_{L^\9_t \cH^{N+3}_x}
   \lesssim& \norm{(\wt u_{q},\wt B_{q})}_{L^\9_t \cH^{N+3}_x}+\norm{(w_{q+1},d_{q+1})}_{L^\9_t \cH^{N+3}_x}  \notag \\
   	&\lesssim \mq^{\frac{N}{2\a}}\la^5+ \laq^{N+5} \lesssim \lambda_{q+1}^{N+5},
\end{align}
which verifies \eqref{ubh3} at level $q+1$.

Moreover,
by \eqref{ubpth2}, \eqref{pdvh3}, \eqref{pth2 est-endpt1} and \eqref{bpth2 est-endpt1},
\begin{align}  \label{verifyptuh2}
	\norm{(\p_t u_{q+1},\p_t B_{q+1})}_{L^\9_t \cH^{N}_x}
 \lesssim&  \norm{(\p_t \wt u_{q},\p_t \wt B_{q})}_{L^\9_t \cH^{N}_x}+\norm{(\p_t w_{q+1},\p_t d_{q+1})}_{L^\9_t \cH^{N}_x} \notag \\
 \lesssim&   \thq^{-\frac 32}\lambda_{q}^5+ \lambda_{q+1}^{N+5}\lesssim \lambda_{q+1}^{N+5},
\end{align}
and thus \eqref{ubpth2} is verified at level $q+1$.

In order to obtain the $L^2_{t}\cL^2_x$-decay estimate \eqref{u-B-L2tx-conv},
we apply Lemma~\ref{Decorrelation1} with $f= a_{(k)}$,
$g = \g\psi_{(k_1)}\phi_{(k)}$ and $\theta = \lambda^{2\ve}$
and Lemma \ref{Lem-a-S2} to derive
\begin{align} \label{wpdp-decoup-L2-A1} 	
	\norm{(w^{(p)}_{q+1},d^{(p)}_{q+1})}_{\cL^2_{t,x}}
	&\lesssim \sum\limits_{k\in \Lambda_u \cup \Lambda_B}
       \Big(\|a_{(k)}\|_{L^2_{t,x}}\norm{ \g }_{L^2_{t}} \norm{ \psi_{(k_1)}\phi_{(k)}}_{C_tL^2_{x}} +\lambda^{-\ve}\|a_{(k)}\|_{C^1_{t,x}}\norm{ \g }_{L^2_{t}} \norm{  \psi_{(k_1)}\phi_{(k)}}_{C_tL^2_{x}}\Big) \notag\\
	&\lesssim  \delta_{q+1}^{\frac{1}{2}}+\thq^{-14}\lambda^{-\ve}_{q+1}   \lesssim \delta_{q+1}^{\frac{1}{2}}.
\end{align}
Then, in view of \eqref{larsrp}, \eqref{perturbation}  and Lemma \ref{Lem-perturb-A1},
we get
\begin{align}  \label{e3.41}
	\norm{w_{q+1}}_{L^2_{t}\cL^2_x}
&\lesssim\norm{w_{q+1}^{(p)} }_{L^2_{t}\cL^2_x}
   + \norm{ w_{q+1}^{(c)} }_{L^2_{t}\cL^2_x}
    +\norm{ w_{q+1}^{(t)} }_{L^2_{t}\cL^2_x}+\norm{ \wo }_{L^2_{t}\cL^2_x}\notag \\
	&\lesssim \delta_{q+1}^{\frac{1}{2}} +\thq^{-1}r_{\perp} r_{\parallel}^{-1}
     + \thq^{-2} \mu^{-1} r_{\perp}^{-\frac{1}{2}} r_{\parallel}^{-\frac12} \tau^\frac 12 + \thq^{-16}\sigma^{-1}\lesssim \delta_{q+1}^{\frac{1}{2}}.
\end{align}
Similar upper bound also holds for the magnetic perturbation $d_{q+1}$.
Hence, taking into account \eqref{uuql2} and \eqref{uB-q+1-A1}
we obtain
\begin{align} \label{e3.43.2}
 \norm{(u_{q} - u_{q+1}, B_{q} - B_{q+1})}_{L^2_{t}\cL^2_x}
	  \lesssim&  \norm{(u_q - \wt u_q, B_q - \wt B_q)}_{L_t^\9 \cL_x^2}
                 + \norm{(w_{q+1}, d_{q+1})}_{L^2_{t}\cL^2_x}   \nonumber  \\
	  \lesssim& \la^{-3}+\delta_{q+1}^{\frac{1}{2}}   \leq M^*\delta_{q+1}^{\frac{1}{2}},
\end{align}
where $M^*$ is a universal large constant.
Hence, we prove the $L^2_{t}\cL^2_x$-decay estimate \eqref{u-B-L2tx-conv} at level $q+1$.

Furthermore,
we derive from Lemma \ref{Lem-perturb-A1} and \eqref{larsrp} that,
\begin{align*}
	\norm{(w_{q+1},d_{q+1})}_{L^1_t \cL^2_x}
 \lesssim& \thq^{-1}\tau^{-\frac12}+\thq^{-1} r_{\perp} r_{\parallel}^{-1} \tau^{-\frac12}
 + \thq^{-2} \mu^{-1} r_{\perp}^{-\frac{1}{2}} r_{\parallel}^{-\frac12}
 + \thq^{-16}\sigma^{-1}\lesssim \lambda_{q+1}^{-\ve}.
\end{align*}
This along with \eqref{uuql2} yields that
\begin{align}  \label{uql1l2.2}
	  \norm{(u_{q} - u_{q+1}, B_{q} -B_{q+1})}_{L^1_t \cL^2_x}
	  \lesssim& \norm{( u_q - \wt u_q, B_q - \wt B_q) }_{L_t^\9\cL_x^2}
                 + \norm{(w_{q+1}, d_{q+1}) }_{L^1_t\cL^2_x}  \notag \\
	  \lesssim& \lambda_q^{-3}+\lambda_{q+1}^{-\ve} \leq \delta_{q+2}^{\frac{1}{2}},
\end{align}
where the last step is due to
$\lbb_q^{-3} \ll \delta_{q+2}^{\frac 12}$.
Hence, estimate \eqref{u-B-L1L2-conv} is justified.

At last, regarding estimate \eqref{u-B-Lw-conv},
since $(s,\gamma,p)\in \mathcal{S}_2$,
we use the Sobolev embedding (cf. \cite{lqzz22})
\begin{align}\label{sobolevem}
	H^3_x\hookrightarrow W^{s,p}_x,
\end{align}
together with \eqref{def-mq-thetaq}, \eqref{rhN}, \eqref{est-vih3}
and \eqref{def-wtu},
to deduce
\begin{align}\label{wtu-u}
	\norm{(\wt u_q-u_q, \wt B_q-B_q)}_{L^\gamma_t\cW^{s,p}_x}
    & \lesssim \norm{(\sum_i\chi_i(v_i-u_q), \sum_i\chi_i(D_i-B_q))}_{L^\9_t \cH^3_x} \notag\\
	& \lesssim \sup_i  \norm{(v_i-u_q, D_i - B_q)}_{L^\9(\supp(\chi_i); \cH^3_x)} \notag\\
	&\lesssim \sup_i |t_{i+1}+\theta_{q+1}-t_i| \|(|\nabla|\tr,|\nabla|\trb)\|_{L^\9_t\cH^3_x}  \notag \\
    & \lesssim \mq^{-1}\la^{10}\lesssim \la^{-2}.
\end{align}
Taking into account Lemma \ref{Lem-perturb-A1}
we deduce
\begin{align}\label{lw-est}
	\norm{(u_{q+1} - u_q, B_{q+1} - B_q) }_{L^\gamma_t\cW^{s,p}_x}
	&\lesssim \norm{(\wt u_q-u_q, \wt B_q-B_q)}_{L^\gamma_t \cW^{s,p}_x}
	+\norm{(w_{q+1}, d_{q+1})}_{L^\gamma_t\cW^{s,p}_x}   \notag\\
	&\lesssim \la^{-2} + \thq^{-1}\lbb_{q+1}^{s}\rs^{\frac{1}{p}-\frac 12}\rp^{\frac{1}{p}-\frac12}\tau^{\frac12-\frac{1}{\gamma}}
	+\thq^{-44}\sigma^{-1} \notag\\
	&\lesssim \lambda_{q}^{-2}
	+ \lambda_{q+1}^{s+2\a-1-\frac{2}{p}-\frac{4\a-4}{\gamma} +\ve(2+\frac8p-\frac{12}{\gamma}) }
	+ \lbb_{q+1}^{-\ve}
    \lesssim \delta_{q+1}^\frac 12,
\end{align}
where in the last inequality we also used \eqref{b-beta-ve}, \eqref{larsrp} and
the fact that
\begin{align}\label{endpt1-condition}
	s+2\a-1-\frac{2}{p}-\frac{4\a-4}{\gamma} +\ve(2+\frac8p-\frac{12}{\gamma})
	 <-10\ve.
\end{align}

Therefore, the inductive estimate \eqref{u-B-Lw-conv} is also verified
at level $q+1$.

\subsection{Reynolds and magnetic stresses}   \label{Sec-Rey-mag-stress-A1}

We now treat the new Reynolds and magnetic stresses
at level $q+1$
and prove the corresponding inductive estimates \eqref{rhN} and \eqref{rl1}
in the supercritical regime $\mathcal{S}_2$.

\subsubsection{Decomposition of magnetic and Reynolds stresses}

Using \eqref{equa-mhdr} with $q+1$ replacing $q$,  \eqref{perturbation} and \eqref{uB-q+1-A1}
we derive
\begin{align}
		\displaystyle\div\mathring{R}_{q+1}^B
		&=\underbrace{\partial_t (d^{(p)}_{q+ 1}+d^{(c)}_{q+ 1})+\nu_2(-\Delta)^{\alpha} d_{q+1}  +\div (d_{q + 1} \otimes \wt u_q - \wt u_q \otimes d_{q+1}+ \wt B_q\otimes w_{q+1} -w_{q +1} \otimes \wt B_q )}_{ \div\mathring{R}_{lin}^B  }   \notag\\
		&\quad+\underbrace{\div (d_{q+ 1}^{(p)} \otimes w_{q+1}^{(p)} -w_{q+ 1}^{(p)} \otimes  d_{q+ 1}^{(p)} + \mathring{\wt R^B_{q}})+ \partial_t d_{q+1}^{(t)}+ \partial_t \dqo}_{\div\mathring{R}_{osc}^B }  \notag\\
		&\quad+\div\Big( d_{q+1}^{(p)} \otimes (w_{q+1}^{(c)}+ w_{q+1}^{(t)}+\wo) -(w_{q+1}^{(c)}+w_{q+1}^{(t)}+\wo) \otimes d_{q+1}  \notag  \\
		&\qquad \underbrace{\qquad\quad+(d_{q+1}^{(c)}+d_{q+1}^{(t)}+\dqo)\otimes w_{q+1}-w_{q+1}^{(p)} \otimes (d_{q+1}^{(c)}+d_{q+1}^{(t)}+\dqo)\Big) }_{\div\mathring{R}_{cor}^B }.  \label{rb}
	\end{align}

Using the inverse divergence operator $\mathcal{R}^B$ given by \eqref{calRB-def}
we may choose the magnetic stress
\begin{align}\label{rbcom}
	\mathring{R}_{q+1}^B := \mathring{R}_{lin}^B +   \mathring{R}_{osc}^B+ \mathring{R}_{cor}^B,
\end{align}
where the linear error
\begin{align}
	\mathring{R}_{lin}^B
	:= &  \mathcal{R}^B\(\partial_t (d^{(p)}_{q+1}+d^{(c)}_{q+1})\)
	+ \nu_2 \mathcal{R}^B (-\Delta)^{\a} d_{q+1} \nonumber \\
	&  +\mathcal{R}^B\P_{H}\div\( d_{q + 1} \otimes \wt u_q - \wt u_q \otimes d_{q+1}+ \wt B_q\otimes w_{q+1} -w_{q +1} \otimes \wt B_q\),\label{rbp}
\end{align}
the oscillation error
\begin{align}
	\mathring{R}_{osc}^B &:=  \sum_{k \in \Lambda_B}\mathcal{R}^B\P_{H}\P_{\neq 0}\left (\g^2 \P_{\neq 0}(D_{(k)}\otimes W_{(k)}-W_{(k)}\otimes D_{(k)} )\nabla (a_{(k)}^2)\right)\notag\\
	&\quad -\mu^{-1} \sum_{k \in \Lambda_B}\mathcal{R}^B\P_{H}\P_{\neq 0}(\p_t (a_{(k)}^2\g^2)\psi_{(k_1)}^2\phi_{(k)}^2k_2)\notag \\
	&\quad-\sigma^{-1}\sum_{k\in \Lambda_B}\mathcal{R}^B\P_{H}\P_{\neq 0}\(h_{(\tau)}\aint_{\T^3}D_{(k)}\otimes W_{(k)}-W_{(k)}\otimes D_{(k)} \d x\p_t\nabla(a_{(k)}^{2})\)\notag\\
	&\quad+ \(\sum_{k \neq k' \in \Lambda_B}+ \sum_{k \in \Lambda_u, k' \in \Lambda_B}\)\mathcal{R}^B\P_{H}\div\(a_{(k)}a_{(k')}\g^2(D_{(k')}\otimes W_{(k)}-W_{(k)}\otimes D_{(k')})\), \label{rob}
\end{align}
and the corrector error
\begin{align}    \label{rbp2}
	\mathring{R}_{cor}^B := &\,\mathcal{R}^B\P_{H}\div\bigg( d_{q+1}^{(p)} \otimes (w_{q+1}^{(c)}+ w_{q+1}^{(t)}+\wo) -(w_{q+1}^{(c)}+w_{q+1}^{(t)}+\wo) \otimes d_{q+1}\notag\\
	&\qquad \qquad \qquad + (d_{q+1}^{(c)}+d_{q+1}^{(t)}+\dqo)\otimes w_{q+1} -w_{q+1}^{(p)} \otimes (d_{q+1}^{(c)}+d_{q+1}^{(t)}+\dqo) \bigg).
\end{align}

Furthermore, regarding the Reynolds stress we compute
\begin{align}
		\displaystyle\div\mathring{R}_{q+1}^u
		&\displaystyle = \underbrace{\partial_t (w_{q+1}^{(p)}+w_{q+1}^{(c)}) +\nu_1(-\Delta)^{\alpha} w_{q+1} +\div\big(\wt u_q \otimes w_{q+1} + w_{q+ 1} \otimes \wt u_q - \wt B_q \otimes d_{q+1} - d_{q+1} \otimes \wt B_q\big) }_{ \div\mathring{R}_{lin}^u +\nabla P_{lin} }   \notag\\
		&\displaystyle\quad+ \underbrace{\div (w_{q+1}^{(p)} \otimes w_{q+1}^{(p)} - d_{q+1}^{(p)} \otimes d_{q+1}^{(p)} +  \tr)+\partial_t w_{q+1}^{(t)}+\partial_t \wo}_{\div\mathring{R}_{osc}^u +\nabla P_{osc}}  \notag\\
		&\displaystyle\quad+\div\Big((w_{q+1}^{(c)}+ w_{q+1}^{(t)}+\wo)\otimes w_{q+1}+ w_{q+1}^{(p)} \otimes (w_{q+1}^{(c)}+ w_{q+1}^{(t)}+\wo)  \notag \\
		&\qquad \underbrace{\qquad - (d_{q+1}^{(c)}+ d_{q+1}^{(t)}+\dqo)\otimes d_{q+1}- d_{q+1}^{(p)} \otimes (d_{q+1}^{(c)}+ d_{q+1}^{(t)}+\dqo) \Big)}_{\div\mathring{R}_{cor}^u +\nabla P_{cor}} + \nabla P_{q+1} . \label{ru}
\end{align}

Then, using the inverse divergence operator $\mathcal{R}^u$ given by \eqref{calRu-def}
we choose the Reynolds stress
\begin{align}\label{rucom}
	\mathring{R}_{q+1}^u := \mathring{R}_{lin}^u +   \mathring{R}_{osc}^u+ \mathring{R}_{cor}^u,
\end{align}
where the linear error
\begin{align}
	\mathring{R}_{lin}^u & := \mathcal{R}^u\(\partial_t (w_{q+1}^{(p)} +w_{q+1}^{(c)}  )\)
	+ \nu_1 \mathcal{R}^u (-\Delta)^{\a} w_{q+1} \nonumber \\
	&\quad  + \mathcal{R}^u\P_H \div \(\wt u_q \mathring{\otimes} w_{q+1} + w_{q+ 1}
	\mathring{\otimes} \wt u_q- \wt B_q \mathring{\otimes} d_{q+1} - d_{q+1} \mathring{\otimes} \wt B_q\), \label{rup}
\end{align}
the oscillation error
\begin{align}\label{rou}
	\mathring{R}_{osc}^u :=& \sum_{k \in \Lambda_u} \mathcal{R}^u \P_H\P_{\neq 0}\left(\g^2 \P_{\neq 0}(W_{(k)}\otimes W_{(k)})\nabla (a_{(k)}^2)\right) \notag\\
	&+ \sum_{k \in \Lambda_B} \mathcal{R}^u \P_H\P_{\neq 0}\left(\g^2 \P_{\neq 0}(W_{(k)}\otimes W_{(k)}-D_{(k)}\otimes D_{(k)})\nabla (a_{(k)}^2)\right)\notag\\
	& -\mu^{-1}\sum_{k \in \Lambda_u\cup\Lambda_B}\mathcal{R}^u \P_H \P_{\neq 0}\(\p_t (a_{(k)}^2\g^2)\psi_{(k_1)}^2\phi_{(k)}^2k_1\)\notag\\
	&-\sigma^{-1}\sum_{k\in \Lambda_u}\mathcal{R}^u \P_H \P_{\neq 0}\(h_{(\tau)}\aint_{\T^3}W_{(k)}\otimes W_{(k)}\d x\p_t\nabla(a_{(k)}^{2})\)\notag\\
	&-\sigma^{-1}\sum_{k\in \Lambda_B}\mathcal{R}^u \P_H \P_{\neq 0}\( h_{(\tau)}\aint_{\T^3}W_{(k)}\otimes W_{(k)}-D_{(k)}\otimes D_{(k)}\d x\p_t\nabla(a_{(k)}^{2})\)\notag\\
	&+\sum_{k \neq k' \in \Lambda_u\cup\Lambda_B }\mathcal{R}^u \P_H \div\(a_{(k)}a_{(k')}\g^2W_{(k)}\otimes W_{(k')}\)\notag\\
	&  -\sum_{k \neq k' \in \Lambda_B}\mathcal{R}^u \P_H \div\(a_{(k)}a_{(k')}\g^2D_{(k)}\otimes D_{(k')}\),
\end{align}
and the corrector error
\begin{align}
	\mathring{R}_{cor}^u &
	:= \mathcal{R}^u \P_H \div \bigg( w^{(p)}_{q+1} \mathring{\otimes} (w_{q+1}^{(c)}+w_{q+1}^{(t)}+\wo)
      + (w_{q+1}^{(c)}+w_{q+1}^{(t)}+\wo) \mathring{\otimes} w_{q+1}  \nonumber \\
	&\qquad \qquad  \qquad  - d^{(p)}_{q+1} \mathring{\otimes} (d_{q+1}^{(c)}+d_{q+1}^{(t)}+\dqo)- (d_{q+1}^{(c)}+d_{q+1}^{(t)}+\dqo) \mathring{\otimes} d_{q+1} \bigg). \label{rup2}
\end{align}

As in Remark \ref{Rem-Ru-RB},
the new magnetic and Reynolds stresses at level $q+1$
satisfy equations \eqref{rb} and \eqref{ru}, respectively,
and the algebraic identities \eqref{calRuPHdiv-Ru.2}
and \eqref{calRuPHdiv-Rb.2} hold.

\subsection{Verification of inductive estimates for velocity and magnetic stresses}

Since by \eqref{principal h3 est-endpt1}-\eqref{bpth2 est-endpt1},
the velocity and magnetic perturbations obey the same upper bounds as in
\eqref{principal h3 est.2}-\eqref{bpth2 est.2},
we can argue as in the same fashion as the proof of \eqref{ine-rq1h3}-\eqref{RB-HN-A1}
to derive estimate \eqref{rhN} in the supercritical regime
$\mathcal{S}_2$.
Hence, we focus on the verification of the
$L^1_{t}\cL^1_x$-decay estimate \eqref{rl1} below.

\subsubsection{\bf Verification of $L^1_{t,x}$-decay of magnetic stress}
We choose
\begin{align}\label{defp.2}
p: =\frac{2-8\varepsilon}{2-9\varepsilon}\in (1,2),
\end{align}
where $\ve$ is given by \eqref{e3.1},
such that
\begin{equation}\label{setp.2}
	(1-4\varepsilon)(1-\frac{1}{p})=\frac{\varepsilon}{2},
    \ \ (2-8\ve)(\frac 1p - \frac 12) = 1 -5 \ve,
\end{equation}
and, via \eqref{larsrp},
\begin{align}  \label{rs-rp-p-ve}
\rs^{\frac 1p-1}\rp^{\frac 1p-1} = \lambda^{\varepsilon},
   \quad \rs^{\frac 1p-\frac 12}\rp^{\frac 1p-\frac 12} = \lambda^{-1+5\varepsilon}.
\end{align}

Let us estimate the three parts of the magnetic stress $\mathring{R}^B_{q+1}$
separately below.

\paragraph{\bf (i) Linear error.}

Note that,
by   \eqref{div free magnetic}
and the boundedness of $\mathcal{R}^B\curl$ in $L^p_x$,
\begin{align}
	  \| \mathcal{R}^B\partial_t( d_{q+1}^{(p)}+ d_{q+1}^{(c)})\|_{L_t^1L_x^p}
	 \lesssim& \sum_{k \in \Lambda_B}\| \curl\partial_t (\g a_{(k)} D^c_{(k)})  \|_{L_t^1L_x^p} \nonumber\\
	\lesssim& \sum_{k \in \Lambda_B}\Big(\| \g\|_{L^1_t}\(\|  a_{(k)} \|_{C_{t,x}^2}\| D^c_{(k)} \|_{C_t W_x^{1,p}}
            +\|  a_{(k)} \|_{C_{t,x}^1}\| \p_t D^c_{(k)} \|_{C_t W^{1,p}_x}\) \nonumber \\
	&\qquad\qquad+\| \p_t\g\|_{L_t^1}\| a_{(k)} \|_{C_{t,x}^1}\| D^c_{(k)} \|_{C_t W_x^{1,p}}\Big).
\end{align}
Then, in view of Lemmas \ref{Lem-build-S2} and \ref{Lem-a-S2},
\eqref{e3.1}, \eqref{larsrp} and \eqref{rs-rp-p-ve},
we obtain
\begin{align}  \label{mag time derivative}
	& \| \mathcal{R}^B\partial_t( d_{q+1}^{(p)}+ d_{q+1}^{(c)})\|_{L_t^1L_x^p} \notag\\
    \lesssim& \tau^{-\frac12}(\thq^{-14}\rs^{\frac{1}{p}-\frac12}\rp^{\frac{1}{p}-\frac{1}{2}}\lambda^{-1}
        +\thq^{-7}\rs^{\frac{1}{p}+\frac12}\rp^{\frac{1}{p}-\frac{3}{2}}\mu)
        + \sigma\tau^{\frac12}\thq^{-7}\rs^{\frac{1}{p}-\frac12}\rp^{\frac{1}{p}-\frac{1}{2}}\lambda^{-1}  \notag\\
	\lesssim& \thq^{-14}(\lambda^{-2\a-\ve}+ \lambda^{-2\ve}+ \lambda^{2\a-4+13\ve})
     \lesssim  \thq^{-14}\lambda^{-2\ve}.
\end{align}

For the hyper-resistivity term,
using the interpolation inequality and Lemma \ref{Lem-perturb-A1}
we derive
\begin{align}
	\norm{ \mathcal{R}^B(-\Delta)^{\alpha} d_{q+1}^{(p)} }_{L_t^1L^p_x}
    & \lesssim \norm{ |\na|^{2\a-1} d_{q+1}^{(p)} }_{L_t^1L^p_x}\notag\\
	& \lesssim  \norm{d_{q+1}^{(p)}}_{L_t^1L^p_x} ^{\frac{3-2\a}{2}} \norm{d_{q+1}^{(p)}}_{L_t^1W^{2,p}_x} ^{\frac{2\a-1}{2}}\notag\\
	& \lesssim \thq^{-1}\lbb^{2\alpha-1}\rs^{\frac{1}{p}-\frac12}\rp^{\frac{1}{p}-\frac{1}{2}}\tau^{-\frac12},   \label{e5.18}
\end{align}
and
\begin{align}
	&\norm{ \mathcal{R}^B(-\Delta)^{\alpha} d_{q+1}^{(c)} }_{L_t^1L^p_x}
     \lesssim \thq^{-1}\lbb^{2\alpha-1}\rs^{\frac{1}{p}+\frac12}\rp^{\frac{1}{p}-\frac{3}{2}}\tau^{-\frac12},\label{e5.19}\\
	&\norm{ \mathcal{R}^B(-\Delta)^{\alpha} d_{q+1}^{(t)} }_{L_t^1L^p_x}
      \lesssim \thq^{-2}\lbb^{2\alpha-1}\mu^{-1}\rs^{\frac{1}{p}-1}\rp^{\frac{1}{p}-1},\label{e5.20}\\
	&\norm{ \mathcal{R}^B(-\Delta)^{\alpha} \dqo }_{L_t^1L^p_x}   \lesssim \thq^{-44}\sigma^{-1}.\label{e5.21}
\end{align}
Hence, taking into account \eqref{larsrp}, \eqref{magnetic perturbation} and \eqref{rs-rp-p-ve}
we obtain
\begin{align}  \label{mag viscosity}
	& \norm{ \mathcal{R}^B(-\Delta)^{\alpha} d_{q+1} }_{L_t^1L^p_x}  \notag \\
    \lesssim&  \norm{ \mathcal{R}^B(-\Delta)^{\alpha} d_{q+1}^{(p)} }_{L_t^1L^p_x}+\norm{ \mathcal{R}^B(-\Delta)^{\alpha} d_{q+1}^{(c)} }_{L_t^1L^p_x}
	  +\norm{ \mathcal{R}^B(-\Delta)^{\alpha} d_{q+1}^{(t)} }_{L_t^1L^p_x}+\norm{ \mathcal{R}^B(-\Delta)^{\alpha} \dqo }_{L_t^1L^p_x} \notag \\
	 \lesssim& \thq^{-1}\lbb^{2\alpha-1}\rs^{\frac{1}{p}-\frac12}\rp^{\frac{1}{p}-\frac{1}{2}}\tau^{-\frac12}
	+\thq^{-1}\lbb^{2\alpha-1}\rs^{\frac{1}{p}+\frac12}\rp^{\frac{1}{p}-\frac{3}{2}}\tau^{-\frac12}
	   +\thq^{-2}\lbb^{2\alpha-1}\mu^{-1}\rs^{\frac{1}{p}-1}\rp^{\frac{1}{p}-1}
	+\thq^{-44}\sigma^{-1}\notag\\
	 \lesssim& \thq^{-1}\lambda^{-\ve} .
\end{align}

For the remaining terms in \eqref{rbp},
again Lemma~\ref{Lem-perturb-A1} together with \eqref{larsrp}
and \eqref{rs-rp-p-ve} yield that
\begin{align} \label{magnetic linear estimate1}
	&\norm{ \mathcal{R}^B\P_H\div (d_{q + 1} \otimes \wt u_q - \wt u_q \otimes d_{q+1}+ \wt B_q\otimes w_{q+1} -w_{q +1} \otimes \wt B_q ) }_{L_t^1L^p_x}  \nonumber \\	
    \lesssim\,&\norm{d_{q + 1} \otimes \wt u_q - \wt u_q \otimes d_{q+1}+ \wt B_q\otimes w_{q+1} -w_{q +1} \otimes \wt B_q }_{L_t^1L^p_x}  \nonumber \\
	\lesssim\,& \norm{\wt u_q}_{L^\9_{t}H^3_x} \norm{d_{q+1}}_{L_t^1L^p_x} +\norm{\wt B_q}_{L^\9_{t}H^3_x} \norm{w_{q+1}}_{L_t^1L^p_x}  \nonumber \\
	\lesssim\, &\lambda^5_q (\thq^{-1} \rs^{\frac{1}{p}-\frac12}\rp^{\frac{1}{p}-\frac12} \tau^{-\frac 12}+\thq^{-2}\mu^{-1}\rs^{\frac{1}{p}-1}\rp^{\frac{1}{p}-1}
                  + \thq^{-16} \sigma^{-1} )
    \lesssim \thq^{-17}\lambda^{-2\ve}.
\end{align}

Therefore, combining \eqref{mag time derivative}, \eqref{mag viscosity}, \eqref{magnetic linear estimate1} altogether
and using  \eqref{b-beta-ve} 	
we arrive at
\begin{align}   \label{magnetic linear estimate}
	\norm{\mathring{R}_{lin}^B }_{L_t^1L^p_x}
     & \lesssim \thq^{-14}\lambda^{-2\ve} +\thq^{-1}\lambda^{-\ve}+\thq^{-17}\lambda^{-2\ve}
       \lesssim \thq^{-1}\lambda^{-\ve}.
\end{align}

\paragraph{\bf (ii) Oscillation error.}
In contrast to the previous regime $\mathcal{S}_1$,
because of the presence of the temporal correctors $w^{(o)}_{q+1}$
and $d^{(o)}_{q+1}$,
the magnetic oscillation error consists of four parts:
\begin{align*}
	\mathring{R}_{osc}^B = \mathring{R}_{osc.1}^B +  \mathring{R}_{osc.2}^B+  \mathring{R}_{osc.3}^B+ \mathring{R}_{osc.4}^B,
\end{align*}
where $\mathring{R}_{osc.1}^B$ contains the high-low spatial  oscillations
\begin{align*}
	\mathring{R}_{osc.1}^B
	&:=   \sum_{k \in \Lambda_B}\mathcal{R}^B \P_{H}\P_{\neq 0}\left(\g^2 \P_{\neq 0}(D_{(k)}\otimes W_{(k)}-W_{(k)}\otimes D_{(k)})\nabla (a_{(k)}^2) \right),
\end{align*}
$\mathring{R}_{osc.2}^B$ contains the high temporal oscillation
\begin{align*}
	\mathring{R}_{osc.2}^B
	&:= - \mu^{-1} \sum_{k \in \Lambda_B}\mathcal{R}^B\P_{H}\P_{\neq 0}\(\p_t (a_{(k)}^2\g^2) \psi_{(k_1)}^2\phi_{(k)}^2k_2\),
\end{align*}
 $\mathring{R}_{osc.3}^B$ is of low frequency
\begin{align*}
	\mathring{R}_{osc.3}^B &
      := -\sigma^{-1}\sum_{k\in \Lambda_B}\mathcal{R}^B\P_{H}\P_{\neq 0}
       \(h_{(\tau)}\aint_{\T^3}D_{(k)}\otimes W_{(k)}-W_{(k)}\otimes D_{(k)}\d x\p_t\nabla(a_{(k)}^{2})\),
\end{align*}
and $\mathring{R}_{osc.4}^B$ contains the
the interactions
\begin{align*}
	\mathring{R}_{osc.4}^B &
:=\(\sum_{k \neq k' \in \Lambda_B}+ \sum_{k \in \Lambda_u, k' \in \Lambda_B}\)\mathcal{R}^B\P_{H}\div\(a_{(k)}a_{(k')}\g^2(D_{(k')}\otimes W_{(k)}-W_{(k)}\otimes D_{(k')})\).
\end{align*}
Let us treat the four parts separately in the following.

The decoupling Lemma~\ref{commutator estimate1}
with $a = \nabla (a_{(k)}^2)$, $f =  \psi_{(k_1)}^2\phi_{(k)}^2$
and $k= \lbb \rs/2$
permits to control the high-low oscillations error:
\begin{align}  \label{I1-esti}
	\norm{\mathring{R}_{osc.1}^B }_{L^1_tL^p_x}
	&\lesssim  \sum_{ k \in \Lambda_B}
	\|\g\|_{L^2_t}^2\norm{|\nabla|^{-1} \P_{\not =0}
		\left(\P_{\geq (\lambda \rs/2)}(D_{(k)}\otimes W_{(k)}-W_{(k)}\otimes D_{(k)} )\nabla (a_{(k)}^2)\right)}_{C_tL^p_x} \notag \nonumber  \\
	& \lesssim  \sum_{ k \in \Lambda_B} \lambda^{-1}  \rs^{-1} \norm{ \na^3(a^2_{(k)})}_{C_{t,x}}
        \norm{\psi_{(k_1)}^2\phi_{(k)}^2}_{C_tL^p_x}  \nonumber  \\
	& \lesssim \sum_{ k \in \Lambda_B} \thq^{-23}  \lambda^{-1} \rs^{-1} \norm{ \psi^2_{(k_1)}}_{C_tL^{p}_x} \norm{\phi^2_{(k)} }_{C_tL^{p}_x}  \nonumber  \\
	& \lesssim \thq^{-23}  \lambda^{-1}  \rs^{\frac{1}{p}-2}\rp^{\frac{1}{p}-1},
\end{align}
where we also used Lemmas \ref{Lem-build-S2} and \ref{Lem-a-S2} in
the last two steps.

Regarding the second part $\mathring{R}_{osc.2}^B $,
we note that,
the derivative landing on $g_{(\tau)}$
gives rise to
high temporal oscillations.
The point here is to balance
the high temporal oscillations by
the large parameter $\mu$.
Precisely, we apply Fubini's theorem,
Lemmas \ref{Lem-gk-esti},
\ref{Lem-build-S2} and \ref{Lem-a-S2} to derive
\begin{align}  \label{I2-esti}
	\norm{\mathring{R}_{osc.2}^B }_{L^1_tL_x^p}
	&\lesssim  {\mu}^{-1}  \sum_{k\in\Lambda_B}\norm{|\nabla|^{-1} \P_{H} \P_{\neq 0}\(\p_t (a_{(k)}^2\g^2)\psi_{(k_1)}^2\phi_{(k)}^2k_2\)}_{L^1_tL_x^p} \nonumber  \\
	&\lesssim   {\mu}^{-1}  \sum_{k\in\Lambda_B}
	\( \norm{\p_t (a_{(k)}^2) }_{C_{t,x}}\norm{\g^2 }_{L^1_t}+ \norm{a_{(k)} }_{C_{t,x}}^2\norm{\p_t(\g^2)}_{ L_t^1 } \)
	\norm{\psi_{(k_1)}}_{C_tL^{2p}_x}^2\norm{\phi_{(k)}}_{L^{2p}_x}^2 \nonumber \\
	&\lesssim (\thq^{-9}+\thq^{-2}\tau\sigma)\mu^{-1}\rs^{\frac{1}{p}-1}\rp^{\frac{1}{p}-1}
      \lesssim  \thq^{-9}\tau\sigma\mu^{-1}\rs^{\frac{1}{p}-1}\rp^{\frac{1}{p}-1}.
\end{align}

The low  frequency part $\mathring{R}_{osc.3}^B $
can be estimated easily by \eqref{hk-est} and \eqref{a-mag-S2},
\begin{align}  \label{I3-esti}
	\norm{\mathring{R}_{osc.3}^B }_{L^1_tL^p_x}
    &\lesssim  \sigma^{-1} \sum_{k\in\Lambda_B}
      \|h_{(\tau)}(k_2\otimes k_1-k_1\otimes k_2)\p_t\nabla(a_{(k)}^{2}) \|_{L^1_tL^p_x} \nonumber  \\
	&\lesssim \sigma^{-1} \sum_{k\in\Lambda_B} \|h_{(\tau)}\|_{C_t}
        ( \norm{a_{(k)} }_{C_{t,x}} \norm{a_{(k)} }_{C_{t,x}^2} +\norm{a_{(k)} }_{C_{t,x}^1}^2)\nonumber \\
	&\lesssim \thq^{-15} \sigma^{-1}.
\end{align}

Finally, for the interaction errors $\mathring{R}_{osc.4}^B$,
it can be controlled  by the product estimate in
Lemma~\ref{Lem-build-S2},
due to the small intersections between different spatial
intermittent flows:
\begin{align}\label{I4-esti}
	\norm{\mathring{R}_{osc.4}^B}_{L^1_tL^p_x}
 & \lesssim ( \sum_{ k \neq k' \in \Lambda_B}+\sum_{ k \in \Lambda_u ,  k' \in \Lambda_B })
      \norm{a_{(k)}a_{(k')}\g^2(D_{(k')}\otimes W_{(k)}-W_{(k)}\otimes D_{(k')})}_{L^1_tL^p_x} \notag\\	
	& \lesssim ( \sum_{ k \neq k' \in \Lambda_B}+\sum_{ k \in \Lambda_u ,  k' \in \Lambda_B })
       \norm{a_{(k)} }_{C_{t,x}}\norm{a_{(k')} }_{C_{t,x}}  \norm{\g^2 }_{L^1_t}\|\psi_{(k_1)}\phi_{(k)}\psi_{(k'_1)}\phi_{(k')}\|_{C_tL^p_x}\notag\\
	& \lesssim  \thq^{-2}\rs^{\frac{1}{p}-1}\rp^{\frac{2}{p}-1}  \, .
\end{align}

Therefore, putting estimates \eqref{I1-esti}-\eqref{I4-esti} altogether
and using \eqref{larsrp}, \eqref{rs-rp-p-ve} we arrive at
\begin{align}
	\label{magnetic oscillation estimate}
	\norm{\mathring{R}_{osc}^B}_{L_t^1L^p_x}
    &\lesssim   \thq^{-23}  \lambda^{-1} \rs^{\frac{1}{p}-2}\rp^{\frac{1}{p}-1}
	+\thq^{-9}\tau\sigma\mu^{-1}\rs^{\frac{1}{p}-1}\rp^{\frac{1}{p}-1}+\thq^{-15} \sigma^{-1}+\thq^{-2}\rs^{\frac{1}{p}-1}\rp^{\frac{2}{p}-1} \notag\\
	&\lesssim \thq^{-23}\lbb^{-\va}+\thq^{-9}\lbb^{2\alpha-3+12\ve}+\thq^{-15} \lbb^{-2\va}+ \thq^{-2}\lbb^{-1+9\ve} \notag \\
    & \lesssim \thq^{-23} \lambda^{-\va},
\end{align}
where the last step is due to $\ve <1/10$.

\paragraph{\bf (iii) Corrector error.}
Let $p_1,p_2\in(1,\9)$ be such that
$1/p_1=1-\eta$ and
$1/{p_1}={1}/{p_2}+1/2$
where $\eta\leq \ve/(2(2-8\ve))$.
Then,
using H\"older's inequality,
applying Lemma \ref{Lem-perturb-A1}
and using  \eqref{wpdp-decoup-L2-A1} we get
\begin{align}
	\norm{\mathring{R}_{cor}^B }_{L^1_{t}L^{p_1}_x}
	\lesssim& \norm{w_{q+1}^{(c)}+w_{q+1}^{(t)}+\wo }_{L^2_{t}L^{p_2}_x} (\norm{d^{(p)}_{q+1} }_{L^2_{t,x}} + \norm{d_{q+1} }_{L^2_{t,x}}) \notag \\
    &  +  (\norm{ w_{q+1}^{(p)}}_{L^2_{t,x}} + \norm{ w_{q+1}}_{L^2_{t,x}}) \norm{ d_{q+1}^{(c)}+ d_{q+1}^{(t)}+\dqo }_{L^2_{t}L^{p_2}_x}\notag  \\
     \lesssim&  ( \thq^{-1}\rs^{\frac{1}{p_2}+\frac12}\rp^{\frac{1}{p_2}-\frac32}+ \thq^{-2}\mu^{-1} \rs^{\frac{1}{p_2}-1}\rp^{\frac{1}{p_2}-1} \tau^{\frac 12} +\thq^{-16}\sigma^{-1}) \notag \\
             & \times (\delta_{q+1}^\frac 12 + \thq^{-1} \rs \rp^{-1} + \thq^{-2} \mu^{-1} \rs^{-\frac 12} \rp^{-\frac 12} \tau^{\frac 12}
             + \thq^{-16} \sigma^{-1}) \notag \\
	 \lesssim& \thq^{-32}\delta_{q+1}^\frac 12  (\lambda^{-4\ve+2\eta-8\eta\ve}+\lambda^{-\ve+2\eta-8\eta\ve}+ \lambda^{-2\ve} ) \lesssim \thq^{-32} \lambda^{-\frac{\ve}{2}}.  \label{magnetic corrector estimate}
\end{align}

Now, combining estimates \eqref{magnetic linear estimate},
\eqref{magnetic oscillation estimate} and
\eqref{magnetic corrector estimate} altogether of the
three types of the magnetic stress we get
for the magnetic stress
\begin{align} \label{rq1b.3}
	\|\mathring{R}_{q+1}^B \|_{L^1_{t,x}}
	&\leq \| \mathring{R}_{lin}^B \|_{L_t^1L^p_x} +  \| \mathring{R}_{osc}^B\|_{L_t^1L^p_x}
	+  \|\mathring{R}_{cor}^B \|_{L^1_{t}L^{p_1}_x}    \nonumber  \\
	&\lesssim   \thq^{-1}\laq^{-\varepsilon}+\thq^{-23}\laq^{-\va} + \thq^{-32}\lambda_{q+1}^{-\frac{\ve}{2}}\nonumber  \\
	& \leq  \laq^{-\ve_R}\delta_{q+2},
\end{align}
where we also used \eqref{b-beta-ve} in the last step.
This verifies the $L^1_{t,x}$-decay estimate \eqref{rl1}
for the magnetic stress at level $q+1$.

\subsubsection{\bf Verification of $L^1_{t,x}$-decay of Reynolds stress}
We shall verify the $L^1_{t}\cL^1_x$-decay estimate \eqref{rl1}
of the three parts $\mathring{R}^u_{lin}$, $\mathring{R}^u_{osc}$
and $\mathring{R}^u_{cor}$ for the Reynolds stress $\mathring{R}^u_{q+1}$ at level $q+1$.

\paragraph{\bf (i) Linear error.}
The linear error $\mathring{R}^u_{lin}$
can be estimated in the same fashion as the proof of \eqref{magnetic linear estimate}
and so we also have
\begin{align}   \label{Reynolds linear estimate}
	\norm{ \mathring{R}_{lin}^u }_{L_t^1L^p_x}
   &\lesssim \thq^{-28}\lbb^{-2\ve}+\thq^{-1}\lambda^{-\ve}+ \thq^{-16} \lbb^{-2\ve}
    \lesssim \thq^{-1}\lbb^{-\va}.
\end{align}

\paragraph{\bf (iii) Oscillation error.}
The velocity oscillation error $\mathring{R}_{osc}^u $
consists of the following four parts,
via \eqref{rou},
\begin{align*}
	\mathring{R}_{osc}^u  &= \mathring{R}_{osc.1}^u + \mathring{R}_{osc.2}^u+ \mathring{R}_{osc.3}^u+ \mathring{R}_{osc.4}^u,
\end{align*}
where $\mathring{R}_{osc.1}^u $ is the high-low spatial frequency part
\begin{align}\label{ulhfp}
	\mathring{R}_{osc.1}^u  &:= \sum_{k \in \Lambda_u} \mathcal{R}^u\P_{\neq 0}\(\g^2 \P_{\neq 0}(W_{(k)}\otimes W_{(k)})\nabla (a_{(k)}^2) \) \nonumber \\
	&\quad+  \sum_{k \in \Lambda_B} \mathcal{R}^u \P_{\neq 0}\( \g^2\P_{\neq 0}(W_{(k)}\otimes W_{(k)}-D_{(k)}\otimes D_{(k)})\nabla (a_{(k)}^2)\),
\end{align}
$\mathring{R}_{osc.2}^u $ contains the high temporal oscillations
\begin{align}\label{utop}
	\mathring{R}_{osc.2}^u
	&:= - {\mu}^{-1}  \sum_{k \in \Lambda_u \cup \Lambda_B}
	\mathcal{R}^u \P_{\neq 0}\(\p_t (a_{(k)}^2\g^2)\psi_{(k_1)}^2\phi_{(k)}^2k_1\),
\end{align}
 $\mathring{R}_{osc.3}^u$ is of low frequency
\begin{align}
	\mathring{R}_{osc.3}^u
    :=&- \sigma^{-1}\sum_{k\in \Lambda_u}\mathcal{R}^u \P_{\neq 0}\(h_{(\tau)}\aint_{\T^3}W_{(k)}\otimes W_{(k)}\d x\p_t\nabla(a_{(k)}^{2})\) \notag\\
	&-\sigma^{-1}\sum_{k\in \Lambda_B}\mathcal{R}^u \P_{\neq 0}\(h_{(\tau)}\aint_{\T^3}W_{(k)}\otimes W_{(k)}-D_{(k)}\otimes D_{(k)}\d x\p_t\nabla(a_{(k)}^{2})\), \label{uhtfp}
\end{align}
and $\mathring{R}_{osc.4}^u$ contains the interaction error
\begin{align*}
	\mathring{R}_{osc.4}^u &
	:=\sum_{k \neq k' \in \Lambda_u\cup\Lambda_B }\mathcal{R}^u \P_H \div\(a_{(k)}a_{(k')}\g^2W_{(k)}\otimes W_{(k')}\)\\
	&\quad\  -\sum_{k \neq k' \in \Lambda_B}\mathcal{R}^u \P_H \div\(a_{(k)}a_{(k')}\g^2D_{(k)}\otimes D_{(k')}\).
\end{align*}

Let us treat $\mathring{R}_{osc.1}^u$, $\mathring{R}_{osc.2}^u$,  $\mathring{R}_{osc.3}^u$ and $\mathring{R}_{osc.4}^u $ separately below.
Again we shall apply Lemma \ref{commutator estimate1} to decouple
the high-low interactions to get
\begin{align} \label{J1-esti}
	\norm{\mathring{R}_{osc.1}^u}_{L^1_tL_x^p}
	& \lesssim
	\sum_{k \in \Lambda_u }
       \|\g^2 \|_{L^1_t}
       \ \||\na|^{-1}\P_{\neq 0}( \P_{(\lambda \rs/2)}(W_{(k)}\otimes W_{(k)})\nabla (a_{(k)}^2) ) \|_{L^1_tL_x^p} \nonumber  \\
	&\quad+  \sum_{k \in \Lambda_B}
      \|\g^2 \|_{L^1_t}  \||\na|^{-1}\P_{(\lambda \rs/2)}
          (  \P_{\neq 0}(W_{(k)}\otimes W_{(k)}-D_{(k)}\otimes D_{(k)})\nabla (a_{(k)}^2)) \|_{L^1_tL_x^p} \notag \\
	& \lesssim  \sum_{k \in \Lambda_u \cup \Lambda_B } (\lbb \rs)^{-1} \|\na^3 (a^2_{(k)})\|_{C_{t,x}}
      \|\psi^2_{(k_1)} \phi_{(k)}^2 \|_{C_tL^p} \notag \\
	& \lesssim  \thq^{-44}\lambda^{-1}   \rs^{\frac{1}{p}-2}\rp^{\frac{1}{p}-1},
\end{align}
where we also used Lemmas  \ref{Lem-build-S2} and \ref{Lem-a-S2}
in the last step.

Regarding the high temporal oscillation error $\mathring{R}_{osc.2}^u$,
estimating as in the proof of \eqref{I2-esti} we get
\begin{align}  \label{J2-esti}
	 \norm{\mathring{R}_{osc.2}^u }_{L^1_tL_x^p}
    \lesssim&  {\mu}^{-1}
     \sum_{k \in \Lambda_u \cup \Lambda_B}
     ( \norm{\p_t (a_{(k)}^2) }_{C_{t,x}}\norm{\g^2 }_{L^1_t}+ \norm{a_{(k)} }_{C_{t,x}}^2\norm{\p_t(\g^2)}_{ L_t^1 } )\norm{\psi_{(k_1)}}_{C_tL^{2p}_x}^2\norm{\phi_{(k)}}_{L^{2p}_x}^2 \nonumber \\
	  \lesssim&  \thq^{-16}\tau\sigma\mu^{-1}\rs^{\frac{1}{p}-1}\rp^{\frac{1}{p}-1}.
\end{align}

Moreover, we bound the low frequency part $\mathring{R}_{osc.3}^u $ by
\begin{align}  \label{J3-esti}
	\norm{\mathring{R}_{osc.3}^u}_{L^1_tL^p_x}
     &\lesssim  \sigma^{-1}
     \sum_{k\in\Lambda_u } \||\na|^{-1}\P_{\neq 0} (h_{(\tau)}\aint_{\T^3}W_{(k)}\otimes W_{(k)}\d x\p_t\nabla(a_{(k)}^{2})) \|_{L^1_tL^p_x} \nonumber  \\
	&\quad+ \sigma^{-1}\sum_{k\in\Lambda_B}
        \| |\na|^{-1}\P_{\neq 0} (h_{(\tau)}\aint_{\T^3}W_{(k)}\otimes W_{(k)}-D_{(k)}\otimes D_{(k)}\d x\p_t\nabla(a_{(k)}^{2})) \|_{L^1_tL^p_x} \nonumber  \\
	&\lesssim \sigma^{-1} \sum_{k\in\Lambda_u\cup \Lambda_B}\|h_{(\tau)}\|_{C_t}
        ( \norm{a_{(k)} }_{C_{t,x}} \norm{a_{(k)} }_{C_{t,x}^2} +\norm{a_{(k)} }_{C_{t,x}^1}^2 )         \nonumber \\
	&\lesssim \thq^{-29} \sigma^{-1}.
\end{align}

Finally, for the interaction oscillation $\mathring{R}_{osc.4}^u$,
we apply Lemma \ref{Lem-build-S2} to get
\begin{align}\label{J4-esti}
	\norm{\mathring{R}_{osc.4}^u}_{L^1_tL^p_x}
    &\lesssim \sum_{k \neq k' \in \Lambda_u\cup \Lambda_B}\norm{a_{(k)}a_{(k')} \g^2W_{(k)}\otimes W_{(k')}}_{L^1_tL^p_x}
           +\sum_{k \neq k' \in \Lambda_B}\norm{a_{(k)}a_{(k')}\g^2D_{(k)}\otimes D_{(k')}}_{L^1_tL^p_x}\notag\\
		&\lesssim \sum_{k \neq k' \in \Lambda_u\cup \Lambda_B}
		\norm{a_{(k)}}_{C_{t,x}}\norm{a_{(k')}}_{C_{t,x}}\norm{\g^2}_{L^1_t}\|\psi_{(k_1)}\phi_{(k)}\psi_{(k'_1)}\phi_{(k')}\|_{C_tL^p_x}\notag\\
		&\lesssim  \thq^{-2}\rs^{\frac{1}{p}-1}\rp^{\frac{2}{p}-1}.
\end{align}

Thus, putting estimates \eqref{J1-esti}-\eqref{J4-esti} altogether
we arrive at
\begin{align} \label{Reynolds oscillation estimate}
	\norm{\mathring{R}_{osc}^u }_{L_t^1L^p_x}
    &\lesssim \thq^{-44}\lambda^{-1}  \rs^{\frac{1}{p}-2}\rp^{\frac{1}{p}-1}
      +\thq^{-16}\tau\sigma\mu^{-1}\rs^{\frac{1}{p}-1}\rp^{\frac{1}{p}-1}+ \thq^{-29} \sigma^{-1}+\thq^{-2}\rs^{\frac{1}{p}-1}\rp^{\frac{2}{p}-1} \notag \\
     & \lesssim \thq^{-44}\lbb^{-\va}.
\end{align}

\paragraph{\bf (iii) Corrector error.}
Taking $p_1, p_2\in (1,\9)$ as in the proof of \eqref{magnetic corrector estimate},
applying Lemma~\ref{Lem-perturb-A1} and using \eqref{wpdp-decoup-L2-A1}
we have that, similarly to \eqref{magnetic corrector estimate},
\begin{align}  \label{Reynolds corrector estimate}
	\norm{\mathring{R}_{cor}^u }_{L^1_{t}L^{p_1}_x} &\lesssim \norm{w_{q+1}^{(c)}+w_{q+1}^{(t)}+ w_{q+1}^{(o)}}_{L^2_{t}L^{p_2}_x}
	(\norm{w^{(p)}_{q+1} }_{L^2_{t,x}} + \norm{w_{q+1} }_{L^2_{t,x}})\notag \\
	&\quad +   (\norm{d^{(p)}_{q+1} }_{L^2_{t,x}} + \norm{d_{q+1} }_{L^2_{t,x}}) \norm{ d_{q+1}^{(c)}+ d_{q+1}^{(t)}+d_{q+1}^{(o)} }_{L^2_{t}L^{p_2}_x} \notag \\
	& \lesssim \thq^{-32} \lambda^{-\frac{\ve}{2}}.
\end{align}

Therefore,
combining estimates \eqref{Reynolds linear estimate},
\eqref{Reynolds oscillation estimate} and \eqref{Reynolds corrector estimate}
we obtain the estimates of Reynolds stress
\begin{align} \label{rq1u}
	\norm{\mathring{R}_{q+1}^u }_{L^1_{t,x}}
	&\leq \norm{ \mathring{R}_{lin}^u }_{L_t^1L^p_x} + \norm{ \mathring{R}_{osc}^u}_{L_t^1L^p_x} + \norm{\mathring{R}_{cor}^u }_{L^1_{t}L^{p_1}_x} \nonumber \\
	&\lesssim   \thq^{-1}\laq^{-\varepsilon}+\thq^{-44}\laq^{-\va}+  \thq^{-32}\laq^{-\frac{\ve}{2}} \nonumber  \\
	&  \leq \laq^{-\ve_R}\delta_{q+2},
\end{align}
where $p$ and $p_1$ are as in \eqref{defp} and \eqref{Reynolds corrector estimate}.
Thus, the $L^1_{t,x}$-decay estimate \eqref{rl1} for the
Reynolds stress $\mathring{R}_{q+1}^u$ is verified at level $q+1$.

\section{Proof of main results}   \label{Sec-Proof-Main}

We are now in position to prove the main results,
mainly for the iteration estimates in Theorem \ref{Thm-Iterat}
and the non-uniqueness in Theorem  \ref{Thm-Non-gMHD}.
Theorem \ref{Thm-vanish-viscos} can be proved in a similar fashion
as the proof of \cite[Theorem 1.5]{lzz21},
based on the present convex integration stage in \S \ref{Sec-perturb-S1}-\S \ref{Sec-S2},
and on certain mollification procedure as in \cite{lzz21}
instead of the gluing stage.
We note that, the mollification procedure only involves an extra parameter
$\ell$, which is of negligible order as the role of $\thq$
and hence does not affect the main iteration estimates
as in Theorem \ref{Thm-Iterat}.
Hence, for the simplicity of exposition we omit the details here.
Below we focus on the proof of Theorems \ref{Thm-Iterat} and \ref{Thm-Non-gMHD}.

\subsection{Proof of main iteration in Theorem \ref{Thm-Iterat}} \label{Subsec-Proof-Iter}

The inductive estimates  \eqref{ubh3}-\eqref{u-B-Lw-conv} have been verified in the previous sections.
Hence, it remains to verify the well-preparedness of $(u_{q+1}, B_{q+1}, \rr^u_{q+1}, \rr^B_{q+1})$
and the inductive inclusion \eqref{supprub} for the temporal supports.

For the well-preparedness of $(u_{q+1}, B_{q+1}, \rr^u_{q+1}, \rr^B_{q+1})$,
because the temporal cut-off functions $f_B$ and $f_u$ respect the support
of $(\tr, \trb)$,
in view of \eqref{suppnr},
we deduce
\begin{align*}
	w_{q+1}(t)=d_{q+1}(t)=0 \quad \text{if} \quad \operatorname{dist}(t,I_{q+1}^c)\leq \thq.
\end{align*}
Then, by \eqref{uq1-Bq1-S1} and \eqref{uB-q+1-A1},
$$(u_{q+1}(t),B_{q+1}(t))=(\wt u_q(t),\wt B_q(t))\quad \text{ if } \quad\operatorname{dist}(t,I_{q+1}^c)\leq \thq.$$
Thus, using \eqref{suppnr} again we lead to
\begin{align*}
	(\rr^u_{q+1}(t),\rr^B_{q+1}(t))=(\tr(t),\trb(t))=0 \quad \text{if} \quad \operatorname{dist}(t,I_{q+1}^c)\leq \thq,
\end{align*}
which verifies the well-preparedness of $(u_{q+1}, B_{q+1}, \rr^u_{q+1}, \rr^B_{q+1})$.

The proof of the temporal inductive inclusion \eqref{supprub} is similar to that in \cite{lqzz22}.
For the reader's convenience, we sketch the main arguments below.
By the construction of perturbations and amplitudes in \S \ref{Sec-perturb-S1}
and \S \ref{Sec-S2} and \eqref{suppnr},
\begin{align}
	& \supp_t (w_{q+1},d_{q+1} )\subseteq \bigcup_{k\in \Lambda_u\cup\Lambda_{B} }\supp_t a_{(k)}
      \subseteq N_{2\thq}(\supp_t (\tr,\trb))
      \subseteq N_{2\theta_{q+1}}(I_{q+1}).  \label{e4.43}
\end{align}
But, by the construction of $I_{q+1}$ in \eqref{Iq1-C-def},
one has,
\begin{align}
	& I_{q+1}\subseteq N_{4T/\mq}(\supp_t (\rr^u_{q},\rr^B_q)). \label{ne4.43}
\end{align}
Hence, it follows from \eqref{e4.43} and \eqref{ne4.43} that
\begin{align}
	& \supp_t (w_{q+1},d_{q+1} )\subseteq  N_{6T/\mq}(\supp_t (\rr^u_{q},\rr^B_q)).  \label{e4.45}
\end{align}
Thus,
taking into account \eqref{uq1-Bq1-S1} and \eqref{uB-q+1-A1}
we arrive at
\begin{align} \label{suppbq}
	\supp_t (u_{q+1},B_{q+1})
	&\subseteq \supp_t (\wt u_{q},\wt B_q) \cup \supp_t (w_{q+1},d_{q+1}) \notag\\
	&\subseteq N_{6T/\mq}( \supp_t (u_{q}, B_q, \rr^u_{q},\rr^B_q)) \notag\\
	&\subseteq N_{\delta_{q+2}^{\frac 12}} (\supp_t(u_{q}, B_q, \rr^u_{q},\rr^B_q)),
\end{align}
where the last step is due to
$6T/\mq \ll \delta_{q+2}^{1/2}$.

By the construction of stresses in \S \ref{Sec-stress-S1} and
\S \ref{Sec-S2},
\eqref{e4.43}, \eqref{e4.45} and \eqref{suppwtuq},
we also have
\begin{align} \label{supp-Ru-RB-q+1}
	& \supp_t (\rr^u_{q+1},\rr^B_{q+1})\subseteq \bigcup\limits_{k\in \Lambda_u\cup\Lambda_B }\supp_t a_{(k)}
	\cup \supp_t (\wt u_q,\wt B_q) \subseteq  N_{\delta_{q+2}^{\frac12}} (\supp_t(u_{q}, B_q, \rr^u_{q},\rr^B_q)).
\end{align}

Finally, the inductive inclusion \eqref{supprub}
follows from \eqref{suppbq} and \eqref{supp-Ru-RB-q+1}.
The proof  is complete.
\hfill $\square$ \\

\subsection{Proof of Theorem~\ref{Thm-Non-gMHD}} We prove the statements $(i)$-$(v)$ in Theorem~\ref{Thm-Non-gMHD} below.
\medskip

$(i)$. In the initial step $q=0$,
we take the relaxed solution for equation \eqref{equa-mhdr} as follows:
\begin{align}
	&(u_0,B_0)=(\wt u,\wt B),\label{def-ub0}\\
	&\mathring{R}_0^u=\mathcal{R}^u\(\p_t u_0+\nu_1(-\Delta)^{\alpha} u_0\) + u_0\mathring\otimes u_0-B_0\mathring\otimes B_0, \label{r0u}\\
	& \mathring{R}_0^B=\mathcal{R}^B\(\p_t B_0+\nu_2(-\Delta)^{\alpha} B_0\) + B_0\otimes u_0-u_0\otimes B_0,   \label{r0b}
\end{align}
together with
$ P_0 = -\frac{1}{3} (|u_0|^2 - |B_0|^2)$.

Then, $(u_0, B_0)$ is a well-prepared solution to \eqref{equa-mhdr}
with the set $I_0= [0,T]$
and $\theta_0 = T$.
Moreover, we choose $a$ sufficiently large such that
\eqref{ubh3}-\eqref{rl1} are satisfied at level $q=0$.
Thus,
we apply Theorem~\ref{Thm-Iterat} to obtain
a sequence of relaxed solutions $(u_{q+1},B_{q+1},\rr_{q+1}^u,\rr_{q+1}^B)$ to \eqref{equa-mhdr}
obeying estimates \eqref{ubh3}-\eqref{supprub} for all $q\geq 0$.

Below we show that $\{(u_{q+1}, B_{q+1})\}_q$ is a Cauchy sequence
in $\cH^{\beta'}_x$,  $\beta'\in (0,\frac{\beta}{7+\beta})$,
and the limit solves \eqref{equa-gMHD} in the sense of Definition \ref{weaksolu}.

Actually, by \eqref{ubh3}, \eqref{ubpth2}
and the Sobolev embedding $H^2_x \hookrightarrow L^\infty_x$,
\begin{align}\label{ine-uuqh1}
	\norm{ u_{q+1} - u_q }_{H^1_{t,x}} & \leq \norm{\p_t( u_{q+1} - u_q) }_{L^\9_tH^2_{x}}+ \norm{ u_{q+1} - u_q }_{L^\9_tH^3_{x}} \notag\\
	& \leq \norm{\p_t u_{q+1}}_{L^\9_tH^2_{x}}+\norm{\p_t u_{q}}_{L^\9_tH^2_{x}}+\norm{u_{q+1}}_{L^\9_tH^3_{x}}+ \norm{ u_{q}}_{L^\9_tH^3_{x}}\notag\\
	&\lesssim \laq^{7}+\la^7+\laq^5+\la^5\lesssim \laq^{7},
\end{align}
and similarly,
\begin{align}\label{ine-bbqh1}
	\norm{ B_{q+1} - B_q }_{H^1_{t,x}} & \leq \norm{\p_t( B_{q+1} - B_q) }_{L^\9_tH^2_{x}}+ \norm{ B_{q+1} - B_q }_{L^\9_tH^3_{x}} \lesssim \laq^{7}.
\end{align}
Then, by the interpolation inequality
and \eqref{la}, \eqref{u-B-L2tx-conv}, \eqref{ine-uuqh1} and \eqref{ine-bbqh1}, for any $\beta'\in (0,\frac{\beta}{7+\beta})$,
\begin{align}
	&\,\sum_{q \geq 0} \norm{(u_{q+1} - u_q, B_{q+1} - B_q)}_{\cH^{\beta'}_{t,x}}\notag \\
	\leq  & \, \sum_{q \geq 0} \norm{(u_{q+1} - u_q, B_{q+1} - B_q) }_{\cL^2_{t,x}}^{1- \beta'}
              \norm{(u_{q+1} - u_q, B_{q+1} - B_q )}_{\cH^1_{t,x}}^{\beta'} \notag\\
	\lesssim  &\,  \sum_{q \geq 0} M^{1-\beta'} \delta_{q+1}^{\frac{1-\beta'}{2}}\lambda_{q+1}^{7 \beta' } \lesssim M^{1-\beta'} \delta_{1}^{\frac{1-\beta'}{2}}\lambda_{1}^{7 \beta' } +
	\sum_{ q \geq 1} M^{1-\beta'} \lambda_{q+1}^{-\beta(1 - \beta')  + 7\beta'  } <\9. \label{interpo}
\end{align}

Thus, $\{(u_q,B_q)\}_{q\geq 0}$ is a Cauchy sequence in $\cH^{\beta'}_{t,x}$
and so, there exist $u,B\in \cH^{\beta'}_{t,x}$ such that
\begin{align}   \label{uqBq-uB-0}
   \lim_{q\rightarrow+\infty}(u_q,B_q)=(u,B)\ \ in\ \cH^{\beta'}_{t,x}.
\end{align}
Since by \eqref{rl1}, $\lim_{q \to +\infty} (\mathring{R}_{q}^u, \mathring{R}_{q}^B) = 0 $ in $L^1_{t}\cL^1_x$,
it follows that $(u,B)$ indeed solves \eqref{equa-gMHD} in the sense of Definition~\ref{weaksolu}.

$(ii)$.
By virtue of the inductive decay estimate
\eqref{u-B-Lw-conv}, we have
\begin{align}
	\sum_{q \geq 0}\norm{(u_{q+1} - u_q,  B_{q+1} - B_q)}_{L^\gamma_t \cW^{s,p}_x} < \9.    \label{result-lw}
\end{align}
This yields that $\{(u_q,B_q)\}_{q\geq 0}$
is also a Cauchy sequence in $L^\gamma_t \cW^{s,p}_x$.
Thus, taking into account \eqref{uqBq-uB-0}
we infer that
$$u,B \in H^{\beta^\prime}_{t,x}  \cap  L^\gamma_tW^{s,p}_x,$$
thereby proving the regularity statement $(ii)$.

$(iii).$ Regarding the Hausdorff measure of the singular set,
we set
\[
\mathcal{G}: =   \bigcup_{q \geq 0} I_q^c \setminus \{0,T\}, \ \
\mathcal{B}:=[0,T] \setminus \mathcal{G}.
\]
By construction, $(u_q,B_q)$ is a smooth solution to the MHD equations \eqref{equa-mhdr}
on $\mathcal{G}$,
and  $(u,B) \equiv  (u_q,B_q) $ on $I_{q}^c$ for each $q$.
Thus, the complement set $\mathcal{B}$ contains the singular set of time.
Since by \eqref{Iq1-C-def},
each $I_q$ is covered by at most $m_{q}=\theta_q^{-{\eta}}$ many balls of radius $5\theta_q$,
i.e.,
\begin{align}
   \mathcal{B} \subseteq \bigcup \limits_{i=0}^{m_{q+1}-1}
          [t_i - 2 \theta_{q+1}, t_i + 3 \theta_{q+1}].
\end{align}
Note that, for any $\kappa >\eta$,
\begin{align*}
   \sum\limits_{i=0}^{m_{q+1}-1}
   (5\theta_{q+1})^\kappa
   \lesssim m_{q+1} \theta_{q+1}^\kappa
   \leq \theta_{q+1}^{\kappa-\eta} \to 0, \ \ as\ q\to \infty.
\end{align*}
This yields that the Hausdorff measure $\mathcal{H}^\kappa(\mathcal{B}) =0$
for any $\kappa >\eta$.
Thus,  we deduce
\begin{align*}
	d_{\mathcal{H}} ( \mathcal{B}) \leq \eta<\eta_* .
\end{align*}

$(iv).$ Concerning the temporal support,
let
\begin{align}
   K_q:= \supp_t (u_q, B_q, \mathring{R}^u_q, \mathring{R}^B_q), \ \ q\geq 1.
\end{align}
By \eqref{r0u} and \eqref{r0b},
\begin{align}
	& \supp_t (\mathring{R}^u_0, \mathring{R}^B_0)
      \subseteq K_0 := \supp_t  (u_0,  B_0)
      = \supp_t  (\wt u,\wt B).
\end{align}
Moreover, by the inductive inclusion \eqref{supprub},
\begin{align}
	 K_{q+1} \subseteq N_{\delta_{q+2}^\frac 12} K_{q}
      \subseteq \cdots
      \subseteq N_{\sum\limits_{j=2}^{q+2}\delta_{j}^\frac 12} K_{0}.
\end{align}
Note that, for $a$ large enough,
\begin{align*}
   \sum\limits_{q\geq 0} \delta_{q+2}^\frac 12
   \leq \sum\limits_{q\geq 2} a^{-\beta b^q}
   \leq \sum\limits_{q\geq 2} a^{-\beta bq }
   = \frac{a^{-2\beta b}}{1- a^{-\beta b}}
   \leq \frac 12 \ve_*.
\end{align*}
Thus, it follows that
\begin{align}
	&\supp_t (u, B)
	\subseteq \bigcup_{q\geq 0} K_q
	\subseteq N_{\ve_*} ( \supp_t (\tilde{u}, \tilde{B}) ),
\end{align}
thereby yielding the temporal support statement $(iv)$.

$(v).$ Finally,  the small deviations on average
follow from  \eqref{u-B-L1L2-conv} and \eqref{u-B-Lw-conv}:
\begin{align*}
	\norm{(u - \wt u, B -\wt B)}_{L^1_t\cL^2_x}
	\leq  \sum_{q \geq 0} \norm{ (u_{q+1} - u_q , B_{q+1} - B_q)}_{L^\gamma_t \cL^2_x}
	\leq   2\sum_{q\geq 0} \delta_{q+2} ^{\frac12}
    \leq \frac{2a^{-2\beta b}}{1-a^{-\beta b}}\leq \ve_*,
\end{align*}
and
\begin{align*}
	\norm{(u - \wt u, B - \wt B)}_{L^1_t\cW^{s,p}_x}
	\leq   \sum_{q \geq 0} \norm{ (u_{q+1} - u_q , B_{q+1} - B_q)}_{L^1_t\cW^{s,p}_x}
	\leq  2\sum_{q\geq 0} \delta_{q+2} ^{\frac12} \leq \ve_*.
\end{align*}
Therefore, the proof of Theorem~\ref{Thm-Non-gMHD} is complete.
\hfill $\square$

\appendix


\section{Standard tools in convex integration}   \label{Sec-App-A}

To begin with, let us first recall two geometric lemmas in \cite{bbv20}.

\begin{lemma} ({\bf First Geometric Lemma}, \cite[Lemma 4.2]{bbv20})
	\label{geometric lem 1}
	There exists a set $\Lambda_u \subset \mathbb{S}^2 \cap \mathbb{Q}^3$ that consists of vectors $k$ with associated orthonormal bases $(k, k_1, k_2)$,  $\varepsilon_u > 0$, and smooth positive functions $\gamma_{(k)}: B_{\varepsilon_u}(\Id) \to \mathbb{R}$, where $B_{\varepsilon_u}(\Id)$ is the ball of radius $\varepsilon_u$ centered at the identity in the space of $3 \times 3$ symmetric matrices,  such that for  $S \in B_{\varepsilon_u}(\Id)$ we have the following identity:
	\begin{equation}
		\label{sym}
		S = \sum_{k \in \Lambda_u} \gamma_{(k)}^2(S) k_1 \otimes k_1 .
	\end{equation}
	Furthermore, we may choose $\Lambda_u$ such that $\Lambda_B \cap \Lambda_u = \emptyset$.
\end{lemma}

\begin{lemma} ({\bf Second Geometric Lemma}, \cite[Lemma 4.1]{bbv20})
	\label{geometric lem 2}
	There exists a set $\Lambda_B \subset \mathbb{S}^2 \cap \mathbb{Q}^3$ that consists of vectors $k$ with associated orthonormal bases $(k, k_1, k_2)$,  $\varepsilon_B > 0$, and smooth positive functions $\gamma_{(k)}: B_{\varepsilon_B}(0) \to \mathbb{R}$, where $B_{\varepsilon_B}(0)$ is the ball of radius $\varepsilon_B$ centered at 0 in the space of $3 \times 3$ skew-symmetric matrices, such that for  $A \in B_{\varepsilon_B}(0)$ we have the following identity:
	\begin{equation}
		\label{antisym}
		A = \sum_{k \in \Lambda_B} \gamma_{(k)}^2(A) (k_2 \otimes k_1 - k_1 \otimes k_2) .
	\end{equation}
\end{lemma}

As pointed out in \cite{bbv20},
there exists $N_{\Lambda} \in \mathbb{N}$ such that
\begin{equation} \label{NLambda}
	\{ N_{\Lambda} k,N_{\Lambda}k_1 , N_{\Lambda}k_2 \} \subset N_{\Lambda} \mathbb{S}^2 \cap \mathbb{Z}^3.
\end{equation}
For instance, as in \cite{bbv20}, choosing
$$\Lambda_{u}=\left\{\frac{5}{13} e_{1} \pm \frac{12}{13} e_{2},\frac{12}{13} e_{1} \pm \frac{5}{13} e_{3},\frac{5}{13} e_{2} \pm \frac{12}{13} e_{3}\right\},$$
and
$$\Lambda_{B}=\left\{e_1,e_2,e_3,\frac{3}{5} e_{1}+\frac{4}{5} e_{2},-\frac{4}{5} e_{2}-\frac{3}{5} e_{3}\right\},$$
$N_{\Lambda} = 65$ suffices.

We denote by $M_*$ the geometric constant such that
\begin{align}
	\sum_{k \in \Lambda_{u}} \norm{\gamma_{(k)}}_{C^4(B_{\varepsilon_u}(\Id))}
     + \sum_{k \in \Lambda_{B}} \norm{\gamma_{(k)}}_{{C^4(B_{\varepsilon_B}}(0))} \leq M_*.
	\label{M bound}
\end{align}
This parameter  is universal and will be used later in the estimates of the size of perturbations.

We recall from \cite{bbv20,dls13} the inverse-divergence operator $\mathcal{R}^u$ and $\mathcal{R}^B$,
defined by
\begin{subequations}\label{calR-def}
	\begin{align}
 	& (\mathcal{R}^u v)^{kl} := \partial_k \Delta^{-1} v^l + \partial_l \Delta^{-1} v^k - \frac{1}{2}(\delta_{kl} + \partial_k \partial_l \Delta^{-1})\div \Delta^{-1} v,
       \label{calRu-def}    \\
 	& (\mathcal{R}^Bf)_{ij} :=  \varepsilon_{ijk} (-\Delta)^{-1}(\curl f)_k,\label{calRB-def}
    \end{align}
\end{subequations}
where $v$ is mean-free, i.e., $\int_{\mathbb{T}^3} v dx =0$, $\varepsilon_{ijk}$ is the Levi-Civita tensor, $i,j,k,l \in \{1,2,3\}$.
Note that, the inverse-divergence operator $\mathcal{R}$
maps mean-free functions to symmetric and trace-free matrices, while the operator $\mathcal{R}^B$ returns skew-symmetric matrices.
Moreover, one has the algebraic identities
\begin{align*}
	\div \mathcal{R}^u(v) = v,\ \ \div \mathcal{R}^B(f) = f.
\end{align*}

Both
$|\nabla|\mathcal{R}^u$ and $|\nabla|\mathcal{R}^B$ are Calderon-Zygmund operators
and thus they are bounded in the spaces $L^p$, $1<p<+\infty$.
See \cite{bbv20,dls13} for more details.

\begin{lemma}[\cite{cl21}, Lemma 2.4; see also \cite{bv19b},  Lemma 3.7]   \label{Decorrelation1}
	Let $\theta\in \mathbb{N}$ and $f,g:\mathbb{T}^d\rightarrow \R$ be smooth functions. Then for every $p\in[1,+\9]$,
	\begin{equation}\label{lpdecor}
		\big|\|fg(\theta\cdot)\|_{L^p(\T^d)}-\|f\|_{L^p(\T^d)}\|g\|_{L^p(\T^d)} \big|\lesssim \theta^{-\frac{1}{p}}\|f\|_{C^1(\T^d)}\|g\|_{L^p(\T^d)}.
	\end{equation}
\end{lemma}

\begin{lemma}[\cite{lt20}, Lemma 6; see also \cite{bv19b}, Lemma B.1] \label{commutator estimate1}
	Let $a \in C^{2}\left(\mathbb{T}^{3}\right)$. For all $1<p<+\infty$ we have
	$$
	\left\||\nabla|^{-1} \P_{\neq 0}\left(a \P_{\geq k} f\right)\right\|_{L^{p}\left(\mathbb{T}^{3}\right)} \lesssim k^{-1}\left\|\nabla^{2} a\right\|_{L^{\infty}\left(\mathbb{T}^{3}\right)}\|f\|_{L^{p}\left(\mathbb{T}^{3}\right)},
	$$
	holds for any smooth function $f \in L^{p}\left(\mathbb{T}^{3}\right)$.
\end{lemma}

\begin{lemma}[Semigroup estimates] \label{Lem-semi-est}
	Let $\a\in [1,2)$. For any $1\leq \rho_2\leq \rho_1 \leq \9$ and $n\in \mathbb{N}$, we have
\begin{align}
	& \|e^{-\nu t(-\Delta)^\alpha} v\|_{L^{\rho_1}_x} \leq C
	t^{-\frac{3}{2\alpha}(\frac{1}{\rho_2}-\frac{1}{\rho_1})} \|v\|_{L^{\rho_2}_x}, \label{semigroup-rho1-rho2} \\
	&   \|\na^n e^{-\nu t(-\Delta)^\alpha} v\|_{L^{\rho_1}_x} \leq C
	t^{-\frac{n}{2\alpha}-\frac{3}{2\alpha}(\frac{1}{\rho_2}-\frac{1}{\rho_1})} \|v\|_{L^{\rho_2}_x}.  \label{semigroup-nabla}
\end{align}
	hold for any function $v\in L^{\rho_2}\left(\mathbb{T}^{3}\right)$.
\end{lemma}

\section{Uniqueness of weak solutions}  \label{Sec-App-B}

In this section,
we consider the uniqueness of weak solutions
to \eqref{equa-MHD} in the critical
mixed Lebesgue space $L^\gamma_tW^{s,p}_x$,
where $(s,\gamma,p)$ lies in some suitable part of the
generalized Lady\v{z}enskaja-Prodi-Serrin condition
\eqref{critical-gLPS-mhd}.

Let us define the space $X^{s,\gamma,p}_{(0,t)}$, $0\leq t\leq T$, by
\begin{align}
	& X^{s,\gamma,p}_{(0,t)} =\begin{cases}
		L^\gamma(0,t;W^{s,p}_x),\ \ \text{if}\
		\frac{2\alpha}{\gamma} + \frac 3p = 2\alpha-1+s, 1\leq \gamma <\infty, 1\leq p\leq\infty, s\geq 0,\\
		C([0,t];W^{s,p}_x),\ \,\ \text{if}\ \gamma = \infty, 1\leq p\leq \infty, s\geq 0. 	\end{cases}
\end{align}
The main result is formulated in Theorem \ref{Thm-uniq-glps} below.

\begin{theorem} [Uniqueness of weak solutions]   \label{Thm-uniq-glps}
	Let $(u,B)\in X^{s,\gamma,p}_{(0,T)}\times X^{s,\gamma,p}_{(0,T)}$
	and $(s,\gamma,p)$ satisfy \eqref{critical-gLPS-mhd}
	with $\alpha\geq 1$,
	$s\geq 0$, $2\leq \gamma \leq\9$,
	$1\leq p\leq \infty$
	and $0\leq \frac 1p - \frac s 3 \leq \frac 12$.
	If $(u,B)$ is a weak solution to \eqref{equa-gMHD} with $\nu_1=\nu_2$
	in the sense of Definition~\ref{weaksolu},
	then $(u,B)$ is the unique Leray-Hopf solution.
\end{theorem}

In particular, for the MHD equations \eqref{equa-MHD},
we have the uniqueness of weak solutions in $L^\gamma_tL^{p}_x$
when $(\gamma,p)$ satisfies the
Lady\v{z}enskaja-Prodi-Serrin condition
\begin{align}   \label{critical-LPS}
	 \frac{2}{\gamma} + \frac{3}{p} = 1.
\end{align}

\begin{corollary}   \label{Cor-uniq-LPS}
	Let $(u,B)\in X^{s,\gamma,p}_{(0,T)}\times X^{s,\gamma,p}_{(0,T)}$
	and $(\gamma,p)$ satisfy \eqref{critical-LPS}
	with $2\leq \gamma \leq\9$.
	If $(u,B)$ is a weak solution to \eqref{equa-MHD}
	in the sense of Definition~\ref{weaksolu} with $\alpha=1$ and $\nu_1=\nu_2$,
	then $(u,B)$ is the unique Leray-Hopf solution.
\end{corollary}

As in the context of NSE \cite{cl20.2,lqzz22},
let us first present the following existence
lemma of linearized generalized MHD equations,
which can be proved via the classical Galerkin method.

\begin{lemma}\label{Thm-ext-lin}
	Let $(u,B)\in X^{s,\gamma,p}_{(0,T)}\times X^{s,\gamma,p}_{(0,T)}$ be a weak solution to \eqref{equa-gMHD} with $\nu_1=\nu_2=\nu$
	and $(s,\gamma,p)$ satisfy \eqref{critical-gLPS-mhd}
	with $\alpha\geq 1$,
	$s\geq 0$, $2\leq \gamma \leq\9$,
	$1\leq p\leq \infty$
	and $0\leq \frac 1p - \frac s 3 \leq \frac 12$.
	Given any divergence-free datum
	$(v_0,H_0)\in  L^2(\T^3)\times L^2(\T^3)$,
	there exists weak solutions
	$v,H \in C_{w}([0,T];L^2_x) \cap L^2(0,T;H^\a_x)$
	to the linearized equations
	\begin{equation}   \label{equa-lin-mhd}
		\begin{cases}
			\p_t v+ \nu  (-\Delta)^\a v + u \cdot \nabla v-B\cdot \nabla H + \nabla P = 0, \\
			\p_t H+ \nu  (-\Delta)^\a H + u \cdot \nabla H -B\cdot \nabla v = 0, \\
			\div v =0,\ \div H=0,\\
			v(0)=v_0,\ H(0)=H_0,
		\end{cases}
	\end{equation}
	which satisfies the energy inequality
	\begin{align*}
		\frac{1}{2} \|(v(t),H(t))\|_{\cL^2_x}^2
		+ \nu \int_{t_0}^t \|(\nabla^\a v(s), \nabla^\a H(s))\|_{\cL^2_x}^2 \, \d s
		\leq \frac{1}{2}\|(v(t_0),H(t_0) ) \|_{\cL^2_x}^2 ,
	\end{align*}
	for all $t\in [t_0,T]$, a.e. $t_0\in[0,T]$ (including $t_0=0$).
\end{lemma}

Next, let $(w, D):=(u-v,B-H)$.
We derive from \eqref{equa-lin-mhd} that
\begin{equation}   \label{equa-linear-MHD}
	\begin{cases}
		\p_t w+ \nu  (-\Delta)^\a w + u \cdot \nabla w-B\cdot \nabla D + \nabla P = 0, \\
		\p_t D+ \nu  (-\Delta)^\a D + u \cdot \nabla D -B\cdot \nabla w = 0, \\
		\div w =0,\ \div D=0,\\
		w(0)= 0,\ D(0)= 0.
	\end{cases}
\end{equation}

Moreover, by virtue of linearity,
using $(\wt w, \wt D):=(w+D, w-D)$
we may decouple equations \eqref{equa-linear-MHD}
to get equations of $\wt w$ and $\wt D$
as follows:
\begin{equation}   \label{equa-wtw}
	\begin{cases}
		\p_t \wt w + \nu  (-\Delta)^\a \wt w
		+ (u-B) \cdot \nabla \wt w + \nabla P = 0, \\
		\div \wt w =0, \\
		\wt w(0)= 0,
	\end{cases}
\end{equation}
and
\begin{equation}   \label{equa-wtD}
	\begin{cases}
		\p_t \wt D + \nu  (-\Delta)^\a \wt D
		+ (u+B) \cdot \nabla \wt D + \nabla P = 0, \\
		\div \wt D =0, \\
		\wt D(0)= 0.
	\end{cases}
\end{equation}

Since $u \pm B \in L^\gamma_tW^{s,p}$,
arguing as in the proof of \cite[(B.13)]{lqzz22}
we infer that $\wt w =0$, $\wt D =0$,
and thus
\begin{align}\label{rst-equiv-ubvh}
	u \equiv v,\quad B\equiv H.
\end{align}
This yields that
the weak solutions $u, B$ to \eqref{equa-gMHD} are indeed Leray-Hopf solutions.

Finally, the uniqueness of weak solutions in Theorem~\ref{Thm-uniq-glps} will be a consequence of the following weak-strong uniqueness.
\begin{lemma}[Weak-strong uniqueness]\label{lem-uniq-lps}
	Let $(u,B)$, $(v,H)$ be two Leray-Hopf weak solutions to \eqref{equa-gMHD} with the same initial datum $(u_0,B_0)$. If $u,B\in X^{s,\gamma,p}_{(0,t)}$ and $(s,\gamma,p)$ satisfy \eqref{critical-gLPS-mhd}  with $1\leq \alpha\leq \frac32$,
	$s\geq 0$, $2\leq \gamma \leq\9$, $1\leq p\leq \infty$ and $0\leq \frac 1p - \frac s 3 \leq \frac 12$, then $u\equiv v$ and $B\equiv H$.
\end{lemma}
\begin{proof} We only prove that $u\equiv v,\ B\equiv H$ on $[0,t_*]$ for some $t_*$ sufficiently small, since the general case can be proved by continuation arguments (see e.g., \cite{Tsai18}).

	Let $w:=u-v,\,D:=B-H$, we derive that
	\begin{align*}
		\begin{cases}
			\p_t w +\nu_1  (-\Delta )^\a w + w \cdot \nabla u+((u-w)\cdot\nabla)w -D \cdot \nabla B-((B-D)\cdot\nabla)B + \nabla \wt P = 0 ,\\
			\p_t D +\nu_2  (-\Delta )^\a D + w \cdot \nabla D+((u-w)\cdot\nabla)D-D \cdot \nabla u+((B-D)\cdot\nabla)w= 0 ,\\
			\div w=0,\ \div D=0,\\
			w(0)=0,\ D(0)=0.\\
		\end{cases}
	\end{align*}
Using mollification arguments as in the proof of \cite[Theorem~4.4]{Tsai18} we obtain the energy inequality
	\begin{align}\label{est-ener-w}
		&\quad \frac12 (\|w(t)\|_{L_x^2}^2 +\|D(t)\|_{L_x^2}^2) +  \int_{0}^{t} \nu_1\| \nabla^\a w \|_{L_x^2}^2+\nu_2\| \nabla^\a D \|_{L_x^2}^2  \d s \notag\\
		&\leq \int_{0}^{t} \int_{\T^3} u\cdot (w\cdot\nabla)w -B\cdot (D\cdot\nabla)w+B\cdot (w\cdot\nabla)D-u\cdot (D\cdot\nabla)D\d x\d s.
	\end{align}
	Concerning the nonlinear terms on the right-hand side above, when $2\leq \gamma <\9$,
	 using H\"{o}lder's inequality, Sobolev's embedding and Young's inequality we obtain
	\begin{align} \label{est-ener-right}
		&\quad \left|\int_{0}^{t} \int_{\T^3} B\cdot (w\cdot\nabla)D \d x\d s \right| \notag\\
		&\lesssim \|B\|_{L^\gamma(0,t;L^q_x)} \|w\|_{L^\9(0,t;L^2_x)}^\theta \|w\|_{L^2(0,t;L^{p_1}_x)}^{1-\theta} \|\nabla D\|_{L^2(0,t;L^{p_2}_x)}\notag\\
		&\lesssim \|B\|_{X^{s,\gamma,p}_{(0,t)}} \|w\|_{L^\9(0,t;L^2_x)}^\theta \|\nabla^\a w\|_{L^2(0,t;L^{2}_x)}^{1-\theta} \|\nabla^\a D\|_{L^2(0,t;L^{2}_x)}\notag\\
		&\lesssim \|B\|_{X^{s,\gamma,p}_{(0,t)}} \|w\|_{L^\9(0,t;L^2_x)}^\theta \|(\nabla^\a w,\nabla^\a D)\|_{L^2(0,t;L^{2}_x)}^{2-\theta} \notag\\
		&\lesssim \|(u,B)\|_{X^{s,\gamma,p}_{(0,t)}}( \|w\|_{L^\9(0,t;L^2_x)}^2 +\|D\|_{L^\9(0,t;L^2_x)}^2 +\|\nabla^\a w\|_{L^2(0,t;L^{2}_x)}^2+\|\nabla^\a D\|_{L^2(0,t;L^{2}_x)}^2 ),
	\end{align}
	where $\frac{1}{q}=\frac1p-\frac{s}{3},\ \frac{3}{p_1}=\frac32-\a,\ \frac{3}{p_2}=\frac52-\a, \ \frac1\gamma +\frac{1-\theta}{2}=\frac12$. The other terms on the right-hand side of \eqref{est-ener-w} can be estimated in the same manner. Hence, we have
	\begin{align}\label{est-ener-w-end}
	&\quad 	\|w \|_{L^\9(0,t;L^2_x)}^2 +\|D\|_{L^\9(0,t;L^2_x)}^2 + \nu_1\|\nabla^\a w\|_{L^2(0,t;L^{2}_x)}^2+ \nu_2\|\nabla^\a D\|_{L^2(0,t;L^{2}_x)}^2 \notag\\
		&\leq C\|(u,B)\|_{X^{s,\gamma,p}_{(0,t)}}( \|w\|_{L^\9(0,t;L^2_x)}^2 +\|D\|_{L^\9(0,t;L^2_x)}^2 +\|\nabla^\a w\|_{L^2(0,t;L^{2}_x)}^2+\|\nabla^\a D\|_{L^2(0,t;L^{2}_x)}^2 ),
	\end{align}
	for some universal constant $C$.
	
	Thus, taking $t_*$ sufficiently close to $0$, such that $C\|(u,B)\|_{X^{s,\gamma,p}_{(0,t)}}<1$ for $t\in[0,t_*]$, we obtain $w\equiv D\equiv 0$ on $[0,t_*]$.
	
	In the case where $\gamma=\9$, since $u,B \in C([0,t];W^{s,\frac{3}{2\a-1+s}}_x)$, for any $\ve>0$ we can take $t_*$ small enough and decompose $u = u_1 +u_2$  (see e.g., \cite[P.174]{Tsai18}) such that
	\begin{align} \label{u1-u2-decom}
		\| (u_1,B_1) \|_{X^{s,\9,\frac{3}{2\a-1+s}}_{(0,t)}\times X^{s,\9,\frac{3}{2\a-1+s}}_{(0,t)}} \leq \ve\quad \text{ and }\quad  u_2,B_2 \in X^{s,\9,\9}_{(0,t)},\quad \text{for}\quad t\in[0,t_*]  .
	\end{align}
    Then, H\"{o}lder's inequality,
	Sobolev's embedding and Young's inequality yield
	\begin{align} \label{est-ener-right.2}
		\left|\int_{0}^{t} \int_{\T^3} B\cdot (w\cdot\nabla)D \d x\d s\right| & \lesssim \|B_1\|_{X^{s,\9,\frac{3}{2\a-1+s}}_{(0,t)}} \|(w\cdot\nabla) D\|_{L^1(0,t;L^\frac{3}{4-2\a}_x)}+\|B_2\|_{X^{s,\9,\9}_{(0,t)}} \|(w\cdot\nabla) D\|_{L^1(0,t;L^{1}_x)} \notag\\
		&\lesssim \ve \|w(t)\|_{L^2(0,t;L_x^{\frac{6}{3-2\a}})}\|\nabla D\|_{L^2(0,t;L_x^{\frac{6}{5-2\a}})}+ \|B_2\|_{X^{s,\9,\9}_{(0,t)}} \|w\|_{L^2(0,t;L^{2}_x)} \|\nabla D\|_{L^2(0,t;L^{2}_x)} \notag\\
		&\lesssim \ve \|\nabla^\a w\|_{L^2(0,t;L^{2}_x)}\|\nabla^\a D\|_{L^2(0,t;L^{2}_x)} + \ve\|\nabla^\a D\|_{L^2(0,t;L^{2}_x)}^2 +  \|w\|_{L^2(0,t;L^{2}_x)}^2 \notag\\
		&\leq	C_1\ve (\|\nabla^\a w\|_{L^2(0,t;L^{2}_x)}^2+\|\nabla^\a D\|_{L^2(0,t;L^{2}_x)}^2)+ C_2(\|w\|_{L^2(0,t;L^{2}_x)}^2+\|D\|_{L^2(0,t;L^{2}_x)}^2).
	\end{align}
Similarly, the other terms on the right-hand side of \eqref{est-ener-w} obey the same upper bound. Thus, taking $\ve$ small enough, it follows that
	\begin{align*}
		\|w(t)\|_{ L^2_x}^2+	\|D(t)\|_{L^2_x}^2\leq C\int_0^t \|w(s)\|_{L^2_x}^2+ \|D(s)\|_{L^2_x}^2 \d s, \quad \forall t\in [0,t_*],
	\end{align*}
	which, via Gronwall's inequality, yields that $w\equiv D\equiv 0$ on $[0,t_*]$. The proof of Lemma~\ref{lem-uniq-lps} is complete.
\end{proof}


\noindent{\bf Acknowledgment.}
The authors would like to thank Peng Qu for many helpful discussions on the
non-uniqueness and LPS conditions.
Yachun Li thanks the support by NSFC (No. 11831011, 12161141004).
Deng Zhang  thanks the supports by NSFC (No. 11871337, 12161141004)
and Shanghai Rising-Star Program 21QA1404500.
Yachun Li, Zirong Zeng, and Deng Zhang are also grateful for the supports by
Institute of Modern Analysis--A Shanghai Frontier Research Center.


\begin{thebibliography}{99}
	
	 \bibitem{ABC21}
	D. Albritton, E. Bru\'e, and M. Colombo.
	\newblock Non-uniqueness of Leray solutions of the forced Navier-Stokes equations.
	\newblock {\em Ann. of Math. (2)}, 196(1): 415--455, 2022.
	
	
	\bibitem{A09}
	H. Aluie.
	\newblock Hydrodynamic and magnetohydrodynamic turbulence: Invariants, cascades, and locality.
	\newblock {Ph.D. thesis}, Johns Hopkins University, 2009.
	
	
	\bibitem{bbv20}
	R. Beekie, T. Buckmaster, and V. Vicol.
	\newblock Weak solutions of ideal {MHD} which do not conserve magnetic
	helicity.
	\newblock {\em Ann. PDE}, 6(1):Paper No. 1, 40, 2020.
	
	
	\bibitem{b1993}
	D. Biskamp.
	\newblock {\em Nonlinear Magnetohydrodynamics}.
	\newblock Cambridge Monographs on Plasma Physics. Cambridge University Press,
	1993.
	
	\bibitem{b2003}
	D. Biskamp.
	\newblock {\em Magnetohydrodynamic Turbulence}.
	\newblock Cambridge University Press, 2003.
	
	
	
	\bibitem{B15}
	T. Buckmaster.
	\newblock Onsager's conjecture almost everywhere in time.
	\newblock {\em Comm. Math. Phys.}, 333(3):1175-1198, 2015.
	
	
	\bibitem{bcv21}
	T. Buckmaster, M. Colombo, and V. Vicol.
	\newblock Wild solutions of the Navier–Stokes equations whose singular sets in time have Hausdorff dimension strictly less than 1.
	\newblock {\em J. Eur. Math. Soc.}, doi: 10.4171/JEMS/1162, 2021.
	
	\bibitem{bdis15}
	T. Buckmaster, C. De~Lellis, P. Isett, and L. Sz\'{e}kelyhidi, Jr.
	\newblock Anomalous dissipation for {$1/5$}-{H}\"{o}lder {E}uler flows.
	\newblock {\em Ann. of Math. (2)}, 182(1):127--172, 2015.
	
	
	\bibitem{bdls16}
	T. Buckmaster,  C. De Lellis, and L. Sz\'{e}kelyhidi, Jr.
	\newblock Dissipative Euler flows with Onsager-critical spatial regularity.
	\newblock {\em Comm. Pure Appl. Math.}, 69(9):1613-1670, 2016.
	
	\bibitem{bdsv19}
	T. Buckmaster, C. De~Lellis, L. Sz\'{e}kelyhidi, Jr., and
	V. Vicol.
	\newblock Onsager's conjecture for admissible weak solutions.
	\newblock {\em Comm. Pure Appl. Math.}, 72(2):229--274, 2019.
	
	\bibitem{bsv19}
	T. Buckmaster, S. Shkoller, and V. Vicol.
	\newblock Nonuniqueness of weak solutions to the {SQG} equation.
	\newblock {\em Comm. Pure Appl. Math.}, 72(9):1809--1874, 2019.
	
	\bibitem{bv19b}
	T. Buckmaster and V. Vicol.
	\newblock Nonuniqueness of weak solutions to the {N}avier-{S}tokes equation.
	\newblock {\em Ann. of Math. (2)}, 189(1):101--144, 2019.
	
	
	\bibitem{bv19r}
	T. Buckmaster and V. Vicol.
	\newblock Convex integration and phenomenologies in turbulence.
	\newblock {\em EMS Surv. Math. Sci.}, 6(1-2):173--263, 2019.
	
	\bibitem{bv21}
	T. Buckmaster and V. Vicol.
	\newblock Convex integration constructions in hydrodynamics.
	\newblock {\em Bull. Amer. Math. Soc. (N.S.)}, 58(1):1--44, 2021.
	
	\bibitem{bmnv21}
	T. Buckmaster, N. Masmoudi, M. Novack,  and V, Vicol
	\newblock Non-conservative $H^{\frac12}-$ weak solutions of the incompressible 3D Euler equations
	\newblock  arXiv:2101.09278v2, 2021.
	

	\bibitem{CKS97}
	R.~E. Caflisch, I. Klapper, and G. Steele.
	\newblock Remarks on singularities, dimension and energy dissipation for ideal
	hydrodynamics and {MHD}.
	\newblock {\em Comm. Math. Phys.}, 184(2):443--455, 1997.
	
	\bibitem{cvy21}
	R.~M. Chen, A. Vasseur, and C. Yu.
	\newblock Global ill-posedness for a dense set of initial data to the
	isentropic system of gas dynamics.
	\newblock {\em Advances in Mathematics}, 393:108057, 2021.
	
	
	
	\bibitem{cl21}
	A. Cheskidov and X. Luo.
	\newblock Nonuniqueness of weak solutions for the transport equation at
	critical space regularity.
	\newblock {\em Ann. PDE}, 7(1):Paper No. 2, 45, 2021.
	
	\bibitem{cl20.2}
	A. Cheskidov and X. Luo.
	\newblock Sharp uniqueness for the Navier-Stokes equations.
	\newblock {\em Invent. math.}, https://doi.org/10.1007/s00222-022-01116-x, 2022.
	
	\bibitem{cl21.2}
	A. Cheskidov and X. Luo.
	\newblock $L^2$-critical nonuniqueness for the 2D Navier-Stokes equations.
	\newblock arXiv:2105.12117, 2021.
	
	\bibitem{cl22}
	A. Cheskidov and X. Luo.
	\newblock Extreme temporal intermittency in the linear Sobolev transport: almost smooth nonunique solutions.
	\newblock  arXiv:2204.0895, 2022.
	
	
	\bibitem{CDM20}
	M. Colombo, C. De Lellis, and A. Massaccesi. The generalized Caffarelli-Kohn-Nirenberg theorem for
	the hyperdissipative Navier-Stokes system. Comm. Pure Appl. Math., 73(3):609-663, 2020.
	
	\bibitem{CD12}
	S. Conti, C. De Lellis, and L. Sz\'{e}kelyhidi, Jr.
	h-principle and rigidity for $C^{1,\alpha}$ isometric embeddings.
	In {\em Nonlinear partial differential equations},
	Volume 7 of Abel Symp., 83--116, Springer, Heidelberg, 2012.
	
	\bibitem{dai18}
	M. Dai.
	\newblock Non-uniqueness of Leray-Hopf weak solutions of the 3d Hall-MHD system,
	\newblock {\em SIAM J. Math. Anal.}, 53(5): 5979--6016, 2021.
	
	
	\bibitem{d2001}
	P.~A. Davidson.
	\newblock {\em An Introduction to Magnetohydrodynamics}.
	\newblock Cambridge Texts in Applied Mathematics. Cambridge University Press,
	2001.
	
	\bibitem{dls09}
	C. De~Lellis and L. Sz\'{e}kelyhidi, Jr.
	\newblock The {E}uler equations as a differential inclusion.
	\newblock {\em Ann. of Math. (2)}, 170(3):1417--1436, 2009.
	
	\bibitem{dls10}
	C. De~Lellis and L. Sz\'{e}kelyhidi, Jr.
	\newblock On admissibility criteria for weak solutions of the {E}uler
	equations.
	\newblock {\em Arch. Ration. Mech. Anal.}, 195(1):225--260, 2010.
	
	\bibitem{dls13}
	C. De~Lellis and L. Sz\'{e}kelyhidi, Jr.
	\newblock Dissipative continuous {E}uler flows.
	\newblock {\em Invent. Math.}, 193(2):377--407, 2013.
	
	\bibitem{dls14}
	C. De~Lellis and L. Sz\'{e}kelyhidi, Jr.
	\newblock Dissipative Euler flows and Onsager's conjecture.
	\newblock {\em J. Eur. Math. Soc.}, 16(7):1467--1505, 2014.
	
	\bibitem{dls17}
	C. De~Lellis and L. Sz\'{e}kelyhidi, Jr.
	\newblock High dimensionality and h-principle in {PDE}.
	\newblock {\em Bull. Amer. Math. Soc. (N.S.)}, 54(2):247--282, 2017.
	
	\bibitem{DR19}
	L. De Rosa.
	Infinitely many Leray-Hopf solutions for the fractional Navier-Stokes equations.
	{\em Comm. Partial Differential Equations}, 44(4): 335--365, 2019.
	
	\bibitem{iss03}
	L.~Escauriaza, G.~A. Seregin, and V.~\v{S}ver\'{a}k.
	\newblock {$L_{3,\infty}$}-solutions of {N}avier-{S}tokes equations and
	backward uniqueness.
	\newblock {\em Uspekhi Mat. Nauk}, 58(2(350)):3--44, 2003.
	
	\bibitem{FJR72}
	E.B. Fabes, B.F. Jones and N.M. Rivi\`{e}re,
	\newblock The initial value problem for the Navier-Stokes equations with data in $L^p$.
	\newblock {\em Arch. Rational Mech. Anal.} 45 (1972), 222--240.
	
	\bibitem{FL18}
	D. Faraco and S. Lindberg. Magnetic helicity and subsolutions in ideal MHD. arXiv:1801.04896 (2018).
	
	\bibitem{fl20}
	D. Faraco and S. Lindberg.
	\newblock Proof of {T}aylor's conjecture on magnetic helicity conservation.
	\newblock {\em Comm. Math. Phys.}, 373(2):707--738, 2020.
	
	\bibitem{FL22}
	D. Faraco and S. Lindberg.
	\newblock Rigorous results on conserved and dissipated quantities in ideal MHD turbulence.
	\newblock {\em Geophysical \& Astrophysical Fluid Dynamics},
	DOI: 10.1080/03091929.2022.2060964.
	
	
	\bibitem{fls21}
	D. Faraco, S. Lindberg, and L. Sz\'{e}kelyhidi, Jr.
	\newblock Bounded solutions of ideal {MHD} with compact support in space-time.
	\newblock {\em Arch. Ration. Mech. Anal.}, 239(1):51--93, 2021.
	
	\bibitem{FLMV22}
	D. Faraco, S. Lindberg, D. MacTaggart, A. Valli,
	\newblock On the proof of Taylor's conjecture in multiply connected domains.
	\newblock {\em Appl. Math. Lett.}, 124 (2022), Paper No. 107654, 7 pp
	
	\bibitem{fls21.2}
	D. Faraco, S. Lindberg, and L. Sz\'{e}kelyhidi, Jr.
	\newblock Magnetic helicity, weak solutions and relaxation of ideal MHD.
	\newblock arXiv:2109.09106 (2021).
	
	
	
	\bibitem{hopf1951}
	E. Hopf.
	\newblock \"{U}ber die {A}nfangswertaufgabe f\"{u}r die hydrodynamischen
	{G}rundgleichungen.
	\newblock {\em Math. Nachr.}, 4:213--231, 1951.
	
	\bibitem{I18}
	P. Isett.
	\newblock A proof of {O}nsager's conjecture.
	\newblock {\em Ann. of Math. (2)}, 188(3):871--963, 2018.
	
	
	\bibitem{js14}
	H. Jia and V. \v{S}ver\'{a}k.
	\newblock Local-in-space estimates near initial time for weak solutions of the
	{N}avier-{S}tokes equations and forward self-similar solutions.
	\newblock {\em Invent. Math.}, 196(1):233--265, 2014.
	
	\bibitem{js15}
	H. Jia and V. \v{S}ver\'{a}k.
	\newblock Are the incompressible 3d {N}avier-{S}tokes equations locally
	ill-posed in the natural energy space?
	\newblock {\em J. Funct. Anal.}, 268(12):3734--3766, 2015.
	
	\bibitem{KL07}
	E. Kang and J. Lee.
	\newblock Remarks on the magnetic helicity and energy conservation for ideal
	magneto-hydrodynamics.
	\newblock {\em Nonlinearity}, 20(11):2681--2689, 2007.
	
	
	\bibitem{K00}
	S. Klainerman.
	\newblock PDE as a unified subject.
	\newblock {\em Geom. Funct. Anal.}, Number Special Volume, Part I, pages
	279--315, 2000. GAFA 2000 (Tel Aviv, 1999).
	
	\bibitem{K17}
	S. Klainerman.
	\newblock On Nash's unique contribution to analysis in just three of his papers.
	\newblock {\em Bull. Amer. Math. Soc.} (N.S.), 54(2):283--305, 2017.
	
	\bibitem{leray1934}
	J. Leray.
	\newblock Sur le mouvement d'un liquide visqueux emplissant l'espace.
	\newblock {\em Acta Math.}, 63(1):193--248, 1934.
	
	\bibitem{LQ14}
	T. Li and T. Qin.
	\newblock {\em Physics and partial differential equations.}
	Vol. II. Translated from the 2000 Chinese edition by Yachun Li.
	Society for Industrial and Applied Mathematics (SIAM), Philadelphia, PA; Higher Education Press, Beijing, 2014.
	
	\bibitem{lzz21}
	Y. Li, Z. Zeng, and D. Zhang.
	\newblock Non-uniqueness of weak solutions to 3D magnetohydrodynamic equations.
	\newblock To appear in \newblock {\em J. Math. Pure. Appl.}, arXiv:2112.10515, 2021.
	
	\bibitem{lqzz22}
	Y. Li, P. Qu, Z. Zeng, and D. Zhang.
	\newblock Sharp non-uniqueness for the 3D hyperdissipative Navier-Stokes eqwuations:
beyond the Lions exponent.
	\newblock arXiv:2205.10260, 2022.
	
	\bibitem{lions69}
	J.-L. Lions.
	\newblock {\em Quelques m\'{e}thodes de r\'{e}solution des probl\`emes aux
		limites non lin\'{e}aires}.
	\newblock Dunod; Gauthier-Villars, Paris, 1969.
	
	
	\bibitem{lq20}
	T. Luo and P. Qu.
	\newblock Non-uniqueness of weak solutions to 2{D} hypoviscous
	{N}avier-{S}tokes equations.
	\newblock {\em J. Differential Equations}, 269(4):2896--2919, 2020.
	
	
	\bibitem{lt20}
	T. Luo and E.~S. Titi.
	\newblock Non-uniqueness of weak solutions to hyperviscous {N}avier-{S}tokes
	equations: on sharpness of {J}.-{L}. {L}ions exponent.
	\newblock {\em Calc. Var. Partial Differential Equations}, 59(3):Paper No. 92, 15, 2020.
	
	\bibitem{lxx16}
	T. Luo, C. Xie, and Z. Xin.
	\newblock Non-uniqueness of admissible weak solutions to compressible {E}uler
	systems with source terms.
	\newblock {\em Adv. Math.}, 291:542--583, 2016.
	
	\bibitem{luo19}
	X. Luo.
	\newblock Stationary solutions and nonuniqueness of weak solutions for the
	{N}avier-{S}tokes equations in high dimensions.
	\newblock {\em Arch. Ration. Mech. Anal.}, 233(2):701--747, 2019.
	
	
	\bibitem{ms18}
	S. Modena and L. Sz\'{e}kelyhidi, Jr.
	\newblock Non-uniqueness for the transport equation with {S}obolev vector
	fields.
	\newblock {\em Ann. PDE}, 4(2):Paper No. 18, 38, 2018.
	
	\bibitem{nv22}
	M. Novack and V. Vicol.
	\newblock An Intermittent Onsager Theorem.
	\newblock  arXiv:2203.13115v2, 2022.
	
	\bibitem{rrs16}
	J. C. Robinson, J. L. Rodrigo, and W. Sadowski.
	\newblock {\em The three-dimensional Navier-Stokes equations, Classical theory.}
	\newblock { Cambridge Studies in Advanced Mathematics, 157.} Cambridge University Press, Cambridge, 2016.
	
	\bibitem{ST83}
	M. Sermange and R. Temam.
	Some mathematical questions related to the MHD equations.
	\newblock {\em Comm. Pure Appl. Math.}, 36(5):635-664, 1983.
	
	\bibitem{taylor74}
	J. B. Taylor. Relaxation of toroidal plasma and generation of reverse magnetic fields. \newblock {\em Physical Review Letters}, 33(19): 1139, 1974.
	
	
	\bibitem{taylor86}
	J. B. Taylor. Relaxation and magnetic reconnection in plasmas.
	\newblock {\em Reviews of Modern Physics}, 58(3): 741, 1986.
	
	
	\bibitem{Tsai18}
	T. Tsai.
	\newblock {\em Lectures on Navier-Stokes equations.}
	\newblock Graduate Studies in Mathematics, 192. American Mathematical Society, Providence, RI, 2018.
	
	\bibitem{wu03}
	J. Wu.
	\newblock Generalized {MHD} equations.
	\newblock {\em J. Differential Equations}, 195(2):284--312, 2003.
	
\end{thebibliography}
\end{document}